\documentclass{article}

\usepackage{graphicx}
\usepackage[nonewpage]{imakeidx}
\usepackage{amssymb}
\usepackage{amsthm}
\usepackage[cmtip,all]{xy}
\usepackage{amsmath}
\usepackage{amsfonts}
\usepackage{courier}
\usepackage[noadjust]{cite}
\usepackage{comment}
\usepackage{tikz}
\usepackage{tikz-cd}
\usetikzlibrary{calc,matrix,arrows,decorations.markings}
\usepackage{array}
\usepackage{color}
\usepackage{enumerate}
\usepackage{nicefrac}
\usepackage{listings}
\usepackage{hyperref}
\usepackage{geometry}
\usepackage{float}
\usepackage{multirow}

\newcommand{\Q}{\mathbb{Q}}

\newcommand{\Z}{\mathbb{Z}}

\newcommand{\Aut}{{\rm Aut}}
\newcommand{\Isom}{{\rm Isom}}

\renewcommand{\O}{\mathcal{O}}

\newcommand{\N}{{\rm N}}

\newcommand{\id}{{\rm id}}
\newcommand{\U}{{\mathcal{U}}}
\newcommand{\UU}{{\overline{\mathcal{U}}}}
\newcommand{\B}{{\mathcal{B}}}

\newcommand{\Spec}{{\text{Spec}}}

\newcommand{\V}{\mathcal{V}}
\newcommand{\Un}{\text{Un}}

\newcommand{\Hom}{\text{Hom}}

\newcommand{\PP}{\infty}
\newcommand{\Res}{\text{Res}}
\newcommand{\res}{\text{res}}

\newcommand{\F}{\mathcal{F}}

\newcommand{\Udr}{U^{dr}}

\theoremstyle{definition} 
\newtheorem{definition}[]{Definition}
\newtheorem{example}[definition]{Example}
\newtheorem{algorithm}[definition]{Algorithm}

\theoremstyle{plain}
\newtheorem{lemma}[definition]{Lemma}
\newtheorem{proposition}[definition]{Proposition}

\theoremstyle{remark}
\newtheorem{remark}[definition]{Remark}

\theoremstyle{plain}
\newtheorem{thm}[definition]{Theorem}

\numberwithin{definition}{section}
\numberwithin{equation}{section}

\allowdisplaybreaks

\makeindex[title=Index of notation,columns=4]

\begin{document}

\title{Computation of the Unipotent Albanese Map on Elliptic and Hyperelliptic Curves}
\author{Jamie Beacom \footnote{Mathematical Institute, University of Oxford, Andrew Wiles Building, Radcliffe Observatory Quarter, Woodstock
Road, Oxford, OX2 6GG, UK}}

\maketitle

\begin{abstract}
We study the unipotent Albanese map appearing in the non-abelian Chabauty method of Minhyong Kim. In particular we explore the explicit computation of the $p$-adic de Rham period map $j^{dr}_n$ on elliptic and hyperelliptic curves over number fields via their universal unipotent connections $\U$. 

Several algorithms  forming part of the computation of finite level versions $j^{dr}_n$ of the unipotent Albanese maps are presented. The computation of the logarithmic extension of $\U$ in general requires a description in terms of an open covering, and can be regarded as a simple example of computational descent theory. We also demonstrate a constructive version of a lemma of Hadian used in the computation of the Hodge filtration on $\U$ over affine elliptic and odd hyperelliptic curves. 

We use these algorithms to present some new examples describing the co-ordinates of some of these period maps. This description will be given in terms iterated $p$-adic Coleman integrals. We also consider the computation of the co-ordinates if we replace the rational basepoint with a tangential basepoint, and present some new examples here as well. 

\end{abstract}

\section{Introduction}

This paper will examine some explicit aspects of a method introduced by Minhyong Kim in \cite{kim09}. This method, known as the ``Chabauty-Kim method" or ``non-abelian Chabauty" is an analogue to and extension of an earlier method developed by Chabauty in \cite{chabauty41} and made more explicit by Coleman in \cite{coleman85b}.

Suppose $X$ is a curve over a number field $K$. Let $v$ be a non-archimedean place of $K$, and $K_v$ its completion at $v$. The main restriction in the method of Chabauty-Coleman is the rank-genus condition, which requires our curve to have genus strictly greater than the rank of its Jacobian. Kim in \cite{kim09} provides a possible way around this by replacing the Jacobian with a certain quotient of the de Rham fundamental group $U^{dr}$ of $X$, relative to some fixed basepoint $b \in X(K)$. The quotient in question comes from the Hodge filtration $F^{\bullet}$ on $U^{dr}$ defined in \cite{vol03}. In \cite{kim09} Kim shows that the quotient group $F^0U^{dr} \setminus U^{dr}$ naturally classifies the de Rham path spaces. This property of the quotient group allows us to define a \textit{higher de Rham unipotent Albanese map} $j^{dr}$. Taking quotients of $U^{dr}$ by its lower central series yields a chain of sub-quotients $U^{dr}_n$, each of which also carries a Hodge filtration. Composition of the map $j^{dr}$ with the natural projection $U^{dr} \twoheadrightarrow U^{dr}_n$ give us a family of maps $j_n^{dr} : X(K_v) \rightarrow F^0U^{dr}_n \setminus U^{dr}_n$ which are compatible in the following sense:  

\begin{center}
\begin{tikzcd}
& \vdots \arrow[d] \\
& F^0U^{dr}_n \setminus U^{dr}_n \arrow[d] \\
X(K_v) \arrow[r, "j^{dr}_{n-1}", swap] \arrow[ru, "j^{dr}_n"]& F^0U^{dr}_{n-1} \setminus U^{dr}_{n-1} \\
\end{tikzcd}
\end{center}

The maps $j^{dr}_n$ have a Zariski dense image in the quotient $F^0U^{dr}_n \setminus U^{dr}_n$. It is this property which allows us to carry out a Chabauty type argument. Kim constructs a variety over $\Q_p$, the \textit{Selmer variety}, with suitable finite level versions denoted here by $Sel^{glob}_n$, along with finite level \'{e}tale unipotent Albanese maps $j^{\text{\'{e}t}}_n: X(K) \rightarrow Sel^{glob}_n$. There is also an algebraic map $\log_{v,n}:Sel^{glob}_n \rightarrow F^0U^{dr}_n \setminus U^{dr}_n(K_v)$ such that the the image of $X(K)$ in $F^0U^{dr}_n \setminus U^{dr}_n(K_v)$ will be contained in the image of $Sel^{glob}_n$. The connection to Chabauty then is that if  the following \textit{dimension hypothesis} holds

\begin{equation}\label{Equation: Dimension Hypothesis}
\dim (Sel^{glob}_n) < \dim (F^0U^{dr}_n \setminus U^{dr}_n)
\end{equation}

the algebraicity of $\log_{v,n}$ implies the existence of an algebraic function on $X(K_v)$ such that $X(K)$ is contained in its zero locus. The density of the image of $X(K_v)$ under $j^{dr}_n$ in $F^0U^{dr}_n \setminus U^{dr}_n$ implies that this algebraic function is non-zero on $X(K_v)$, and hence its zero locus is finite by the $p$-adic Weierstrass Preparation Theorem. The finiteness of $X(K)$ then follows. 

One need not only consider $K$-rational points in the above. Let $R$ be the ring of integers of $K_v$ and let $S$ be some finite set of places of $K$ (which is usually taken to contain the places above primes of bad reduction for $X$). Then we can also try to use the machinery above to prove finiteness results for $\mathcal{X}(R_S)$, where $\mathcal{X}$ is some smooth model of $X$ over $R$ and $R_S$ are the $S$-integers of $K$.

There are a growing number of instances where the dimension hypothesis has been verified, yielding Diophantine applications. In a roughly chronological order, on the thrice punctured projective line $\mathbb{P}^1_{\Q}- \lbrace 0,1,\infty \rbrace$ \cite{kim05} it has been verified for the $\Z_S$-points and for the $S$-integers of a totally real field \cite{hadian11}; if $X$ is a once punctured elliptic curve with CM over $\Q$ then $n \geq \#S + r+1$ is sufficient to show the finiteness of $X(\Z_S)$ for some $r$ depending on $E$ \cite{kim10b} and if all the Tamagawa numbers are $1$ and $X$ has rank $1$ then $n=2$ is sufficient \cite{kim10}; if $X$ is a complete hyperbolic curve with CM Jacobian \cite{ck10}; and if $X$ is a complete hyperbolic curve and a solvable cover of $\mathbb{P}^1_{\Q}$ (and hence any smooth superelliptic curve over $\Q$ with genus at least $2$ \cite{eh17}. In \cite{baldog16b},\cite{baldog16a} Balakrishnan and Dogra have made an explicit application of non-abelian Chabauty when $n=2$ - what they refer to as `Quadratic Chabauty' - to $p$-adic heights on elliptic and hyperelliptic curves. They extended their methods in \cite{bd18} to propose an effective Chabauty-Kim theorem which provides bounds of the type produced by Coleman under certain hypothesis even when $r=g$.  Recently with M\"{u}ller, Tuitman and Vonk they demonstrated an application of the Chabauty-Kim method in \cite{bdmtv17} to a non-hyperelliptic curve and used the method to complete the classification of non-CM elliptic curves over $\Q$ with split Cartan level structure.  Also of interest is the fact that a number of important conjectures - Bloch-Kato, Fontaine-Mazur and Jansenn - each imply the existence of an $n \gg 0$ such that the dimension hypothesis is satisfied yielding an effective form of Falting's Theorem. 

The algebraic function whose existence is implied by the dimension hypothesis will be described locally as a $p$-adic analytic power series, and in fact will be defined by $p$-adic iterated integrals. These $p$-adic iterated integrals will come from the parallel transport associated to the $n$-th finite level quotient of the universal unipotent connection $\U$ associated to $X$. 

The universal unipotent connection $\U$ is a universal object among pointed unipotent connections on $X$ and is in fact a pro-unipotent connection with finite level quotients $\U_n$. It transpires that the dual of $\U$ is the co-ordinate ring of the canonical $U^{dr}$ torsor $P^{dr}$. As outlined in \cite{kim09} the universal connection comes with an associated Hodge filtration inducing the Hodge filtration on $U^{dr}$. As shall be shown in Section \ref{Section 2}, in order to determine the co-ordinates of $j^{dr}_n(x)$ it will be necessary to determine this Hodge filtration. Lemma 3.6 in \cite{hadian11} will be used to demonstrate how this Hodge filtration may be computed when $X$ is an affine elliptic or hyperelliptic curve. The practical computation of the Hodge filtration is motivated by the approach of Dogra in \cite{dogra15} and Balakrishnan and Dogra in \cite{baldog16b}.

As part of this computation the logarithmic extension $\UU$ of the universal connection $\U$ on $X$ will need to be computed. In Section \ref{Section 3} a general algorithm for computation of $\UU$ is outlined (Algorithm \ref{Algorithm - Computing GT on UC}), which is an example of computational descent theory. That is,  the logarithmic extension of the universal connection on the complement $X=C-D$ of a complete curve $C$ is computed as a collection of logarithmic connections on an open cover $(U_i)$ of $X$ together with isomorphism satisfying certain descent conditions. This will be necessary because unlike the case of $\mathbb{P}^1 \setminus \lbrace 0,1,\infty \rbrace$, for curves of positive genus the extension $\UU$ will in general have non-trivial bundle, or it may not be possible to express the connection on just one affine piece as having logarithmic poles at all the missing points. This is what necessitates taking the approach outlined above. The existence of suitable logarithmic extensions of the $\U_n$ compatible with projection is proven (Theorem \ref{Theorem - Compute GT on UC of EC}). Conditions are imposed on the extensions, and using these two algorithms are provided, Algorithm \ref{Algorithm - Computing GT on UC of EC} and Algorithm \ref{Algorithm - Computing GT on UC of HEC}, allowing for the iterative computation of $\UU_n$ for elliptic and odd hyperelliptic curves respectively. We then demonstrate how the computation of the Hodge filtration is contained in the computation of these extensions. In Theorem \ref{Theorem - Hodge Filtration on HEC or EC} we present a  constructive version of Lemma 3.6 in \cite{hadian11} when $X$ is an elliptic curve or an odd hyperelliptic curve. From this we develop Algorithm \ref{Algorithm - Compute HF on HEC or EC} for the explicit computation of the Hodge filtration on the $\U_n$ in this case. 

In Section \ref{Section 5} we apply the previous algorithms to the computation of $j^{dr}_n$ for several new $n$. Previously, $j^{dr}_n$ has been determined for elliptic curves only when $n=1,2$ and for hyperelliptic curves it has only been computed for $n=1,2$ in specific cases. We compute the co-ordinates of $j^{dr}_3$ (Proposition \ref{Prop - Level 3 Map EC}) and $j^{dr}_4$ (Proposition \ref{Prop - Level 4 Map EC}) for elliptic curves, and $j^{dr}_2$ (Proposition \ref{Prop - level 2 map hec}) for general odd hyperelliptic curves. In Section \ref{Section 6} we consider the scenario where our basepoint is tangential: this is useful in those cases where a rational basepoint is lacking, and to provide a greater wealth of examples. We provide new explicit descriptions of the coordinates of the maps $j^{dr}_2$ (Proposition \ref{Prop - Level 2 TB Map EC}) and $j^{dr}_3$ (Proposition \ref{Prop - Level 3 TB Map EC}) for elliptic curves with a tangential basepoint at infinity, and $j^{dr}_2$ (Proposition \ref{Prop - level 2 tb map hec}) for odd hyperelliptic curves with a tangential basepoint at infinity.

Although we concentrate on elliptic and odd degree hyperelliptic curves in this paper it should be possible to generalise much of the material to more general classes of curves. Affine elliptic curves and odd degree hyperelliptic curves, having one point at infinity removed from the complete curves, represented a relatively simply yet broad class of examples to consider. It should be simple to translate the results of Section \ref{Section 3} for more general curves. However, the conditions in Theorem \ref{Theorem - Hodge Filtration on HEC or EC} and the proof of said theorem are already very complicated and this is where I would expect to find some difficulty in generalising these results for more general curves.


\textbf{Acknowledgments:}The author would like to thank Minhyong Kim for introducing him to the non-abelian Chabauty method, for the suggestion of this problem, and for his help and encouragement. This work was completed while the author was an EPSRC funded D.Phil. student (Award No. MATH1503) at the Mathematical Institute, University of Oxford. He would also like to thank the referree for many helpful comments.

\section{Background} \label{Section 2}

\subsection{The universal unipotent connection}

Here we introduce unipotent connections and background material on the unipotent Albanese map. Throughout both sections \cite{kim09} is used as a primary reference for definitions and results. 

Let \index{$K$}$K$ be a field of characteristic $0$, let $X$ be a $K$-scheme, and suppose we have a fixed basepoint $b \in X(K)$. Let $\mathcal{V}$ be a vector bundle on $X$. 

\begin{definition}\label{definition - connection} A \textit{connection} on $\mathcal{V}$ is a $K$-linear morphism of sheaves \index{$\nabla$}$\nabla$ such that

\[
\nabla: \mathcal{V} \rightarrow \mathcal{V}\otimes \Omega^1_{X/K}
\]satisfying the Leibniz condition: for $U \subset X$ open, $\nabla(fs) = s \otimes df + f\nabla(s)$ where $f \in \O_X(U)$, $s \in \mathcal{V}(U)$. Here $\Omega^1_{X/K}$ is the sheaf of $1$-forms on $X/K$. 
\end{definition}

\begin{remark}
We will often refer to a vector bundle $\mathcal{V}$ with connection $\nabla$ simply as a connection and write such objects either as $(\mathcal{V}, \nabla)$ or simply as $\mathcal{V}$. 
\end{remark}

\begin{remark}
We may extend $\nabla$ to a covariant derivative $\nabla^1: \mathcal{V}\otimes \Omega^1_{X/K} \rightarrow \mathcal{V}\otimes \Omega^2_{X/K}$ as follows: for $U \subset X$ open define $\nabla^1(s \otimes \omega):= s \otimes d\omega - \nabla(s) \wedge \omega$ for $s \in \mathcal{V}(U),\omega \in \Omega^1_{X/K}(U)$
\end{remark}

\begin{definition}
We say that $\mathcal{V}$ is a \textit{flat} or \textit{integrable} connection if the induced morphism 

\[
\nabla^1 \circ \nabla : \mathcal{V} \rightarrow \mathcal{V} \otimes \Omega_{X/K}^2
\]is the zero map. Note that if $X$ is a curve, then any connection $\mathcal{V}$ is automatically flat. 
\end{definition}

Given a connection \index{$(\mathcal{V}, \nabla)$}$(\mathcal{V}, \nabla)$ with $\mathcal{V}$ of rank $n$, there is a matrix $\Omega \in  \mathfrak{gl}_n \otimes \Omega^1_{X/K} $ called the \textit{connection matrix} which determines $\nabla$: suppose that we have a local basis $e_i: \mathcal{O}_X \hookrightarrow \mathcal{V}$ ($1 \leq i \leq n$). Let $U \subset X$ be some trivialising neighbourhood in $X$. Then $\nabla(e_i) \in \mathcal{V} \otimes \Omega^1_{X/K} (U)$, and so there are $\omega_{ij} \in \Omega^1_{X/K}(U)$ such that 

\[
\nabla(e_i) = \sum_j e_j \otimes \omega_{ij}
\]We let $\Omega := (\omega_{ij})$. We may show that in matrix notation $\nabla(e\cdot f)=e\cdot (df+\Omega\cdot f)$, and so $\nabla$ acts locally as $d+\Omega$.

\begin{remark}
A connection $(\mathcal{V}, \nabla = d+ \Omega)$ is flat if and only if 

\[
d\Omega + \Omega \wedge \Omega =0
\]

\end{remark} 

\begin{definition}
A \textit{morphism of connections} $(\mathcal{V}, \nabla) \rightarrow (\mathcal{W}, \nabla')$ is a morphism $f:\mathcal{V} \rightarrow \mathcal{W}$ of sheaves preserving the connection. 
\end{definition}

\begin{definition}
A connection $\mathcal{V}$ is \textit{unipotent} with \textit{index of unipotency} less than or equal to $n$ if there is a decreasing sequence of sub-connections

\[
\mathcal{V}=\mathcal{V}_n \supset \mathcal{V}_{n-1} \supset ... \supset \mathcal{V}_1 \supset \mathcal{V}_0
\] such that the quotients $\mathcal{V}_{i+1}/\mathcal{V}_i$ are isomorphic to a direct sum of copies of $(\O_X,d)$ i.e. they are trivial. 
\end{definition}

We obtain the following category on $X$:

\begin{definition} \label{Definition - Category of UC n}
Let \index{$\Un_n(X)$}$\Un_n(X)$ be defined to be category whose objects are unipotent vector bundles on $X$ with flat connection having index of unipotency less than or equal to $n$ with morphisms being morphisms of connections. Define \index{$\Un(X)$}$\Un(X)$ to be $\cup_{n\geq 1} \Un_n(X)$
\end{definition} 

Given some $b \in X(K)$ we can define a functor $e_b: \Un(X) \rightarrow \text{Vect}_K$ from $\Un(X)$ to the category of vector spaces over $K$, sending $\mathcal{V} \mapsto \mathcal{V}_b:=b^*\mathcal{V}$. We may show that \index{$e_b$}$e_b$ is a \textit{fibre functor} $(\Un(X),e_b)$ is a neutral Tannakian category. In \cite{andiovkim} the authors make the following definition: 

\begin{definition}
Given a neutral Tannakian category $(\mathcal{C}, \omega)$ over a field $k$, define the \textit{pointed category} \index{$\mathcal{C}^*$}$\mathcal{C}^*$ to be the category whose objects are pairs $(V,v)$, where $V$ is an object of $\mathcal{C}$ and $v \in \omega(V)$ and morphisms $f: (V,v) \rightarrow (W,w)$ being morphisms $f:V \rightarrow W$ in $\mathcal{C}$ such that $\omega(f)(v)=w$.
\end{definition} 

With this in mind we may define the universal connection as a universal projective system of connections:

\begin{definition}\label{Definition - Universal Projective System} 
A projective system of objects \index{$\lbrace (\U_n, u_n) \rbrace_{n \geq 0}$}$\lbrace (\U_n, u_n) \rbrace_{n \geq 0}$ in $\Un(X)^*$ with $\U_n$ having index of unipotency $\leq n$ for all $n \geq 0$ is called \textit{universal} if for every $(\mathcal{V},v)$ in $\Un(X)$ with index of unipotency $\leq n$ there is a unique morphism in $\Un(X)^*$ $\phi$ such that

\[
\phi: (\U_n, u_n) \rightarrow (\mathcal{V},v).
\]That is, there is a morphism $\phi:\U_n \rightarrow \mathcal{V}$ of connections such that 

\[
\phi_b: u_n (\in b^*\U_n) \mapsto v (\in b^*\mathcal{V}).
\]

\end{definition}

It is shown in \cite[I, Chapter 2]{wildeshaus} that such a universal projective system $\lbrace (\U_n, u_n) \rbrace_{n \geq 0}$ exists in $\Un(X)^*$. There it is the object referred to as the ``generic pro-sheaf", $\mathcal{G}_{b,dr}$. Another construction is contained in Theorem A in \cite{wojt93}. A nice reference for the construction on a general scheme $X$ over a field $K$ is contained in \cite[\S 1]{kim09} where we see that we can take $(\U_0,u_0)= (\O_X,1)$ and $u_n$ to be the point $1 \in b^*\U_n$ for all $n$.

\begin{definition}
Let $\lbrace \U_n, 1 \rbrace_{n \geq 0}$ be a universal projective system in $\Un(X)^*$. Then we call this projective system the universal unipotent connection \index{$(\U,u)$}$(\U,u)$ on $X$. 
\end{definition}

\begin{remark}
By abuse of notation we will often also refer to the projective limit $\lim_{\leftarrow} \U_n$ as the universal connection \index{$\U$}$\U$. This is a pro-unipotent connection but as it does not have a finite index of unipotency it is not an object of $\Un(X)$. However, it will be useful to consider this pro-object in the next section where it's relationship to the de Rham fundamental group will be explored. 
\end{remark}

\begin{remark}
We will throughout this paper refer to the universal unipotent connection on $X$ simply as the universal connection on $X$. 
\end{remark}

When $X$ is an affine curve we have an explicit description of the universal unipotent pointed connection:

\begin{definition} \label{Definition - UC on General Affine Curve}
Let $C$ be a smooth projective curve of genus $g$ over a field $K$ of characteristic $0$. Let $D$ be a non-empty divisor of size $r$ and let $X:=C-D$. Let $\alpha_0$,...,$\alpha_{2g+r-2}$ be $1$-forms on $X$ such that their cohomology classes are a $K$-basis of $H^1_{dr}(X/K)$. We will assume that this basis is chosen so that the cohomology classes of $\alpha_0,...,\alpha_{g-1}$ form a $K$-basis of $H^0(C, \Omega_{X/K}^1)$. Let $V_{dr}:= H^1_{dr}(X/K)^{\vee}$ with basis elements \index{$A_i$}$A_i$ dual to \index{$\alpha_i$}$\alpha_i$. Let \index{$R$}$R$ be the tensor algebra of \index{$V_{dr}$}$V_{dr}$ i.e.

\[
R:= \bigoplus_{k \geq 0} V_{dr}^{\otimes k}.
\]

Write the basis element $A_{i_1}\otimes A_{i_2}\otimes ...\otimes A_{i_k}$ as the word $A_{i_1}A_{i_2}...A_{i_k}$. Let $I$ be the two-sided ideal generated by $A_0,...,A_{2g+r-2}$ and define \index{$R_n$}$R_n$ to be the quotient

\[
R_n:=R/I^{n+1}
\]
of $R$ by words of length $\geq n+1$. Then define \index{$\U_n$}

\[
\U_n:= R_n \otimes \O_X
\]
and let $\nabla_n$ be the connection such that 

\[
f \in R_n \mapsto -\sum_i A_i f \otimes\alpha_i.
\]
For convenience we will often write $\nabla$ instead of $\nabla_n$.
\end{definition} 
\begin{thm}[\cite{kim09}, Lemma 3] \label{Theorem - Kim - UC on General Affine Curve}
Let $X$, $\U_n$ be as in Definition \ref{Definition - UC on General Affine Curve}. For every $(\mathcal{V},v)$ a pointed connection with index of unipotency $\leq n$ there is a unique map $(\U_n, 1) \rightarrow (\mathcal{V},v)$. 
\end{thm}

We will need to consider filtrations on connections throughout this paper, and so we make the following definition:

\begin{definition}
By a \textit{filtered connection} $\mathcal{V}:= (\mathcal{V}, \nabla, F^{\bullet})$ on $X$ we mean a vector bundle $\mathcal{V}$ with a connection $\nabla$ which is equipped with a decreasing filtration \index{$F^{\bullet}$}by sub-bundles

\[
\mathcal{V}=F^m \mathcal{V} \subset F^{m+1} \mathcal{V} \subset... \subset F^n \mathcal{V}=0
\] for some $m<n \in \Z$ satisfying the \textit{Griffiths' transversality} property:

\[
\nabla (F^i \mathcal{V}) \subset F^{i-1}\mathcal{V} \otimes \Omega^1_{X/K}
\] for all $i$.  
\end{definition}

\subsection{The de Rham unipotent Albanese map}
From now on, we will assume the following: let $C$ be a smooth curve over a number field $K$ and $D$ a non-empty divisor defined over $K$. Then define $X:=C-D$ and let $b \in X(K)$ be a rational basepoint. Let \index{$v$}$v$ be a non-Archimedean valuation on $K$ lying above a rational prime $p$ of good reduction for $X$ and take \index{$K_v$}$K_v$ to be the completion of $K$ with respect to $v$. Let $R_v$ be the ring of integers of $K_v$ and \index{$k$}$k$ its residue field. Finally, let $X_v:=X \otimes K_v$ denote the basechange of $X$.

Recall from Definition \ref{Definition - Category of UC n} that $\Un(X_v)$ is defined to be the category of unipotent connections on $X_v$ with finite unipotency index. Note we have dropped the requirement that the connections are flat here since $X_v$ is a curve and, therefore, $\Omega^2_{X_v/K_v}=0$. Given $b \in X(K)$ then the functor 
\[
e_b: \Un(X_v) \mapsto \text{Vect}_{K_v}; \mathcal{V}\mapsto b^*\mathcal{V}
\]
is a fibre functor and $(\Un(X_v),e_b)$ is a neutral Tannakian category. We denote by $\langle \Un_n(X_v)\rangle$ the Tannakian sub-category of $\Un(X_v)$ generated by $\Un_n(X_v)$ and we let $\langle e^n_b \rangle$ denote the restriction of $e_b$ to this sub-category. Using Tannaka duality we make the following definition:

\begin{definition}
The \emph{de Rham fundamental group} \index{$U^{dr}$}$U^{dr}$ of $X_v$ with basepoint $b$ is that group scheme associated to the Tannakian category $(\Un(X_v),e_b)$ representing $\Aut^{\otimes}(e_b)$. There is also a group scheme \index{$U^{dr}_n$}$U^{dr}_n$ associated to the Tannakian category $(\langle\Un_n(X_v)\rangle,\langle e^n_b\rangle)$ and representing $\Aut^{\otimes}(\langle e^n_b\rangle)$. For $x \in X(K_v)$ the \emph{de Rham path torsor} \index{$P^{dr}(x)$}$P^{dr}(x)$ is the right $U^{dr}$-torsor representing $\Isom^{\otimes}(e_b,e_x)$ and \index{$P^{dr}_n(x)$}$P^{dr}_n(x)$ similarly is the right $U^{dr}_n$-torsor representing $\Isom^{\otimes}(\langle e^n_b\rangle,\langle e^n_x\rangle)$.  
\end{definition}

\begin{remark}
The group scheme $U^{dr}_n$ is a quotient of $U^{dr}$ in the following sense: let $G$ be a group scheme and define $Z^1G:=G$ and for $n \geq 1$ let $Z^{n+1}G:=[G,Z^nG]$. Then we have $U^{dr}_n = U^{dr}/Z^nU^{dr}$. 
\end{remark}

\begin{remark}
The $P^{dr}(x)$ fit together to form the \emph{canonical torsor} $P^{dr} \rightarrow X$ which is a right torsor for $X \times_{K_v} U^{dr}$ with fibre over $x\in X(K_v)$ being $P^{dr}(x)$. Similarly there is a canonical torsor $P^{dr}_n$ for $U^{dr}_n$. 
\end{remark}

We now elucidate the relationship between the universal connections of the previous section and the de Rham fundamental group through the following lemmas and propositions. We omit the proofs, but these can be found in \cite{kim09}. 

\begin{lemma} \label{Lemma - Hom and fibres}
There are functorial isomorphisms
\[ x^*\U_n \cong \Hom(e^n_b,e^n_x); \;\;x^*\U \cong \Hom(e_b,e_x). 
\]
\end{lemma}

\begin{lemma}
Let $\U_n$ be as in Definition \ref{Definition - UC on General Affine Curve}. Then there is a unique morphism of connections \index{$\Delta$}$\Delta: \U_{n+m} \mapsto \U_n \otimes \U_m$ such that $\Delta(1)=1 \otimes 1$ and $\Delta(A_i)=A_i \otimes 1 + 1 \otimes A_i$. 
\end{lemma}

\begin{remark}
This morphism of connections $\Delta$ extends to $\U$ by taking the limits over all $m$ and $n$. This in turn makes $\U$ into a sheaf of co-commutative co-algebras. Recall from the previous lemma that there is a functorial isomorphism $\lim_{\leftarrow}R_n=x^*\U \cong \Hom(e_b,e_x)$. Following \cite[Section 1]{kim09} note that $b^*\U$ is the universal enveloping algebra of $\text{Lie}(U^{dr})$. This then is a co-commutative Hopf algebra and the co-product will be that induced by the map $\Delta$ of the previous lemma. 
\end{remark}

\begin{definition}
The \emph{group-like} elements of the Hopf algebra $b^*\U$ are the $g \in b^*\U$ such that $\Delta(g)=g \otimes g$. The \emph{primitive} elements are the $h \in b^*\U$ such that $\Delta(h)=h \otimes 1 + 1 \otimes h$.
\end{definition}

Using the functoriality of the isomorphism $x^*\U \cong \Hom(e_b,e_x)$ together with it's explicit description as in \cite{kim09} we deduce that $f \in x^*\U$ is group-like if and only if it belongs to $U^{dr}$. From this we may deduce the following result. 

\begin{proposition} \label{Proposition - Coordinate ring of Pdr}
The coordinate ring $\mathcal{P}^{dr}$ of \index{$P^{dr}$}$P^{dr}$ is the dual sheaf \index{$\mathcal{P}$}$\mathcal{P}=\U^{\vee}$. 
\end{proposition}

\begin{remark} \label{Remark - Exp and Log and Grouplike}
Following this, $U^{dr}$ is then identified with the group-like elements of $b^*\U$ and $\text{Lie}(U^{dr})$ is identified with the primitive elements. The exponential map
\[
\exp: g \mapsto \sum_{n=0}^{\infty} \frac{g^n}{n!}
\]
converges on the image of $\text{Lie}(U^{dr})$ in each $R_n=b^*\U_n$. It is simple to see that when $g$ is primitive and $\exp(g)$ converges then $\exp(g)$ is group-like. Similarly, there is a logarithm map 
\[
\log: g \mapsto \sum_{n=1}^{\infty} \frac{(1-g)^n}{n}
\]
which, when it converges, is inverse to the exponential map. Thus, for any $K_v$-algebra $A$ we have an isomorphism \[\exp: \text{Lie}(U^{dr})\otimes A \cong U^{dr}(A).\]  
\end{remark}

Using Proposition \ref{Proposition - Coordinate ring of Pdr} we can define a filtration on $P^{dr}$ which turns out to be a filtration by lengths of iterated integrals.

\begin{definition}Let $\mathcal{P}^{dr}_n:= \U_n^{\vee}$. The \emph{Eilenberg-Maclane} filtration on $P^{dr}=\Spec(\mathcal{P})$ is defined by

\[
\O_X \subset \mathcal{P}^{dr}_1 \subset  \mathcal{P}^{dr}_2  \subset .. \subset  \mathcal{P}^{dr}  = \U^{\vee}.
\]
\end{definition}

The projection $\U \twoheadrightarrow \U_n$ corresponds then to the projection $P^{dr} \twoheadrightarrow P^{dr}_n$. One may then consider the co-product $\Delta$ on the fibres $x^*\U_n$ and identify the group-like elements with those of $P^{dr}_n$. In \cite[Theorem E]{wojt93} Wojtkowiak demonstrates that $\mathcal{P}^{dr}$ possess a Hodge filtration $F^{\bullet}$:
\[
\mathcal{P}^{dr}=F^0\mathcal{P}^{dr} \supset ... \supset F^i\mathcal{P}^{dr} \supset ...
\]
by sub $\O_X$-modules where the $F^i\mathcal{P}^{dr}$ are all ideals. The filtration $F^{\bullet}\mathcal{P}^{dr}$ in turn induces a filtration on $P^{dr}$:

\begin{definition}
The \emph{Hodge filtration} \index{$F^{\bullet}P^{dr}$} is such that $F^iP^{dr}$ has defining ideal $F^{-i+1}\mathcal{P}^{dr}$. 
\end{definition}

In loc.cit. it is shown that with the induced filtration $F^{\bullet}\Udr$on $\Udr$ the space $F^0P^{dr}$ becomes an $F^0\Udr$ torsor and is hence trivialised over an $K_v$-algebra $Z$ as $F^0\Udr$ is unipotent.

Now let $Y$ be the reduction of $X_v$ over $k$. Let $\Un(Y)$ be the category of overconvergent unipotent isocrystals on $Y$. If we basechange to $K_v$, this will be identified with unipotent connections convergent on every residue disk on $X$, and overconvergent near points of $D_v=D \otimes K_v$. For $c \in Y(k)$, let $]c[$ denote the residue disk of $c$. Then there is a fibre functor $e_c: (\mathcal{V}, \nabla) \mapsto \mathcal{V}(]c[)^{\nabla = 0}$ which takes the horizontal sections of $\mathcal{V}$ on the residue disk of $c$. Tannaka duality then gives us a crystalline fundamental group \index{$U^{cr}$}$U^{cr}$, and a right-torsor of crystalline paths \index{$P^{cr}(y)$}$P^{cr}(y)$ for $y \in Y(k)$. We similarly obtain $U^{cr}_n, P^{cr}_n(y)$ with overconvergent isocrystals of unipotency index less than or equal to $n$. 

The $q=|k|$-power map on $\O_Y$ induces a Frobenius automorphism $\phi:P^{cr}_n(y) \simeq P^{cr}_n(y)$. By the comparison theorem of Chiarellotto (\cite{chiar99}) between de Rham and crystalline fundamental groups we obtain a Frobenius automorphism $\phi$ on $P^{dr}_n(x)$. Besser shows that this Frobenius automorphism satisfies the following property:

\begin{thm}[\cite{besser02}, Theorem 3.1] \label{Theorem - Frobenius Invariant Element}
The map $U_n^{dr} \rightarrow U_n^{dr}$ given by $g \mapsto \phi(g)g^{-1}$ is an isomorphism.
\end{thm}

As a consequence of this Besser shows in \cite[Corollary 3.2]{besser02} that one may deduce that for any $x \in X_v(K_v)$ there is a unique Frobenius-invariant de Rham path \index{$p_n^{cr}(x)$}$p_n^{cr}(x)$ from $b$ to $x$. Both existence and uniqueness follow in a fairly straightforward manner from the above theorem. We are now ready to define the de Rham period map. In \cite{kim09} Kim defines the notion of an \textit{admissible $\Udr$-torsor}.

Let $T= \Spec(\mathcal{T})$ be a right $U_n^{dr}$-torsor over a $K_v$-scheme $Z$. We say that a $\Udr$-torsor $T$ is admissible if it has an Eilenberg-Maclane filtration; a Hodge filtration such that $F^0T$ is trivialised over $Z$; it has a Frobenius morphsim of $Z$-schemes $\phi$ semilinear with respect to the $\Udr$-action and preserving the Eilenberg-Maclane filtration; $\phi$ has a unique invariant $Z$-point; and, there is a universal injectivity property on the filtrations. 

\begin{remark}
Given our previous results observe then that $P^{dr}$ and $P^{dr}_n$ are admissible torsors for $U^{dr}$ and $U^{dr}_n$ respectively. 
\end{remark}

Let $T$ be an admissible torsor over a $K_v$-algebra $L$. Then as it is a right torsor, there is a $u_T \in U^{dr}_n$ such that $p^{cr}_T=p^H_Tu_T$. The point $u_T$ will be unique up to multiplication on the left by $F^0U^{dr}_n$, and so we have a $[u_T] \in F^0\U^{dr}_n \setminus U^{dr}_n$. There leads to the following bijective correspondence.

\begin{proposition}[Kim, \cite{kim09} Proposition 1] \label{Proposition - Classifying Space Good Redn}
There is a natural bijection between $F^0U_n^{dr} \setminus U_n^{dr}$ (resp. $F^0U^{dr}\setminus U^{dr}$) and isomorphism classes of admissble $U_n^{dr}$-torsors (resp. admissible $U^{dr}$-torsors) given by the map 
\[
T \mapsto [u_T].
\]
\end{proposition}

We are now in a position to define the de Rham period map, which is the de Rham realisation of the unipotent Albanese map. Note here that we will write $[P^{dr}(x)]$ as the image of $P^{dr}(x)$ under the map from the preceding proposition rather than $[u_{P^{dr}(x)}]$.

\begin{definition}
The de Rham period maps \index{$j^{dr},j^{dr}_n$}$j^{dr}, j^{dr}_n$ are defined as follows:
\begin{align*}
j^{dr}:& X_v(K_v) \rightarrow F^0U^{dr} \setminus U^{dr} \\
&x \mapsto [P^{dr}(x)]\\
& \\
j^{dr}_n: &X_v(K_v) \rightarrow F^0U^{dr} \setminus U^{dr} \twoheadrightarrow F^0U_n^{dr} \setminus U_n^{dr}\\
& x \mapsto [P^{dr}(x)] \mapsto [P^{dr}_n(x)]
\end{align*}
\end{definition}

In applications to Diophantine problems it is the finite level maps $j^{dr}_n$ that we are primarily interested in. Thus our aim should be to find explicit representatives for $[P^{dr}_n(x)]$ in $F^0U_n^{dr} \setminus U_n^{dr}$ for arbitrary $x \in X_v(K_v)$. To do this we need to find a Frobenius invariant $p_n^{cr}(x)\in P^{dr}_n(x)$, a trivialisation \index{$p_n^H(x)$}$p_n^H(x)\in F^0P^{dr}_n(x)$ and \index{$u_n(x)$}$u_n(x)\in U^{dr}_n$ such that $p_n^{cr}(x)=p_n^H(x)u_n(x)$. Then we can take $[P^{dr}_n(x)]=[u_n(x)]$. The element $p^{cr}_n(x)$ is computed as the parallel transport of $1 \in b^*\U_n$ to the fibre $x^*\U_n$. 

\begin{lemma}\label{Lemma - phi invariant path}[\cite{kim09}, \S 1]
The Frobenius invariant path \index{$p_n^{cr}(x)$}$p^{cr}_n(x)$ in $P^{dr}_n(x)$ is given by 

\[
p_n^{cr}(x) = 1 + \sum_{|w| \leq n} \int_b^x \alpha_w w
\]
where $\alpha_w = \alpha_{i_1}\alpha_{i_2}...\alpha_{i_s}$ if $w=A_{i_1}A_{i_2}...A_{i_s}$ and the sum is taken over all words in $A_0,..,A_{2g-r+2}$ of length at most $n$. 
\end{lemma}

\begin{remark}
In the above lemma the iterated integrals appearing are iterated Coleman integrals defined by \index{$\int_x^y \omega_1..\omega_r$}

\[
\int_x^y \omega_1..\omega_r := \int_x^y \omega_1(t_1) \int_x^{t_1} \omega_2(t_2)...\int_x^{t_{r-1}}\omega_{r-1}(t_{r-1}) \int_x^{t_r}\omega_r(t_r).
\]

Balakrishnan has developed algorithms for computing iterated Coleman integrals of this type on elliptic and hyperelliptic curves (see  \cite{balakrishnan13},\cite{balakrishnan10}) which have seen applications to Kim's non-abelian Chabauty. 
\end{remark}

Determining $p_n^H(x)$ will requires us to be able to compute the Hodge filtration on $P^{dr}_n$. The filtration $F^{\bullet}\mathcal{P}^{dr}$ induces a filtration on the dual $\U$ and each quotient $\U_n$. This will give each $\U_n$ the structure of a filtered connection and by computing this filtration we may identify the filtration on $P^{dr}$ and $P^{dr}_n$. In what follows we spend some time showing how this may be explicitly calculated in the case of elliptic curves and odd hyperelliptic curves.

\section{Logarithmic extensions of unipotent connections}\label{Section 3}
In the introduction we noted that in order to compute the Hodge filtration on the universal pointed connection of an affine curve we will need to make use of an approach due to Hadian in \cite{hadian11}. This will require us to compute a universal projective system of logarithmic connections on the compact curve extending the universal projective system of connections on the affine curve. This section is concerned with presenting a computational method to do this in the case that we have elliptic or odd hyperelliptic curves.  

Let $C,D$ and $X$ be as in Definition \ref{Definition - UC on General Affine Curve} and let \index{$\Omega^1_C(D)$}$\Omega^1_C(D)$ be the sheaf of \textit{logarithmic differentials} on $C$ along $D$. This sheaf will consist of differentials on $C$ regular on $X$ and with at worst logarithmic poles along $D$.

\begin{definition}
A \textit{logarithmic connection} on $C$ with logarithmic poles along $D$ is a vector bundle $\mathcal{V}$ equipped with a $K$-linear morphism of sheaves \index{$\nabla$}$\nabla$ such that

\[
\nabla : \mathcal{V} \rightarrow \mathcal{V} \otimes \Omega^1_C(D)
\]

satisfying the Leibniz condition as in Definition \ref{definition - connection}. A trivial logarithmic connection on $C$ along $D$ is a direct sum of copies of $(\O_C,d)$ where $\O_C$ is the structue sheaf of $C$ and $d$ is its exterior derivative. A morphism of logarithmic connections is defined analogously to morphisms of connections. 

For any open $Y \subset C$ a logarithmic connection on $C$ with poles along $D$ is a vector bundle $\mathcal{W}$ equipped with a $K$-linear morphism of sheaves 
\[
\nabla : \mathcal{V} \rightarrow \mathcal{V} \otimes \Omega^1_Y(D \cap Y).
\]
\end{definition}

\begin{definition}
Let $\mathcal{V}$ be a connection on $X=C-D$. A logarithmic extension of $\mathcal{V}$ to $C$ along $D$ is a logarithmic connection $\overline{\mathcal{V}}$ on $C$ with logarithmic poles along $D$ such that $\overline{\mathcal{V}}|_X \cong \mathcal{V}$. 
\end{definition}

That such extensions exist is a consequence of a theorem of Deligne (\cite[Proposition 5.2]{del70}). In this section we provide some algorithms to compute such extensions when $C$ is either an elliptic or odd hyperelliptic curve and $D$ consists of the point at infinity. 

In \cite{hadian11} Hadian defines unipotent logarithmic connections, which are iterated extensions of trivial logarithmic connections $(P \otimes _K \O_C, \id_P \otimes_K d)$. Here $P$ is a finite dimensional $K$-vector space and $d: \O_C \rightarrow \Omega^1_C(D)$ is the usual exterior derivative. The logarithmic extension \index{$\UU$}$\UU$ of the universal connection $\U$ on the affine curve $X$ is then used as a model for $\U$. Lemma 3.6 in loc.cit. suggests a practical way to compute the Hodge filtration of $\U$ by making use of this logarithmic extension $\UU$ and we will employ this approach in Section \ref{Section 4}. We now turn our attention to the computation of the logarithmic extension of $\U$. 

\subsection{A general algorithm for logarithmic extensions}
Recall that $\U$ consists of a projective system of pointed unipotent connections $\lbrace (\U_n,u_n) \rbrace$. We compute the logarithmic extension of $\U$ as a projective system \index{$\lbrace (\UU_n,u_n)\rbrace$}$\UU = \lbrace (\UU_n,u_n)\rbrace$ where each \index{$\UU_n$}$\UU_n$ is a logarithmic extension of $\U_n$. We shall see later that this construction will ensure that the resulting projective system is a universal projective system among pointed unipotent logarithmic connections. In order to compute logarithmic extensions $\UU_n$ of the $\U_n$ we utilise a description in terms of an open covering of $C$. That is, an object is described on open subsets of some cover together with gluing morphisms (\textit{descent datum}) which satisfy some cocycle condition.\\

\begin{definition}
Descent for logarithmic connections on $C$ along $D$ is given by the following descent datum: 

\begin{enumerate} 
\item An open cover $(Y^i)_i$ of $C$
\item Logarithmic connections $\V_i=(\O^r_{Y^i},\nabla_i)$ with poles along $Y^i \cap D$
\item Isomorphisms of logarithmic connections $ G_{ij}: (\O_{Y^{ij}}^r,\nabla_i|^{Y_{ij}}) \xrightarrow{\sim} (\O_{Y^{ij}}^r,\nabla_j|_{Y^{ij}})$ such that for all $i$ we have $G_{ii}=\text{id}_{(\O_{Y^{i}}^r,\nabla_i)}$ and such that for all $i,j,k$ the following cocycle condition is satisfied:\[
(G_{jk}|_{Y^{ijk}})\circ (G_{ij}|_{Y^{ijk}})= (G_{ik}|_{Y^{ijk}})
\]
\end{enumerate} 

This descent datum will be written as $(\V_i,G_{ij})$.
\end{definition}

To construct the descent datum we will need the following lemma which follows from an easy computation. 

\begin{lemma} \label{definition - gauge transformations}
Given a (logarithmic) connection $\mathcal{V}$ on a curve $Z$, suppose that with respect to the local basis $(e_i)$ it has connection matrix $\Omega$. If $G$ is an automorphism of $\mathcal{V}$, the transformation 

\[
\Omega \mapsto G^{-1} dG + G^{-1} \Omega G
\]is called a \textit{gauge transformation} of $\Omega$. This is the connection matrix of $\mathcal{V}$ with respect to the local basis $(G^{-1}e_i)$.  
\end{lemma}

We now turn our attention towards computing the logarithmic extensions of the connections $\U_n$. We want an iterative algorithm by which we may compute the logarithmic extensions $\UU_n$ of the $\U_n$ successively. Observe that $\UU_0 = (\O_C,d)$ is a logarithmic extension of $\U_0=(\O_X,d)$ and we take this as our base case. The construction of the extensions $\UU_n$ is based upon the following observation in \cite{hadian11}:

\begin{proposition}[\cite{hadian11}, Lemma 2.3 \& Proposition 2.6]\label{Proposition - Universal Logarithmic Connection}
For every $n \geq 0$ there exists an extension $\UU_{n+1}$ of $\UU_n$ by $(V_{dr}^{\otimes(n+1)} \otimes \O_C,d)$ such that $1 \in b^* \UU_{n+1}$ maps to $1 \in b^*\UU_n$ under projection. Let $\mathcal{V}$ be a unipotent logarithmic connection on $C$ with poles along $D$ of unipotency index $m$. Then for all $v \in b^*\mathcal{V}$ and $n \geq m$ there exists a unique morphism $\phi_v: \UU_n \rightarrow \mathcal{V}$ and $1 (\in b^*\UU_n) \mapsto v \in (b^*\mathcal{V})$. 
\end{proposition}

In the parlance of Definition \ref{Definition - Universal Projective System} we then say that $\lbrace(\UU_n, 1) \rbrace$ forms a universal projective system in the category of pointed unipotent logarithmic connections on $C$ with logarithmic poles along $D$ and we denote this projective system by $\UU$. By an abuse of notation we will also denote the universal pro-unipotent logarithmic connection $\lim_{\leftarrow} \UU_n$ by $\UU$ when the context is clear. In light of Proposition \ref{Proposition - Universal Logarithmic Connection} we will construct finite level extensions $\UU_n$ as extensions fitting into an exact sequence
\begin{equation} \label{Equation - Exact Sequence original}
0 \rightarrow V_{dr}^{\otimes (n+1)} \otimes \O_C \rightarrow \UU_{n+1} \rightarrow \UU_{n} \rightarrow 0
\end{equation} of logarithmic connections. For each $n$ we shall, therefore, require a suitable projection map $\UU_n \twoheadrightarrow \UU_{n-1}$. To incorporate this into the construction we need to define we what we mean by morphisms of descent data. In the following definition we assume that both descent data are described over a single covering of $C$.

\begin{definition}
A morphism of descent data $(\V_i,G_{ij})$ and $(\V_i',G'_{ij})$ for a logarithmic connection on $C$ along $D$ is given by a family $\rho=(\rho_i)_{i}$ of morphisms of logarithmic connections $\rho_i:\V_i\rightarrow \V_i'$ such that all of the  diagrams 

\[
\begin{tikzcd}
\V_i|_{Y^{ij}} \arrow[swap]{d}{\rho_i|_{Y^{ij}}} \arrow{r}{G_{ij}} & \V_j|_{Y^{ij}} \arrow{d}{\rho_j|_{Y^{ij}}}\\
\V'_i|_{Y^{ij}} \arrow{r}{G_{ij}'} & \V'_j|_{Y^{ij}}
\end{tikzcd}
\]  commute. 
\end{definition}

We compute gauge transformations $G$ which map the connection matrix of $\U_{n+1}$ over $X$ to a connection matrix with at worst logarithmic poles at the  points of $D$. The gauge transformations will be the transition functions for the descent data. We then determining a suitable open cover $(Y_{n+1}^i)_i$ of $C$ such that these connection matrices define logarithmic connections $\U^i_{n+1}$ on each of the patches $Y^{n+1}_i$. Over the patch $Y^{n+1}_i$ we simply define the extension to have bundle $R_{n+1} \otimes \O_{Y_{n+1}^i}$. This is done subject to the condition that for each candidate open patch $Y_{n+1}^i$ we should have  a commutative diagram

\begin{equation}\label{Equation - descent projection}
\begin{tikzcd}
\U_n|_{X \cap Y_{n+1}^i} \arrow[twoheadrightarrow]{d} \arrow{r}{G} & \U^i_n|_{X \cap Y_{n+1}^i} \arrow[twoheadrightarrow]{d}\\
\U_{n-1}|_{X \cap Y_{n+1}^i} \arrow{r}{G} & \U^i_{n-1}|_{X \cap Y_{n+1}^i}
\end{tikzcd}
\end{equation} Then the logarithmic connections on the $Y_{n+1}^i$ with poles along  together with the connection $\U_n$ on $X$ and the gauge transformations define descent datum for a logarithmic connection on $C$ with log poles along $D$. Note that it will be convenient for the computations that follow to describe the gauge transformations $G$ as elements of $K(C)\otimes_K \mathfrak{gl}_N$ for some $N$. 

\begin{algorithm} \label{Algorithm - Computing GT on UC}
(Computing the logarithmic extension of the universal connection on $X=C-D$) \newline

\textbf{Input}
\begin{itemize}
\item A smooth projective curve $C$ over a field $K$ of characteristic $0$, non-empty divisor $D = \lbrace d_1,...,d_r \rbrace $ defined over $K$, $X=C-D$
\item The universal connection $\U_n$ on $X=Y^0_n=Y^0_{n+1}$ with respect to a basis of $H^1_{dr}(X)$
\item The logarithmic extension $\UU_n$ of $\U_n$ defined by the following descent datum:

\begin{enumerate}
\item Trivial logarithmic connections  $(\U^i_n,d+C^i_n)$ over open subsets $Y^i_n \subset C$ where $d_i \in Y^i_n$ for $i \neq 0$ and $(Y^i_n)_i$ a cover of $C$
\item Gauge transformations $G^{ij}_{n}$ with:

\begin{itemize}
\item $G^{ij}_{n}:C^i_{n} \mapsto C^j_{n}$ on $Y^{ij}_n=Y^j_{n} \cap Y^i_{n}$ compatible with projections to level $n-1$ for all $i,j$
\item $G^{ij}_{n}=(G^{ji}_{n})^{-1}$ on $Y^{ij}_{n}$ for all $i,j$
\item $G^{ii}_{n}=\text{id}$ on $Y^{ii}_{n}$ for all $i$
\item $G^{jk}_{n} \circ G^{ij}_{n}=G^{ik}_{n}$ on $Y^{ijk}_{n}=Y^{i}_{n} \cap Y^j_{n} \cap Y^k_{n}$ for all $i,j,k$ (cocycle condition)
\end{itemize}
\end{enumerate}  
\end{itemize}

\textbf{Output}

\begin{itemize}
\item The logarithmic extension $\UU_{n+1}$ of $\U_{n+1}$ defined by the following descent datum:

\begin{enumerate} 
\item Trivial logarithmic connections  $(\U^i_{n+1},d+C^i_{n+1})$ over open subsets $Y^i_{n+1} \subset C$ with $d_i \in Y^i_{n+1}$ for $i \neq 0$ and $(Y^i_{n+1})_i$ a cover of $C$
\item Gauge transformations $G^{ij}_{n+1}$ with:

\begin{itemize}
\item $G^{ij}_{n+1}:C^i_{n+1} \mapsto C^j_{n+1}$ on $Y^{ij}_{n+1}=Y^j_{n+1} \cap Y^i_{n+1}$ compatible with projections to level $n$ for all $i,j$
\item $G^{ij}_{n+1}=(G^{ji}_{n+1})^{-1}$ on $Y^{ij}_{n+1}$ for all $i,j$
\item $G^{ii}_{n+1}=\text{id}$ on $Y^{ii}_{n+1}$ for all $i$
\item $G^{jk}_{n+1} \circ G^{ij}_{n+1}=G^{ik}_{n+1}$ on $Y^{ijk}_{n+1}=Y^{i}_{n+1} \cap Y^j_{n+1} \cap Y^k_{n+1}$ for all $i,j,k$ (cocycle condition)
\end{itemize}
\end{enumerate}
\end{itemize}

\textbf{Algorithm}

\begin{enumerate}
   
\item[(1)] For $i \in \lbrace 1,.., r \rbrace$

\begin{enumerate}
\item Compute a gauge transformation $G^{0i}_{n+1}$ of $C_{n+1}$, the connection matrix of $\U_{n+1}$, such that 

\begin{itemize}
\item $G^{0i}_{n+1}$ is compatible with projection to level $n$
\item the image $C^i_{n+1}$ of $C_{n+1}$ has at worst logarithmic poles at $d_i$
\end{itemize}

\end{enumerate}

\item[(2)] For $i \in \lbrace 1,.., r \rbrace$

\begin{enumerate}
\item Choose open $Y_{n+1}^i \subset C$ containing $d_i$ such that $C^i_{n+1}$ has no poles on $Y_{n+1}^i$ except possibly at $d_i$; together with $Y_{n+1}^0=X$ these cover $C$
\item Let $\U^i_{n+1} = R_{n+1} \otimes \O_{Y_{n+1}^i}$
\item Give $\U^i_{n+1}$ the logarithmic connection $d+C^i_{n+1}$.
\end{enumerate}  
\item[(3)] Define:
\begin{itemize}
\item $G^{j0}_{n+1}:=(G^{0j}_{n+1})^{-1}$ for all $j$
\item $G^{ii}_{n+1}:=\text{id}$ for all $i$
\item $G^{ij}_{n+1}:=G^{0j}_{n+1}(G^{0i}_{n+1})^{-1}$ for all $i,j$
\end{itemize} 

\item[(4)] Glue the logarithmic connections $\U^i_{n+1}$, $\U_{n+1}$ together via the isomorphisms $G_{n+1}^{ij}$ to obtain a logarithmic connection $\UU_{n+1}$ with log poles along $D$.   
\end{enumerate}
\end{algorithm}

\begin{remark}
It should be noted that the above algorithm relies on the ability to complete Step 1) a) for each $i \in \lbrace 1,..,r \rbrace$. The next two sections show that we can do this in the case that $C$ is an elliptic curve or an odd model of a  hyperelliptic curve. 
\end{remark}

\begin{remark}
It may be the case that several of the opens $Y^i_{n+1}$ coincide i.e. that a single $Y^{i}_{n+1}$ may be chosen to contain several of the points of $C$ missing from $X$. 
\end{remark}

\begin{remark}
Although the above algorithm implies that the choice of open cover $(Y^i_{n})_i$ at each level of iteration depends on $n$, the results in subsections \ref{Subsection - Logarithmic Extensions on EC} and \ref{Subsection - Logarithmic Extensions on HEC} show that we can often eliminate this dependence on $n$ through some judicious choices for $G^{ij}_{n}$. It should be possible to replicate this in other more general examples of curves by choosing the entries of the $G^{0i}_n$ to be polynomial in some $F$ with a single simple pole at the point $d_i$ where possible. 
\end{remark}

\begin{remark}
As noted in the opening paragraph of this section the logarithmic connection $\UU$ constructed above is universal among pointed unipotent logarithmic connections on $C$ with poles along $D$. In the following sections we shall see that there is a certain amount of choice available when computing these extensions. These different choices give rise to isomorphic unipotent logarithmic connections each of which satisfies the universal property described in Definition \ref{Definition - Universal Projective System}. 
\end{remark}

\begin{remark}As the universal connection $\widetilde{\U}$ on $C$ is also a pointed logarithmic connection there is a morphism of pointed logarithmic connections $\Phi: \UU \rightarrow \widetilde{\U}$. We thus exhibit $\widetilde{\U}$ as a \textit{maximal quotient} of $\UU$ without poles. 
\end{remark}

\subsection{Logarithmic extensions on affine elliptic curves} \label{Subsection - Logarithmic Extensions on EC}

In what follows, we describe this process explicitly for an arbitrary elliptic curve. However, it should be noted that the results presented in this section should easily translate to any general smooth projective curve $C$ punctured at $r$ points. Replacing the dimensions of $H^1_{dr}(X/K)$ with a variable $s:=2g+r-2$ the same results will still apply. However, where $r>1$ care must be taken to compute the logarithmic extension near \textit{each} of the punctured points individually then glue the resulting logarithmic connections together. This would add another layer of notational complexity in what is admittedly an already notation heavy set of results, which is why we have elected here to stick to a simpler example. 

Let $C$ be an elliptic curve over a field $K$ of characteristic $0$ with $K$-rational point at infinity \index{$\PP$}$\PP$. Let $X:=C-\lbrace \PP \rbrace$ be the punctured elliptic curve with model $y^2=f(x)$ where $f(x) \in K[x]$ is a degree 3 polynomial. Recall that $C$ is a genus $1$ curve. We specialise the construction of Definition \ref{Definition - UC on General Affine Curve} to $X$:

Let $\alpha_0,\alpha_1 \in H^0(X, \Omega^1_X)$ be 1-forms on $X$ with $\alpha_0$ regular on $C$ and $\alpha_1$ having a pole of order $2$ at $\PP$ such that the cohomology classes of $\alpha_0,\alpha_1$ are a $K$-basis for $H^1_{dr}(X/K)$. Let $R_n$ and $\U_n$ be as in Definition \ref{Definition - UC on General Affine Curve}. 

\begin{remark}\label{basis uc ec}
It will be convenient at this stage to fix a choice of ordered basis \index{$\B_n$}$\B_n$ for $R_n$. We take as a $K$-basis the words of length less than or equal to $n$ with a graded lexicographic ordering such that $A_0 < A_1 < 1$. With respect to this ordered basis we denote by \index{$w^k_l$}$w^k_l$ the $k$-th word of length $l$. For example, the ordering on all words of length up to $2$ is $w^1_2, w^2_2, w^3_2, w^4_2, w^1_1, w^2_1, w^1_0$. 

Note that there are $2^l$ words of length $l$ and $\B_n$ has order $2^{n}+2^{n-1}+..+1=2^{n+1}-1$. Thus if we have a word $w^k_l$, then $A_0 w^k_l = w^k_{l+1}$ and $A_1 w^k_l=w^{k+2^l}_{l+1}$. We can describe the action of $\nabla$ on a basis for $R_n$:

\[
\nabla( w^k_l) =	\begin{cases} 
						-A_0w^k_l\alpha_0 - A_1w^k_l\alpha_1 = -w^k_{l+1}\alpha_0-w^{k+2^l}_{l+1}\alpha_1 & \text{ if } l \leq n-1 \\ 
						0 & \text{ if } l=n 
					\end{cases}
\]

\end{remark}

\begin{lemma}\label{Lemma - Connection Matrix UC}
The connection matrix of $\U_0$ is the zero matrix. If \index{$C_n$}$C_n$ is the connection matrix of $\U_n$ with respect to the basis $\B_n$, then  

\[
C_{n+1} =\left(	\begin{array}{cc}
					0_{2^{n+1} \times 2^{n+1}} & D_{n+1} \\
					0_{2^{n+1}-1 \times 2^{n+1}} & C_{n}
					\end{array} \right)
\] is the connection matrix of $\U_{n+1}$ with respect to the basis $\B_{n+1}$ where \index{$D_{n+1}$}

\[
D_{n+1} = \left( \begin{array}{cc}
					-\alpha_0 I_{2^n} & 0_{2^n \times 2^{n}-1} \\
					-\alpha_1 I_{2^n} & 0_{2^n \times 2^{n}-1} 
				\end{array} \right).
\] Here $0_{r \times s}$ is the $r \times s$ null matrix and $I_r$ is the $r \times r$ identity matrix. 
\end{lemma}

\begin{proof}
This is just a straightforward calculation given Remark \ref{basis uc ec}. 
  \end{proof}

We now explicitly compute the extension of the $\U_n$ to logarithmic connections on $C$ by application of Algorithm \ref{Algorithm - Computing GT on UC}. Recall that the logarithmic extension $\UU_0$ of $\U_0$ is defined to be $(\O_C,d)$. We now present an example calculation of the computation of the extension at level $1$. 

\begin{example} \label{Example - Level 1 Extension EC}
The ordered basis elements for $R_1$ are $\B_1=\lbrace A_0,A_1,1 \rbrace$. Then we have 
\begin{align*}
&\nabla(A_0)=\nabla(A_1)=0 \\
&\nabla(1) = -A_0\alpha_0 - A_1 \alpha_1
\end{align*} The connection matrix of $\nabla$ on $\U_1=R_1 \otimes \O_X$ with respect to $\B_1$ is 
\[
C_1 =	\left( \begin{array}{ccc}
			0 & 0 & -\alpha_0 \\
			0 & 0 & -\alpha_1 \\
			0 & 0 & 0 \\
		\end{array} \right) 
\] There is a natural projection map $\pi_1: \U_1 \twoheadrightarrow \O_X$. We wish to find an open $Y \subset C$ containing $\PP$, a connection $\U'_1$ on $Y$, and a gauge transformation $G_1$ such that over $X \cap Y$ we have a commutative diagram
\begin{equation}\label{Equation - Commutative Condition Level 1}
\begin{tikzcd}
\U_1 \arrow[swap]{d}{\pi_1} \arrow{r}{G_1} & \U'_1 \arrow{d}{\pi_1'}\\
\O_X \arrow{r}{\id} & \O_Y
\end{tikzcd}
\end{equation} where $\pi'_1$ is a projection map from $\U'_1 \rightarrow \O_Y$. In order to ensure that $G_1$ satisfies (\ref{Equation - Commutative Condition Level 1}) it must be of the form 

\[
G_1 = \left( \begin{array}{ccc}
			1 & 0 & h^1 \\
			0 & 1 & h^2 \\
			0 & 0 & 1 \\
		\end{array} \right)
\] where $h^1,h^2 \in K(C)$. The gauge transformation of $C_1$ by $G_1$ is

\[
C'_1= G_1^{-1}dG_1+G_1^{-1}C_1G_1=C_1+dG_1=	\left( 	\begin{array}{ccc}
			0 & 0 & -\alpha_0 +dh^1 \\
			0 & 0 & -\alpha_1+dh^2 \\
			0 & 0 & 0 \\
		\end{array} \right)
\]

Choose $h^1=0$ and $h^2=F \in K(C)$ with a pole of order $1$ at $\PP$ such that $-\alpha_1+dF$ is regular at $\PP$. Let $Y$ be the open set $\lbrace P \in C: \alpha_1+dF \text{ is regular at } P \rbrace \subset C$. Then we define $\U'_1$ to be the connection on $Y$ defined with connection matrix $C_1'$ and bundle  $R_1 \otimes \O_Y$. The logarithmic extension $\UU_1$ of $\U_1$ is then described by the descent datum of the logarithmic connections $\U_1$ (on $X$) and $\U'_1$ (on $Y$) together with the gauge transformation $G_1: \U_1|_{X \cap Y} \xrightarrow{\sim} \U'_1|_{X \cap Y}$. 
\end{example}

Having dealt with the first non-trivial case we now turn our attention to computing the extension of $\U_{n+1}$ for general level. We will do this by showing that we can calculate a suitable gauge transformation $G_{n+1}$ given an extension $\UU_n$ defined at level $n$. As in Example \ref{Example - Level 1 Extension EC} we find that $G_{n+1}$ should be of the form \index{$G_{n+1}$}\[
G_{n+1} = \left( 	\begin{array}{cc}
						I_{2^{n+1}} & H_{n+1} \\
						0 & G_n 
					\end{array} \right)
\] where \index{$H_{n+1}$}$H_{n+1}$ is some $2^{n+1} \times 2^{n+1}-1$ matrix over $K(C)$ which we need to determine. The bundle $\U'_{n+1}$ should again be a trivial bundle $R_{n+1} \otimes \O_Y$, where we would like to choose $Y$ as in Example \ref{Example - Level 1 Extension EC}. 

\begin{lemma}\label{Lemma - Gauge Transformation of Connection Matrix}
Let $n \geq 1$ and let $C_n$ (resp. $C_{n+1}$) be the connection matrix of $\U_n$ (resp. $\U_{n+1}$) with respect to the basis $\B_n$ (resp. $\B_{n+1}$). Suppose that $\U_n$ extends to a logarithmic connection $\UU_n$ described by a logarithmic connection $\U_n$ over $X$, a logarithmic connection $\U'_n$ over some open $Y$ and a gauge transformation $G_n$ over $X \cap Y$. Suppose that $G_n$ takes $C_n$ to $C'_n$, the connection matrix of $\U'_n$. If $G_{n+1}$ is a gauge transformation of $C_{n+1}$  over $X\cap Y$ of the form 

\[
G_{n+1} = \left( 	\begin{array}{cc}
						I_{2^{n+1}} & H_{n+1} \\
						0 & G_n 
					\end{array} \right)
\]then the gauge transformation of the matrix $C_{n+1}$ of Lemma \ref{Lemma - Connection Matrix UC} by $G_{n+1}$ is \index{$C'_{n+1}$}

\[
C'_{n+1} = G_{n+1}^{-1}dG_{n+1}+G_{n+1}^{-1}C_{n+1}G_{n+1}= \left( 	\begin{array}{cc}
						0 & D_{n+1}' \\
						0 & C_n' 
					\end{array} \right)
					\]where \index{$D'_{n+1}$}$D_{n+1}' = D_{n+1}G_n+dH_{n+1}-H_{n+1}C'_n$ and \[
D_{n+1} = \left( \begin{array}{cc}
					-\alpha_0 I_{2^n} & 0_{2^n \times 2^{n}-1} \\
					-\alpha_1 I_{2^n} & 0_{2^n \times 2^{n}-1} 
				\end{array} \right).
\]
\end{lemma}

\begin{proof}

This follows easily from the definition of the gauge transformation and noting that \[
G_{n+1}^{-1} = \left( 	\begin{array}{cc}
						I_{2^{n+1}} & -H_{n+1}G_n^{-1} \\
						0 & G_n^{-1} 
					\end{array} \right).
\]
  \end{proof} 

\begin{remark}
Note that in the above $C_n'$ will be a matrix of $1$-forms with at worst logarithmic poles at $\PP$ by the assumption that $\UU_n$ is logarithmic. Therefore, in order to compute a suitable gauge transformation $G_{n+1}$ we need to find a matrix of functions  $H_{n+1}$ such that \[dH_{n+1}+D_{n+1}G_n-H_{n+1}C_n'\] has entries with at worst logarithmic poles at $\PP$. This is the content of the following theorem. 
\end{remark}

In the course of the proof of the theorem we will need to make use of the following two auxiliary functions. 

\begin{definition} \label{Definition - Auxiliary Functions}
Let \index{$\psi$}\index{$\phi$}$\psi,\phi:\Z^3 \times \N \mapsto \Z$ be the functions defined by \begin{equation} \label{Equation - Block Nesting Function}
\psi(r,i,j,k) := \begin{cases} \lfloor k^{i-j}r \rfloor +1 & \text{ if } r \not \equiv 0 \mod k^{j-i} \\
					\lfloor k^{i-j}r \rfloor & \text{ otherwise} \end{cases} 
\end{equation} and 
\begin{equation}
\phi(r,i,j,k):= (r-1)k^i-(\psi(r,i,j,k)-1)k^j.
\end{equation}
\end{definition}

\begin{remark}
It is helpful at this stage to think about the functions $\phi$ and $\psi$ in terms words in the ordered basis $\B_n$ from Remark \ref{basis uc ec}. Suppose that we have an alphabet with $k$-letters $A_i$ and let $w$ be the $r$-th word of length $j$ with respect to the lexicographic ordering $A_0<..<A_{k-1}<1$. Then for $i <j$ the first $i$ letters of $w$ are the $\psi(r,i,j,k)$-th word of length $i$ and the last $j-i$ letters of  $w$ are the $(\phi(r,i,j,k)2^{-i}+1)$-th word of length $j-i$.
\end{remark}

We are now ready to state and prove the following theorem. 

\begin{thm}\label{Theorem - Compute GT on UC of EC} Let $C,X, \U_n$ and $C_n$ be as above. Then there is an open $Y \subset C$ containing $\PP$ such that for all $n$ there is an isomorphism $G_n$ of $\U_n$ with a connection $\U'_{n} = (R_n \otimes \O_Y,d+C_n')$ over $X \cap Y$ which is compatible with projection to lower level. The  isomorphism $G_n$ can be chosen such that 

\[
G_n =  \left( \begin{array}{cc}
						I_{2^n} & H_n \\
						0 & G_{n-1} \end{array} \right)
\] where $H_n$ is a matrix of functions on $C$ and $C_n'$ is a matrix of $1$-forms on $C$ such that 
\index{$H^{r,i}_n$}\index{$C^{r,i}_n$}
\[H_n = \left( \begin{array}{ccccc}
						H^{1,n-1}_n& \hdots & H^{1,i}_n & \hdots & H^{1,0}_n \\
						\vdots & & \vdots & & \vdots \\
						H^{2,n-1}_n& \hdots & H^{2^{n-i},i}_n & \hdots & H^{2^n,0}_n \end{array} \right); \; \;
C'_n = \left( \begin{array}{ccccc}
			C^{1,n}_{n}& \hdots & C^{1,i}_{n} & \hdots & C^{1,0}_{n} \\
			\vdots & & \vdots & & \vdots \\
			\vdots & \hdots & C^{2^{n+1-i}-1,i}_n & \hdots & C^{2^{n+1}-1,0}_n \\
			C^{2,n}_n & \hdots & C^{2^{n+1-i},i}_n & \hdots& C^{2^{n+1},0}_n \end{array} \right)
\] where 
\begin{enumerate}
\item[(A)] For all $r,i$ there are rational functions \index{$h^{r,i}_n$}$h^{r,i}_{n}\in K(C)$ such that $H^{r,i}_n= h^{r,i}_{n}I_{2^i}$
\item[(B)] For all $i$ we have 

\begin{enumerate}
\item[(B1)] For all $r=1,..,2^{n+1-i}-2$ there are $1$-forms \index{$c^{r,i}_n$}$c^{r,i}_{n+1} \in \Omega^1_{X/K}$ on $X$ with at worst logarithmic poles at $\PP$ such that $C^{r,i}_n = c^{r,i}_n I_{2^i}$
\item[(B2)] $C^{2^{n+1-i}-1,i}_n = 0_{2^j \times 2^j}$
\item[(B3)] $C^{2^{n+1-i},i}_n=0_{2^i-1 \times 2^i}$
\end{enumerate}
\end{enumerate} 

\end{thm}

\begin{remark}
Note that the matrices $H^{r,i}_n$ and $C^{r,i}_n$ have entries indexed by words. We can think of these two matrices in the following way: if I take a basis element coming from a word of length $i$ in $\B_n$ and apply the gauge transformation $G_n$ to it then \index{$H^{r,i}_n$}$H^{r,i}_n$ gives the coefficients of the basis elements in the result coming from words of length $n$ whose first $n-i$ letters are the word with index $r$. Similarly, if I differentiate a basis element coming from a word of length $i$ in $\B_n$ using the connection $d+C'_n$ then \index{$C^{r,i}_n$}$C^{r,i}_n$ gives the coefficients of the basis elements in the result coming from words of length $n$ whose first $n-i$ letters are the word with index $r$.
\end{remark}

\begin{proof}
We shall proceed by induction and assume that the statement of the theorem holds for all $r\leq n$. In Example \ref{Example - Level 1 Extension EC} we have already shown that this is true in the case that $n=1$ with $h^{1,0}_1=0$ and $h^{2,0}_1=F$. Let $Y$ be as in the aforementioned example. 

In Lemma \ref{Lemma - Gauge Transformation of Connection Matrix} we saw that $C_{n+1}'$ is of the form 

\[
 \left( \begin{array}{cc}
 			0 & D'_{n+1}\\
 			0 & C_n' \end{array} \right)
\] where $D'_{n+1} = dH_{n+1}+D_{n+1}G_n-H_{n+1}C'_n$. By inductive hypothesis $C_n'$ satisfies condition \textit{(B)}. Therefore, we must show that $H_{n+1}$ can be chosen such that 
\begin{enumerate}
\item[(i)] $D'_{n+1}$ has entries with at worst logarithmic poles at $\PP$ 
\item[(ii)] that $H_{n+1}$ satisfies condition \textit{(A)}
\end{enumerate}
and conclude then that $C_{n+1}'$ satisfies \textit{(B)}. 

We split $H_{n+1}$ and $C'_{n+1}$ into block matrices 
\[
H_{n+1} = \left( \begin{array}{ccccc}
						H^{1,n}_{n+1}& \hdots & H^{1,i}_{n+1} & \hdots & H^{1,0}_n \\
						\vdots & & \vdots & & \vdots \\
						H^{2,n}_{n+1}& \hdots & H^{2^{n+1-i},i}_{n+1} & \hdots & H^{2^{n+1},0}_{n+1} \end{array} \right); \; \; C'_{n+1} = \left( \begin{array}{ccccc}
			C^{1,n+1}_{n+1}& \hdots & C^{1,i}_{n+1} & \hdots & C^{1,0}_{n+1} \\
			\vdots & & \vdots & & \vdots \\
			\vdots & \hdots & C^{2^{n+2-i}-1,i}_{n+1} & \hdots & C^{2^{n+2}-1,0}_{n+1} \\
			C^{2,n+1}_{n+1} & \hdots & C^{2^{n+2-i},i}_{n+1} & \hdots& C^{2^{n+2},0}_{n+1} \end{array} \right)
\] where each $H^{r,i}_{n+1}$ is a $2^i \times 2^i$ matrix; for $r=1,..,2^{n+2-i}-1$ each $C^{r,i}_{n+1}$ is a $2^i \times 2^i$ matrix; and, each $C^{2^{n+2-i},i}_{n+1}$ is a $2^i-1 \times 2^i$ matrix. In order to complete the induction step we will try to express $C^{r,i}_{n+1}$ in terms of block matrices from $H_n,H_{n+1}$ and $C_n'$. 

It will be convenient to  say that $H^{r,i}_{n+1}$ in $H_{n+1}$ \textit{stems} from the block matrix $H^{s,j}_{n+1}$ for $i+1 \leq j \leq n$ if the rows of $H_{n+1}$ containing $H^{r,i}_{n+1}$ are a subset of the rows of $H_{n+1}$ containing $H^{s,j}_{n+1}$. It is not difficult then to see that $H^{r,i}_{n+1}$ in $H_{n+1}$ stems from the block matrix \index{$\psi$}$H^{\psi(r,i,j,2),j}_{n+1}$. 

A simple calculation shows that 

\[
D_{n+1}G_n = \left( \begin{array}{cc}
						-\alpha_0 I_{2^n} & -\alpha_0 H_n \\
						-\alpha_1 I_{2^n} & -\alpha_1 H_n
					\end{array} \right)
\] and by the inductive hypothesis we compute that the contribution of $dH_{n+1}+D_{n+1}G_n$ to $C^{r,i}_{n+1}$ is 

\[
\begin{cases} dH^{r,n}_{n+1} - \alpha_{r-1} I_{2^n} & \text{ if } i=n \\
			  dH^{r,i}_{n+1} - \alpha_{\psi(r,i,n,2)-1}h_n^{r-(\psi(r,i,n,2)-1)2^{n-i},i}I_{2^i} & \text{ if } i < n. \end{cases}
\]

The contribution coming from $H_{n+1}C'_n$ is more complicated to work out. We introduce a second inductive step to the argument, with our inductive hypothesis being that condition \textit{(A)} is satisfied by $H_{n+1}$ for $i=n,n-1,...,j+1$. We then show that this implies that condition \textit{(A)} is satisfied by $H_{n+1}$ for $i=j$. 

For our base case, we need to show that $H_{n+1}$ satisfies condition \textit{(A)} for $i=n$. It is straightforward, however, to see that 

\[
H_{n+1}C'_n = \left( \begin{array}{cc}
						0_{2^n \times 2^n} & * \\
						0_{2^n \times 2^n} & *  \end{array} \right) 
\] and so $C^{r,n}_{n+1} = dH^{r,n}_{n+1} - \alpha_{r-1} I_{2^n}$ for $r=1,2$. We choose $H^{r,n}_{n+1}=h^{r,n}_{n+1}I_{2^n}$ where $h^{r,n}_{n+1} \in K(C)$  such that $dh^{r,n}_{n+1} - \alpha_{r-1}$ is regular at $\PP$. To ensure that we may continue to define the resulting logarithmic connection $\U'_{n+1}$ over the same open $Y$ as $\U'_n$, we choose $h^{r,n}_{n+1}$ to be polynomial in $h^{2,0}_1=F$. So we have that condition \textit{(A)} holds for $i=n$. 

Now let $j < n$ and assume then that condition \textit{(A)} holds for $i=n,...,j+1$. Now, rows $(r-1)2^j+1$ to $(r-1)2^j+2^j$ are the rows of $H_{n+1}$ containing $H^{r,j}_{n+1}$. They have the form

\[
\left( \begin{array}{ccccccc} 
			H^{r,j,n}_{n+1} & H^{r,j,n-1}_{n+1} & ... & H^{r,j,j+2}_{n+1}& H^{r,j,j+1}_{n+1} & H^{r,j}_{n+1} & ... \end{array} \right)
\] where each \index{$H^{r,j,i}_{n+1}$}$H^{r,j,i}_{n+1}$ is a $2^j \times 2^i$ matrix. By the inductive hypothesis on $H^{r,i}_{n+1}$ we have the for $i=n,..,j+1$ 

\[
H^{r,j,i}_{n+1} = \left( \begin{array}{ccc} 0_{2^j \times \phi(r,j,i,2)} &  h^{\psi(r,j,i,2),j}_{n+1}I_{2^j} & 0_{2^j \times 2^i - (\phi(r,j,i,2)+2^j)} \end{array} \right) .
\]

Therefore we conclude that the contribution of $H_{n+1}C'_n$ to $C^{r,j}_{n+1}$ is 

\[
\begin{split}
 & \left( \begin{array}{ccccccc} 
			H^{r,j,n}_{n+1} & H^{r,j,n-1}_{n+1} & ... & H^{r,j,j+2}_{n+1}& H^{r,j,j+1}_{n+1} & H^{r,j}_{n+1} & ... \end{array} \right) \times \left( \begin{array}{ccccc}
						c^{1,j}_{n}I_{2^j} & \hdots & c^{2^{n+2-j}-2,j}_n I_{2^j}& 0_{2^j \times 2^j} & 0_{2^j-1\times 2^j}
						\end{array} \right)^t \\
& = \left( \sum_{t=j+1}^n c^{(2^n+2^{n-1}+..+2^{t+1}+\phi(r,j,t,2))2^{-j}+1,j}_n h^{\psi(r,j,t,2),t}_{n+1} \right)  I_{2^j} \\
 & = \left( \sum_{t=j+1}^n c^{2^{n+1-j}-2^{t+1-j}+\phi(r,j,t,2)2^{-j}+1,j}_n h^{\psi(r,j,t,2),t}_{n+1} \right)  I_{2^j} \\
\end{split}
\]

This expression can be simplified further by noting the following: since we are assuming that $j<n$ the terms $c^{2^{n+1-j}-2^{t+1-j}+\phi(r,j,t,2)2^{-j}+1,j}_n$ appearing above can be replaced with $c^{k,j}_m$ for some $k,m$ with $m<n$. In particular, the original inductive hypothesis implies that if $r > 2^{n-j}$ then $C^{r,j}_n = C^{r-2^{n-j},j}_{n-1}$ and hence that $c^{r,j}_n = c^{r-2^{n-j},j}_{n-1}$. Therefore, we conclude that we may replace each instance of $c^{2^{n+1-j}-2^{t+1-j}+\phi(r,j,t,2)2^{-j}+1,j}_n$ appearing above with $c^{\phi(r,j,t,2)2^{-j}+1,j}_t$. Then we may express $C^{r,j}_{n+1}$ as

\[
C^{r,j}_{n+1} = dH^{r,j}_{n+1} - \left( \alpha_{\psi(r,j,n,2)-1}h_n^{r-(\psi(r,j,n,2)-1)2^{n-j},j}+ \sum_{t=j+1}^n c^{\phi(r,j,t,2)2^{-j}+1,j}_t h^{\psi(r,j,t,2),t}_{n+1} \right)I_{2^n}
\]

The bracketed expression is a $1$-form on the curve $C$, and so we can choose $h^{r,j}_{n+1}\in K(C)$ polynomial in $h^{2,0}_1=F$ such that \index{$c^{r,i}_n$}\index{$h^{r,i}_n$}

\begin{equation}\label{Equation - Proof of Log Ext}
dh^{r,j}_{n+1}- \alpha_{\psi(r,j,n,2)-1}h_n^{r-(\psi(r,j,n,2)-1)2^{n-j},i}- \sum_{t=j+1}^n c^{\phi(r,j,t,2)2^{-j}+1,i}_t h^{\psi(r,j,t,2),t}_{n+1}
\end{equation} has at worst logarithmic poles at $\PP$ and which has no other poles on $Y$. Take $H^{r,j}_{n+1}:= h^{r,j}_{n+1}I_{2^j}$. Repeating this for $r=1,...,2^{n+1-j}$ we see that condition \textit{(A)} is satisfied for $H_{n+1}$ for $i=j$. Hence, by induction condition \textit{(A)} is satisfied by $H_{n+1}$. The construction also makes it clear that condition \textit{(B)} is satisfied by $C'_{n+1}$ with $c^{r,i}_{n+1}:= (\ref{Equation - Proof of Log Ext})$. 

Now let $\U'_{n+1}$ be the logarithmic connection on $Y$ with bundle $R_{n+1} \otimes \O_Y$ and connection $d+C'_{n+1}$. By induction we conclude that for each $n$ there is an isomorphism $G_n$ from $\U_n$ to $\U'_{n}$ over $Y_{\infty}$ such that the connection matrix of $\U'_n$ is logarithmic on $Y$ along $\PP$. For each $n$ the logarithmic extension $\UU_{n}$ of $\U_n$ is described, therefore, by the descent datum of the logarithmic connection $\U_n$ over $X$, the logarithmic connection $\U'_n$ over $Y$ and the isomorphism $G_n : \U_n |_{X \cap Y} \xrightarrow{\sim} \U'_n|_{X \cap Y}$. 
  \end{proof}

The following lemma shows how we may simplify the computations involved in application of the above theorem. 

\begin{lemma}\label{Lemma - Computing GT on UC of EC}
Let $C,D$ and $X$ be as above. Let $n\geq 1$ and let $\U_n$ be the $n$-th level finite quotient of the universal connection $\U$ on $X$. Let $\U'_n$ be the connection and $G_n$ the isomorphism defined in Theorem \ref{Theorem - Compute GT on UC of EC}. Then the functions $h^{r,i}_n$ may be chosen such that for $i \geq 0$ we have $h^{r,i+1}_{n+1}=h^{r,i}_n$ and $c^{r,i+1}_{n+1}=c^{r,i}_n$. 
\end{lemma}

\begin{proof}
This is a simple induction argument on noting that $\psi(r,i,j,2) = \psi(r,i-1,j-1,2)$ and $\phi(r,i,j,2)=2\phi(r,i-1,j-1,2)$ and using condition (\ref{Equation - Proof of Log Ext}) in the proof of Theorem \ref{Theorem - Compute GT on UC of EC}. 
  \end{proof}

\begin{remark}
Theorem \ref{Theorem - Compute GT on UC of EC} demonstrates that we may iteratively compute suitable logarithmic extensions of the universal connection $\U$ on $X$ by computing compatible logarithmic extensions of the finite level quotients $\U_n$. Lemma \ref{Lemma - Computing GT on UC of EC} simplifies the computations involved in computing the connection matrices and gauge transformations. Assume we have computed the extension $\UU_n$ in the form of the datum $\U_n,\U'_n$ and $G_n$. Then we simply need to determine what the image of $1$ at level $n+1$ under the gauge transformation $G_{n+1}$ should be. By Lemma \ref{Lemma - Computing GT on UC of EC} this immediately determines a suitable gauge transformation. That is, at level $n+1$ for each $r=1,..,2^n$  we need to compute a $h^{r,0}_{n+1} \in K(C)$ (polynomial in $F=h^{2,0}_1$) such that 

\begin{equation} \label{easy gauge computation ec}
dh^{r,0}_{n+1}- \alpha_{\psi(r,0,n,2)-1}h_n^{r-(\psi(r,0,n,2)-1)2^n,0}- \sum_{t=1}^n c^{\phi(r,0,t,2)+1,0}_t h^{\psi(r,0,t,2),t}_{n+1} 
\end{equation}

has at worst logarithmic poles at $\PP$. By expanding (\ref{easy gauge computation ec}) in a local parameter $t$ at $\PP$, we can compute $h^{r,0}_{n+1}$ locally as the formal integral of \[\sum_{t=1}^n c^{\phi(r,0,t,2)+1,0}_t h^{\psi(r,0,t,2),t}_{n+1} + \alpha_{\psi(r,0,n,2)-1}h_n^{r-(\psi(r,0,n,2)-1)2^n,0}. \] In the following algorithm we compute the gauge transformations iteratively. 
\end{remark}

\begin{algorithm} \label{Algorithm - Computing GT on UC of EC}
(Computing the logarithmic extension of $\U_n$ on elliptic curves) \newline 

\noindent Input:
\begin{itemize}
\item Elliptic curve $C$ over a characteristic $0$ field $K$ with affine model $X$ of the form $y^2=f(x)$ with $f(x) \in K[x]$ and $\deg f =3$.
\item Universal connection $\U= \lbrace (\U_n,1) \rbrace$ on $X$ with respect to the basis $\alpha_0,\alpha_1$ of $H^1_{dr}(X)$ as at the beginning of Section \ref{Subsection - Logarithmic Extensions on EC}.
\item The connection matrix $C_n'$ over open $Y \subset C$ of the logarithmic extension $\UU_n$ of $\U_n$ .
\item The gauge transformation $G_n$ defining the extension $\UU_n$ with respect to the basis $\B_n$ .
\end{itemize} 
Output:

\begin{itemize}
\item The connection matrix $C_{n+1}'$ over open $Y \subset C$ of the logarithmic extension $\UU_{n+1}$ of $\U_{n+1}$.
\item The gauge transformation $G_{n+1}$ defining the extension $\UU_{n+1}$ with respect to the basis $\B_{n+1}$.
\end{itemize}

\noindent Algorithm:

\begin{enumerate}
\item[I] If $n=0$ then 

\begin{enumerate}
\item[(1)] Take $h^{1,0}_1 \in K(C)$ regular at $\PP$.
\item[(2)] Take $h^{2,0}_1 \in K(C)$ with a single simple pole at $\PP$ such that $dh^{2,0}_1 - \alpha_1$ has at worst a logarithmic pole at $\PP$.
\item[(3)] Define 
\begin{align*}
& G_1 = \left( \begin{array}{ccc}
			1 & 0 & h^{1,0}_1 \\
			0 & 1 & h^{2,0}_1 \\
			0 & 0 & 1 \\
		\end{array} \right) \\ 
& C'_1= \left( 	\begin{array}{ccc}
			0 & 0 & -\alpha_0 +dh^{1,0}_1 \\
			0 & 0 & -\alpha_1+dh^{2,0}_1 \\
			0 & 0 & 0 \\
		\end{array} \right)=\left( 	\begin{array}{ccc}
			0 & 0 & c^{1,0}_1 \\
			0 & 0 & c^{2,0}_1 \\
			0 & 0 & 0 \\
		\end{array} \right).
\end{align*}
\end{enumerate}
\item[II] Else
\begin{enumerate}
\item[(1)] For $0 <i \leq n$, $1 \leq r \leq 2^{n+1-i}$ define: 
\begin{enumerate}
\item $h^{r,i}_{n+1} := h^{r,i-1}_n$
\item $c^{r,i}_{n+1} := c^{r,i-1}_n$
\end{enumerate}
\item[(2)] For $i=0$, $1 \leq r \leq 2^{n+1}$ take $h^{r,0}_{n+1} \in K(C)$ polynomial in $h^{2,0}_1$ such that 

\begin{equation} \label{algorithm - differential ec}
dh^{r,0}_{n+1}- \alpha_{\psi(r,0,n,2)-1}h_n^{r-(\psi(r,0,n,2)-1)2^n,0}- \sum_{t=1}^n c^{\phi(r,0,t,2)+1,0}_t h^{\psi(r,0,t,2),t}_{n+1} 
\end{equation}

has at worst logarithmic poles at $\PP$. Define $c^{r,0}_{n+1}:=$ (\ref{algorithm - differential ec}).  
\item[(3)] For $0 \leq i \leq n$, $1 \leq r \leq 2^{n+1-i}$ let $H^{r,i}_{n+1}:= h^{r,i}_{n+1}I_{2^i}$. Define $H_{n+1}$ to be the matrix  

\[H_{n+1} = \left( \begin{array}{ccccc}
						H^{1,n}_{n+1}& \hdots & H^{1,i}_{n+1} & \hdots & H^{1,0}_n \\
						\vdots & & \vdots & & \vdots \\
						H^{2,n}_{n+1}& \hdots & H^{2^{n+1-i},i}_{n+1} & \hdots & H^{2^{n+1},0}_{n+1} \end{array} \right)\]and define 

\[
G_{n+1} :=  \left( \begin{array}{cc}
						I & H_{n+1} \\
						0  & G_{n} \end{array} \right).
\]

\item[(4)] For $0 \leq i \leq n$, $1 \leq r \leq 2^{n+1-i}$ let $C^{r,i}_{n+1}:= c^{r,i}_{n+1}I_{2^i}$. Define $D'_{n+1}$ to be the matrix  

\[D'_{n+1} = \left( \begin{array}{ccccc}
			C^{1,n}_{n+1}& \hdots & C^{1,i}_{n+1} & \hdots & C^{1,0}_{n+1} \\
			\vdots & & \vdots & & \vdots \\
			C^{2,n}_{n+1} & \hdots & C^{2^{n+1-i},i}_{n+1} & \hdots& C^{2^{n+1},0}_{n+1} \end{array} \right)\] and define

\[
C'_{n+1} :=  \left( \begin{array}{cc}
						0 & D'_{n+1} \\
						0  & C'_{n} \end{array} \right)
\] 
\end{enumerate}
\end{enumerate}
\end{algorithm}

We end this section with an application of Algorithm \ref{Algorithm - Computing GT on UC of EC} to successively compute the extension of $\U_n$ on $X$ to $\UU_n$ on $C$ for $n=2,3$ and $4$. For brevity, we describe only the gauge transformations $G_n$ computed by the algorithm and in level $4$ we only describe the image of $1$ under the gauge transformation $G_4$. However, using the above this is clearly sufficient to describe the extensions. 
 
\begin{proposition} \label{Proposition - GT for EC Level 2}
The gauge transformation $G_2$ computed by Algorithm \ref{Algorithm - Computing GT on UC of EC} and extending $G_1$ as computed in Example \ref{Example - Level 1 Extension EC} is
\[
G_2 = \left( 	\begin{array}{cc}
					I& H_2 \\
					0 & G_1 
				\end{array}\right), \; \; H_2 = \left( 	\begin{array}{ccc}
					0 & 0 & 0 \\
					0 & 0 & 0 \\
					F & 0 & 0 \\
					0 & F & \frac{1}{2}F^2 
				\end{array} \right). 
\]
\end{proposition}
\begin{proof}
In Example \ref{Example - Level 1 Extension EC} we computed the extension of $\U_1$ to $\UU_1$ and saw that we may define

\[
G_1 = \left( \begin{array}{ccc}
				1 & 0 & 0 \\
				0 & 1 & F\\
				0 & 0 & 1
				\end{array} \right)
\] for some $F \in K(C)$ with a simple pole at $\PP$. In the notation of Theorem \ref{Theorem - Compute GT on UC of EC} we have
\begin{align*}
& h^{1,0}_1=0, \; \; h^{2,0}_1=F \\  &c^{1,0}_1=-\alpha_0, \; \; c^{2,0}_1=-\alpha_1+dF. \\
\end{align*} Applying Algorithm \ref{Algorithm - Computing GT on UC of EC} with $n=1$ then as in Step II.1) we define \begin{align*}
& h^{1,1}_2=0, \; \; h^{2,1}_2=F \\  &c^{1,1}_2=-\alpha_0, \; \; c^{2,1}_2=-\alpha_1+dF. \\
\end{align*} Following Step (II)(2) for $r=1,2,3,4$  we need to compute $h^{r,0}_2$ polynomial in $F$  such that

\[
c^{r,0}_2=dh^{r,0}_{2}- \alpha_{\psi(r,0,1,2)-1}h_1^{r-(\psi(r,0,1,2)-1)2,0}- c^{\phi(r,0,1,2)+1,0}_1 h^{\psi(r,0,1,2),1}_{2} 
\] has at worst logarithmic poles at $\PP$. Here we find that  

\[
\begin{split}
r=1 \: &: \: c^{1,0}_2 = dh^{1,0}_2 \\
r=2 \: &: \: c^{2,0}_2 = dh^{2,0}_2 - F\alpha_0\\
r=3 \: &: \: c^{3,0}_2 = dh^{3,0}_2 + F\alpha_0\\
r=4 \: &: \: c^{4,0}_2 = dh^{4,0}_2 - F\alpha_1 -(-\alpha_1+dF)F = dh^{4,0}_2 - FdF.\\
\end{split}
\]

Since $\alpha_0$ is regular at $\PP$ and $F$ has a simple pole there we can take $h^{1,0}_2=h^{2,0}_2=h^{3,0}_2=0$ and also take $h^{4,0}_2= \frac{1}{2}F^2$. Therefore, we find that we can take the matrix $G_2$ to be 

\[
G_2 = \left( 	\begin{array}{cc}
					I& H_2 \\
					0 & G_1 
				\end{array}\right), \; \; H_2 = \left( 	\begin{array}{ccc}
					0 & 0 & 0 \\
					0 & 0 & 0 \\
					F & 0 & 0 \\
					0 & F & \frac{1}{2}F^2 
				\end{array} \right) 
\]

and then $C_2'$, the gauge transformation of $C_2$ by $G_2$ is 

\[
C_2' = \left( 	\begin{array}{cc}
					0& D'_2 \\
					0& C'_1
				\end{array}\right), \; \; D'_2 = \left( 	\begin{array}{ccc}
					-\alpha_0 & 0 & 0 \\
					0 & -\alpha_0 & -F \alpha_0 \\
					\alpha_1' & 0 & F \alpha_0 \\
					0 & \alpha_1' & 0					
				\end{array} \right) 
\]

where $\alpha_1'=-\alpha_1+dF$. 
\end{proof}

\begin{proposition} \label{Proposition - GT on UC of EC Level 3}
The gauge transformation $G_3$ computed by Algorithm \ref{Algorithm - Computing GT on UC of EC} and extending $G_2$ as computed in Proposition  \ref{Proposition - GT for EC Level 2} is

\[
G_3=\left( 	\begin{array}{cc}
					I_{8 \times 8} & H_4 \\
					0_{7 \times 8} & G_2 
				\end{array}\right), \; \;
H_3 = \left( \begin{array}{ccccccc}
				0 & 0 & 0 & 0 & 0 & 0 & 0 \\
				0 & 0 & 0 & 0 & 0 & 0 & 0 \\
				0 & 0 & 0 & 0 & 0 & 0 & 0 \\
				0 & 0 & 0 & 0 & 0 & 0 & \lambda F \\
								F & 0 & 0 & 0 & 0 & 0 & 0 \\
				0 & F & 0 & 0 & 0 & 0 & -2 \lambda F \\
				0 & 0 & F & 0 & \frac{1}{2}F^2 & 0 & \lambda F \\
				0 & 0 & 0 & F & 0 & \frac{1}{2}F^2 & \frac{1}{6}F^3 \end{array} \right)
\]
where $\lambda \in K$ is such that $\lambda dF-\frac{1}{2}F^2 \alpha_0$ has a logarithmic pole at $\PP$. 
\end{proposition}

\begin{proof}
This is a straightforward application of Algorithm \ref{Algorithm - Computing GT on UC of EC} making use of the extension $\UU_2$ computed in Proposition \ref{Proposition - GT for EC Level 2}.
\end{proof}

\begin{proposition} \label{Proposition - GT on UC of EC Level 4}
The gauge transformation $G_4$ computed by Algorithm \ref{Algorithm - Computing GT on UC of EC} and extending $G_2$ as computed in Proposition  \ref{Proposition - GT on UC of EC Level 3} is such that

\begin{align*}
G_4(1) =& (\frac{1}{6}\lambda F^2 + \mu F)A_0A_1^3 + (\frac{1}{2}\lambda F^2 -3 \mu F)A_1A_0A_1^2 +( 3\mu F-\frac{3}{2}\lambda F^2)A_1^2A_0A_1\\& +(\frac{5}{6}\lambda F^2 - \mu F)A_1^3A_0 + \frac{1}{24}F^4 A_1^4 + \text{words of length } \leq 3
\end{align*}
where $\mu \in K$ is such that $\frac{1}{3}\lambda F dF + \mu dF - \frac{1}{6}F^3 \alpha_0$ has at worst a simple pole at $\PP$. 
\end{proposition}

\begin{proof}
Again this is simply an application of Algorithm \ref{Algorithm - Computing GT on UC of EC} lifting the extension computed in Proposition \ref{Proposition - GT on UC of EC Level 3}. 
  \end{proof}

\begin{remark}
We may question whether the constants $\lambda$ and $\mu$ have any dependence on the choice of $F$. In fact they depend only on the choice of differentials $\alpha_0,\alpha_1$ as can be seen by expanding the differentials and $F$ in a local parameter at $\PP$. It is a simple calculation then to verify the independence from $F$. 
\end{remark}

\subsection{Logarithmic extensions on affine hyperelliptic curves} \label{Subsection - Logarithmic Extensions on HEC}

Let $C$ be an odd hyperelliptic curve of genus $g \geq 2$ over a field $K$ of characteristic $0$. Say that we have an affine model of $C$ of the form 

\[
y^2 = f(x), \; f(x) \in K[x]
\] with $\deg(f) =2g+1$. There is a single $K$-rational point $\PP$ at infinity. Let $X$ be the affine curve $y^2 = f(x)$ over $K$ so that $X=C-\lbrace \PP \rbrace$. Then $H^1_{dr}(X/K)$ has a $K$-basis of size $2g$. Specialise Definition \ref{Definition - UC on General Affine Curve} to $X$, taking $\alpha_0,\alpha_1, ... \alpha_{2g-1} \in H^0(X, \Omega^1_X)$ be 1-forms on $X$ such that their cohomology classes form a $K$-basis for $H^1_{dr}(X/K)$. Since $C$ is of genus $g$, we may further assume that $\alpha_0,.., \alpha_{g-1}$ is a $K$-basis of $H^0(C, \Omega^1_C)$. 

As in Section \ref{Subsection - Logarithmic Extensions on EC} we take as a basis \index{$\B_n$}$\B_n$ for $R_n$ the one given by the graded lexicographic order such that $A_0<A_1<..<A_{2g-1}<1$. If the $k$-th word of length $l$ is \index{$w^k_l$}$w^k_l$ with respect to this basis then we easily conclude that $A_i w^k_l = w^{(2g)^li+k}_{l+1}$. So we can describe the action of $\nabla$ on a basis for $R_n$:

\[
\nabla( w^k_l) =	\begin{cases} 
						-\sum_{i=0}^{2g-1}  A_i w^k_l \alpha_i = -\sum_{i=0}^{2g-1}  w^{(2g)^li+k}_{l+1}\alpha_i & \text{ if } l \leq n-1 \\ 
						0 & \text{ if } l=n 
					\end{cases}
\] Given this it is simple to prove the following lemma.

\begin{lemma}
The connection matrix of $\U_0$ is the zero matrix. If $C_n$ is the connection matrix of $\U_n$ with respect to the basis $\B_n$, then  

\[
C_{n+1}=\left(	\begin{array}{cc}
					0_{(2g)^{n+1} \times (2g)^{n+1}} & D_{n+1} \\
					0_{\frac{(2g)^{n+1}-1}{2g-1} \times (2g)^{n+1}} & C_n 
				\end{array}\right)
\]

is the connection matrix of $\U_{n+1}$ with respect to the basis $\B_{n+1}$ where 

\[
D_{n+1} = 	\left(	\begin{array}{cc}
					-\alpha_0I_{(2g)^n} & 0 \\
					\vdots & \vdots \\
					-\alpha_{2g-1}I_{(2g)^n} & 0
				\end{array}\right).
\]

\end{lemma}

It is straightforward to prove the analogous versions of Theorem \ref{Theorem - Compute GT on UC of EC} and Lemma \ref{Lemma - Computing GT on UC of EC} for odd models of hyperelliptic curves. Hence we may iteratively compute the logarithmic extensions of $\U_n$ using the following algorithm:

\begin{algorithm} \label{Algorithm - Computing GT on UC of HEC}
(Computing the logarithmic extension of $\U_n$ on hyperelliptic curves) \newline 

\noindent Input:
\begin{itemize}
\item Hyperelliptic curve $C$ over a characteristic $0$ field $K$ of genus $g$ with affine model $X$ of the form $y^2=f(x)$ with $f(x) \in K[x]$ and $\deg f =2g+1$, $F \in K(C)$ fixed with a pole of order $1$ at $\PP$. 
\item Universal connection $\U= \lbrace (\U_n,1) \rbrace$ on $X$ with respect to the basis $\alpha_0,..,\alpha_{2g-1}$ of $H^1_{dr}(X)$ as above
\item The connection matrix $C_n'$ over open $Y \subset C$ of the logarithmic extension $\UU_n$ of $\U_n$ .
\item The gauge transformation $G_n$ defining the extension $\UU_n$ with respect to the basis $\B_n$ .
\end{itemize} 
Output:

\begin{itemize}
\item The connection matrix $C_{n+1}'$ over open $Y \subset C$ of the logarithmic extension $\UU_{n+1}$ of $\U_{n+1}$.
\item The gauge transformation $G_{n+1}$ defining the extension $\UU_{n+1}$ with respect to the basis $\B_{n+1}$.
\end{itemize}

\noindent Algorithm:

\begin{enumerate}
\item[I] If $n=0$ then 

\begin{enumerate}
\item[(1)] For $r=1,..,2g$ take $h^{r,0}_1 \in K(C)$ polynomial in $F$ such that $dh^{r,0}_1 - \alpha_{r-1}$ has at worst a logarithmic pole at $\PP$.
\item[(2)] For $r=1,..,2g$ define $c^{r,0}_1:= dh^{r,0}_1 - \alpha_{r-1}$
\item[(3)] Define 
\begin{align*}
& G_1 = \left( \begin{array}{cc}
			I_{2g} & H_1 \\
			0 & 1 \end{array} \right), \; \; H_{1} = 	\left(	\begin{array}{c}
					dh^{1,0}_1 \\
					\vdots \\
					dh^{r,0}_1
				\end{array}\right) \\ 
&
C'_1=\left(	\begin{array}{cc}
					0 & D_{1} \\
					0 & 0 
				\end{array}\right), \; \; D_{1} = 	\left(	\begin{array}{c}
					c^{1,0}_1  \\
					\vdots \\
					c^{2g,0}_1
				\end{array}\right). \\
\end{align*}

\end{enumerate}
\item[II] Else
\begin{enumerate}
\item[(1)] For $0 <i \leq n$, $1 \leq r \leq (2g)^{n+1-i}$ define: 
\begin{enumerate}
\item $h^{r,i}_{n+1} := h^{r,i-1}_n$
\item $c^{r,i}_{n+1} := c^{r,i-1}_n$
\end{enumerate}
\item[(2)] For $i=0$, $1 \leq r \leq (2g)^{n+1}$ take $h^{r,0}_{n+1} \in K(C)$ polynomial in $F$ such that 

\begin{equation} \label{algorithm - differential hec}
dh^{r,0}_{n+1}- \alpha_{\psi(r,0,n,2g)-1}h_n^{r-(\psi(r,0,n,2g)-1)(2g)^n,0}- \sum_{t=1}^n c^{\phi(r,0,t,2g)+1,0}_t h^{\psi(r,0,t,2g),t}_{n+1} 
\end{equation}

has at worst logarithmic poles at $\PP$. Define $c^{r,0}_{n+1}:= (\ref{algorithm - differential hec})$.  
\item[(3)] For $0 \leq i \leq n$, $1 \leq r \leq (2g)^{n+1-i}$ let $H^{r,i}_{n+1}:= h^{r,i}_{n+1}I_{(2g)^i}$. Define $H_{n+1}$ to be the matrix  

\[H_{n+1} = \left( \begin{array}{ccccc}
						H^{1,n}_{n+1}& \hdots & H^{1,i}_{n+1} & \hdots & H^{1,0}_n \\
						\vdots & & \vdots & & \vdots \\
						H^{2g,n}_{n+1}& \hdots & H^{(2g)^{n+1-i},i}_{n+1} & \hdots & H^{(2g)^{n+1},0}_{n+1} \end{array} \right)\]and define 

\[
G_{n+1} :=  \left( \begin{array}{cc}
						I & H_{n+1} \\
						0  & G_{n} \end{array} \right).
\]

\item[(4)] For $0 \leq i \leq n$, $1 \leq r \leq (2g)^{n+1-i}$ let $C^{r,i}_{n+1}:= c^{r,i}_{n+1}I_{(2g)^i}$. Define $D'_{n+1}$ to be the matrix  

\[D'_{n+1} = \left( \begin{array}{ccccc}
			C^{1,n}_{n+1}& \hdots & C^{1,i}_{n+1} & \hdots & C^{1,0}_{n+1} \\
			\vdots & & \vdots & & \vdots \\
			C^{2g,n}_{n+1} & \hdots & C^{(2g)^{n+1-i},i}_{n+1} & \hdots& C^{(2g)^{n+1},0}_{n+1} \end{array} \right)\] and define

\[
C'_{n+1} :=  \left( \begin{array}{cc}
						0 & D'_{n+1} \\
						0  & C'_{n} \end{array} \right).
\] 
\end{enumerate}
\end{enumerate}
\end{algorithm}

\begin{remark}
Suppose that we allow $g=1$ in the above algorithm. Then we recover the steps found in Algorithm \ref{Algorithm - Computing GT on UC of EC} and we may take Algorithm \ref{Algorithm - Computing GT on UC of HEC} as a general algorithm to compute universal logarithmic extensions of $\U_n$ for elliptic curves or odd hyperelliptic curves. 
\end{remark}
We conclude this section with one final result which will be of use to us in the next section. This will give us an explicit description of how the gauge transformation $G_n$ acts on a basis $\B_n$ for $R_n$, the words in $A_0,..,A_{2g-1}$ of length at most $n$. 

\begin{lemma}\label{Lemma - Action of G_n on B_n}
Let $\B_n$ be the basis of $R_n$ consisting of words in $A_0,..,A_{2g-1}$ of length at most $n$ ordered with the graded lexicographic ordering such that $A_i < A_j$ if $i<j$. Let $G_n$ be the gauge transformation computed by Algorithm \ref{Algorithm - Computing GT on UC of HEC}. Then $G_n$ acts on $w^k_l$ as follows:

\[
G_{n}: w^k_l \mapsto \begin{cases} 	w^k_l & \text{if } l=n \\
										w^k_l+\sum_{s=l+1}^{n} \sum_{t=1}^{(2g)^{s-l}} w_s^{k+(t-1)(2g)^l}h_s^{t,l} & \text{otherwise.}\end{cases}
\]
\end{lemma}

\begin{proof}
Suppose first that $l=n$. Then by construction we have that $G_n(w^k_n)=w^k_n$ for all $k$. Otherwise, suppose that $l <n$. Using Algorithm \ref{Algorithm - Computing GT on UC of HEC} $G_n$ will be of the form 
\[G_{n}=\left( \begin{array}{cc}
						I & H_{n} \\
						0  & G_{n-1} \end{array} \right), \quad H_{n} = \left( \begin{array}{ccccc}
						h^{1,n-1}_{n}I_{(2g)^{n-1}}& \hdots & h^{1,i}_{n}I_{(2g)^i} & \hdots & h^{1,0}_{n}I_{1} \\
						\vdots & & \vdots & & \vdots \\
						h^{2g,n-1}_{n}I_{(2g)^{n-1}}& \hdots & h^{(2g)^{n-i},i}_{n}I_{(2g)^i} & \hdots & h^{(2g)^{n},i}_{n}I_1 \end{array} \right).\]
The word $w^k_l$ will correspond to column $(2g)^n+(2g)^{n-1}+...+(2g)^{l+1}+k$ in the matrix $G_n$. The submatrix $H_n$ will, therefore, contribute 

\[h^{1,l}_nw^k_n+ h^{2,l}_nw^{k+(2g)^l}_n+..+h^{(2g)^{n-l},l}_nw^{k+((2g)^{n-l}-1)(2g)^l,l}_n.
\]

By a simple induction argument the lemma then follows. 		
  \end{proof}

\section{The Hodge filtration on $\U$} \label{Section 4}

\subsection{Computation of the Hodge filtration}

In this section we shall utilise the computations of Section \ref{Section 3} to provide an explicit version of results due to Hadian in \cite{hadian11}. Hadian provides a characterisation of the Hodge filtration on the universal connection $\U$ of $X:=C-D$, where $C$ is a general smooth projective curve of genus $g$ over a field $K$ of characteristic $0$ and $D$ is a non-empty divisor defined over $K$ of size $r$. Fix a basepoint $b \in X(K)$. 

\begin{definition}
By a \textit{filtered logarithmic connection} \index{$(\mathcal{V}, \nabla, F^{\bullet})$}$\mathcal{V}:= (\mathcal{V}, \nabla, F^{\bullet})$ on $C$ with log poles along $D$ we mean a vector bundle $\mathcal{V}$ with a logarithmic connection $\nabla$ with log poles along $D$ which is equipped with a decreasing filtration  by sub-bundles

\[
\mathcal{V}=F^m \mathcal{V} \subset F^{m+1} \mathcal{V} \subset... \subset F^n \mathcal{V}=0
\] for some $m<n \in \Z$ satisfying the \textit{Griffiths transversality} property:

\[
\nabla (F^i \mathcal{V}) \subset F^{i-1}\mathcal{V} \otimes \Omega^1_{C/K}(D)
\] for all $i$.  
\end{definition}

\begin{remark} \label{Remark - Hodge filtration on trivial log connection}
Note that the trivial connection $\O_C=(\O_C,d)$ is given the trivial filtration $F^0(\O_C)=\O_C$, $F^1(\O_C)=0$.
\end{remark}

\begin{definition}
Let $\mathcal{U},\mathcal{V},\mathcal{W}$ be filtered logarithmic connections. An exact sequence of logarithmic connections

\[
0 \rightarrow \mathcal{U}\rightarrow\mathcal{V}\rightarrow\mathcal{W}\rightarrow  0
\] is an exact sequence of filtered logarithmic connections if for each $p$ we have an exact sequence of sub-bundles

\[
0 \rightarrow F^p\mathcal{U}\rightarrow F^p \mathcal{V}\rightarrow F^p\mathcal{W}\rightarrow 0
\]

\end{definition}

The dual space $V_{dr}$ in Definition \ref{Definition - UC on General Affine Curve} has a Hodge filtration induced by the natural trivial filtration on the de Rham complex of $X/K$, and this in turn induces a Hodge filtration on $V_{dr}^{\otimes n}$. Recall that $V_{dr}$ has basis $A_0,..,A_{2g+r-2}$ dual to differentials $\alpha_0,..,\alpha_{2g+r-2}$. We assume that the $\alpha_i$ are ordered so that $\alpha_0,..,\alpha_{g-1}$ form a $K$-basis for $H^0(C,\Omega^1_{C/K})$. As with elliptic and odd hyperelliptic curves we take as a $K$-basis of $V^{\otimes n}_{dr}$ the set \index{$\B_n$}$\B_n$ of words of length $n$ in $A_0,..,A_{2g+r-2}$ with graded lexicographic ordering such that $A_i < A_j$ if $i <j$.

\begin{definition} \label{Definition - Hodge Filtration on Vdr}
The filtration $F^{\bullet}$ on $V_{dr}^{\otimes n}$ is defined as follows:

For $p >0$

\[
F^pV_{dr}^{\otimes n} := 0.
\]For $p<0$ we let $\tilde{F}^p:= \lbrace w \in \B_n\; : \;  w \text{ contains at most }|p| \text{ occurrences of } A_0,..,A_{g-1} \rbrace$. Then 

\[
F^pV_{dr}^{\otimes n}:= \text{Span}_K(\tilde{F}^p). 
\]
\end{definition}

The natural filtration of $\O_C$ given by $F^0\O_C=\O_C$ and $F^1\O_C=0$ then together induce a filtration on $V_{dr}^{\otimes n} \otimes \O_C$ via the tensor product filtration. We may now state Hadian's Lemma in its full generality for $C$. 

\begin{lemma}[\cite{hadian11}, Lemma 3.6] \label{Lemma - Hadian's Lemma}\index{$F^{\bullet}\UU_n$} Let $C$ be a smooth projective curve over a field $K$ of characteristic $0$, $D$ a non-empty divisor and $X:=C-D$ and take $b \in X(K)$ a rational basepoint. Let $V_{dr}:=H^1_{dr}(X)^{\vee}$ and let $\U_n$ be the $n$-th finite level quotient of the universal connection on $X$ with respect to the basepoint $b$. Let $\UU_n$ be the extension of this to a logarithmic connection on $C$. Then there exists a filtration $(F^{\bullet} \UU_n)$ of vector bundles such that 

\begin{enumerate}
\item[i)] For all $n$ the filtration $F^{\bullet}$ on $\UU_n$ satisfies Griffiths transversality giving $\UU_n$ the structure of a filtered logarithmic connection. This filtration is unique up to automorphism of filtered logarithmic connections. 
\item[ii)] For all $n$ the exact sequence of logarithmic connections 

\[
0 \rightarrow V_{dr}^{\otimes n} \otimes \O_C \rightarrow \UU_n \rightarrow \UU_{n-1} \rightarrow 0
\] becomes an exact sequence of filtered logarithmic connections, where $V_{dr}^{\otimes n} \otimes \O_C$ has the Hodge filtration induced by the filtration on $V_{dr}^{\otimes n}$.

\item[iii)] The distinguished element $1 \in b^*\U_n$  belongs to the fibre $b^*F^0\UU_n$. 
\end{enumerate}
\end{lemma}

We now describe the iterative method of computing the Hodge filtration at level $n$. This method is based on the application presented by Dogra in \cite[\S 4]{dogra15}. The idea is as follows: we want to compute sub-bundles $F^{\bullet} \U_n$ and $F^{\bullet} \U_n^i$ satisfying Griffiths transversality such that $(F^{\bullet} \U_n,F^{\bullet} \U_n^i, G_n^i)$ are the descent datum of sub-bundles $F^{\bullet} \UU_n$ of $\UU_n$ on $C$ satisfying the conditions of Lemma \ref{Lemma - Hadian's Lemma}. Note that it is clear by a simple induction that we must have $F^p\UU_n=0$ for $p>0$ using Definition \ref{Definition - Hodge Filtration on Vdr}, Remark \ref{Remark - Hodge filtration on trivial log connection} and Algorithm \ref{Algorithm - Compute HF on general UC}. 

The computation of this is contained in the following algorithm.

\begin{algorithm} \label{Algorithm - Compute HF on general UC}
(Computing Hodge filtration on logarithmic universal connection of $C$ with poles along $D$) \newline

\textbf{Input}:

\begin{itemize}
\item Smooth projective curve $C$ over field $K$ of characteristic $0$, a non-empty divisor $D$ defined over $K$, $X:=C-D$, and a basepoint $b \in X(K)$.
\item The logarithmic extension $\UU_n$ of $\U_n$ on $X$ computed by Algorithm \ref{Algorithm - Computing GT on UC}.
\item The Hodge filtration $F^{\bullet}\UU_{n}$ on $\UU_n$.
\end{itemize}

\textbf{Output}: 

\begin{itemize}
\item The Hodge filtration $F^{\bullet}\UU_{n+1}$ on $\UU_n$. 
\end{itemize}

\textbf{Algorithm}:

\begin{enumerate}
\item[(1)] For $p > 0$ put $F^p \UU_{n+1} =0$.

\item[(2)] For $p \leq 0$ do:

\begin{enumerate}
\item Lift generators of $F^p\U_{n}$ arbitrarily to sections of $\U_{n+1}$ over $X$ and adjoin the generators of $F^pV_{dr}^{\otimes (n+1)} \otimes \O_C$.
\item For all $i$ lift generators of $F^p\U^i_{n}$ to sections of $\U^i_{n+1}$ over $Y_i$ and adjoin the generators of $F^pV_{dr}^{\otimes (n+1)} \otimes \O_C$.
\item For all $i$:
\begin{enumerate}
\item[(i)] Compute the images under $G^i_n$ of the restrictions to $X \cap Y_i$ of the lifts in Step (2)(a).  
\item[(ii)] Express the images computed in the previous step using restrictions to $X \cap Y_i$ of the lifts in Step (2)(b). 
\item[(iii)] Use the previous step to determine explicit conditions on the lifts from Steps(2)(a) and (b). 
\end{enumerate}

\end{enumerate}

\item[(3)] For $p=0$ compute generators of $F^0\UU_{n+1}$ such that 
\begin{enumerate}
\item The conditions Step (2)(c)(iii) are satisfied.
\item $1 \in b^*F^0\UU_n$.
\end{enumerate} 

\item[(4)] For $p <0$:

\begin{enumerate}
\item For each $i$ and generator of $F^{p+1} \U^i_{n+1}$ compute the image under $\nabla^i$ the connection on $\U_n^i$. 
\item Compute generators of $F^p\UU_{n+1}$ such that:
\begin{enumerate}
\item[(i)] The conditions of Step (2)(c)(iii) are satisfied. 
\item[(ii)] $F^p\U^i_{n+1} \otimes \Omega^1_{C/K}(D)$ contains the images computed in Step (4)(a).
\end{enumerate}
\end{enumerate} 
\end{enumerate}

\end{algorithm}

\begin{remark}
Again this algorithm is completely general and it is not immediately obvious that the Algorithm terminates successfully. The obstruction, if it exists, lies in Steps (3) and (4) - it is not clear how easy it would be to compute lift satisfying the conditions stated. However, if we are able to do so then Lemma \ref{Lemma - Hadian's Lemma} (Hadian's Lemma) ensures that this algorithm will compute the unique Hodge filtration on $\UU_n$. The remainder of this section will be occupied with showing that we can in fact carry out the steps of this algorithm for elliptic curves and odd hyperelliptic curves. 
\end{remark}

\begin{remark}
Note that it is easy to see that this algorithm will terminate with the computation of $F^{-n} \UU_n$. Observe that $F^{-n} V_{dr}^{\otimes n} = V_{dr}^{\otimes n}$ and then a simple induction argument shows that $F^{-n} \UU_n = \UU_n$. 
\end{remark}

\begin{example} \label{Example - Level 1}
In this example we compute the Hodge filtration on $\U_1$ for $X:=C-D$ as in Definition \ref{Definition - UC on General Affine Curve} where $C$ is a smooth projective curve of genus $g$ over $K$  and $D=\lbrace d_1,..,d_r \rbrace$ is a divisor defined over $K$. Then $\U_1 =R_1 \otimes \O_X$ with $R_1 = \langle A_0,...,A_{2g+r-2} \rangle_K$ where the $A_k$ are dual to differentials $\alpha_k$. As per Definition \ref{Definition - Hodge Filtration on Vdr} note that $F^0V_{dr} $ has $K$-basis $\lbrace A_g,...,A_{2g+r-2}\rbrace$. 

Consider the logarithmic extension of $\U_1$ to a logarithmic connection $\UU_1$ on $C$ with log poles along $D$ as computed by Algorithm \ref{Algorithm - Computing GT on UC}. This is given by descent datum $(\U_1^i, G_1^i)_i$ over an open cover $(Y^i)_i$ of $C$, where $Y^0=X$. Then we will find that over $X \cap Y^i$

\[
G_1^i(1) = 1+ \sum_k h^i_k A_k, \; G_1^i(A_l) = A_l \text{ for all } l
\] for some $h^i_k \in K(C)$ such that each $h^i_k$ is regular on $Y^i - D$ and is such that $dh^i_k - \alpha_{k}$ has logarithmic poles along $D \cap Y_i$. Proceeding as in Algorithm \ref{Algorithm - Compute HF on general UC} we lift the generator $1$ of $F^0\UU_0$ over each $Y^i$ and adjoin the generators of $F^0V_dr \otimes \O_C$, obtaining 
\begin{equation}\label{Equation - Level 1 Lifts}
1 + \sum_{k < g} a^i_k A_k,\;  A_g, ..., A_{2g+r-2}
\end{equation}
for some $a^i_k \in H^0(Y^i,\O_{Y^i})$. On applying $G^1_i$ to the lifts (\ref{Equation - Level 1 Lifts}) over $X$ we find that the generators of $F^0\U^i_1|_{X \cap Y^i}$ as an $\O_{X \cap Y^i}$-module are 
\[
1 + \sum_{k < g} (h^i_k+a^0_k) A_k,\; A_g, ..., A_{2g+r-2}.
\]

Restricting the lifts (\ref{Equation - Level 1 Lifts}) over $Y_i$ to $X \cap Y^i$ we conclude that 
\[
(h^i_k+a^0_k)|_{X \cap Y^i} = a^i_k|_{X \cap Y^i}.
\]

Since $\alpha_k$ is regular on $C$ for all $k<g$ then we conclude that each $h^i_k \in H^0(Y^i,\O_{Y^i})$. Therefore, the sections $a^0_k,a^i_k-h^i_k$ glue to give a global section in $H^0(C, \O_C) \simeq K$ and so are constant. Since we require that $1 \in b^*F^0\UU_1$ then we must have that $a^0_k(b)=0$ and hence that $a^0_k=0$ for all $k$. Therefore, the Hodge filtration at on $F^0 \UU_1$ is generated over $X$ by 
\[
1, A_g, ..., A_{2g+r-2}.
\]
\end{example}

\subsection{A constructive algorithm for affine elliptic and hyperelliptic curves}

In what follows we present a proof that the extensions $\UU$ we computed by Algorithm \ref{Algorithm - Computing GT on UC of HEC} in Section 3 have a unique Hodge filtration and, furthermore, we show that there are explicit conditions which uniquely determine the generators of the filtration at each level. Finally, we present an algorithm which may be used to compute the Hodge filtration iteratively. 

Before stating the main theorem of this section we first identify an ordered $K$-basis of $F^0V_{dr}^{\otimes n}$ with respect to the basis in words of length $n$ with the graded lexicographic ordering of Section \ref{Subsection - Logarithmic Extensions on HEC}. Recall that in Definition \ref{Definition - Hodge Filtration on Vdr} we identified this space as being the $K$-span of the set of words of length $n$ in $A_g,..,A_{2g-1}$.

\begin{lemma}\label{Lemma - Hodge Basis for Vdr}
Let \index{$\F_n$}$\F_n$ be the set 
\[
\left\lbrace 1 + \sum_{i=0}^{n-1} f_i(2g)^i : f_i \in \lbrace g,..,2g-1 \rbrace \right\rbrace.
\]
Then $w^f_n$ for $f \in \F_n$ forms a $K$-basis for $F^0V_{dr}^{\otimes n}$.
\end{lemma}

\begin{proof}
By definition $F^0V_{dr}$ has $K$-basis $\lbrace A_g,..,A_{2g-1} \rbrace= \lbrace w^{1+g}_1,..,w^{1+2g-1}_1 \rbrace$. Hence the statement of the lemma is true in the case that $n=1$. We now proceed by induction. Suppose that the statement of the lemma is true for some $n$. Then $F^0V_{dr}^{\otimes n}$ has $K$-basis $\lbrace w^f_n | f \in \F_n\rbrace$ where $\F_n$ is as in the statement of the lemma. 

Note that $F^0V_{dr}^{\otimes n+1}$ has $K$-basis $\bigcup_{i=g}^{2g-1} \lbrace A_iw^f_n | f \in \F_n \rbrace = \bigcup_{i=g}^{2g-1} \lbrace w^{(2g)^ni+f}_{n+1} | f \in \F_n \rbrace$ upon recalling that $A_i w^k_l = w_{l+1}^{k+(2g)^li}$. Therefore, we conclude that the statement of the lemma is true for $n+1$. Hence, by induction we deduce that the lemma is true for all $n\geq 1$.  
  \end{proof}

\begin{remark}
For the sake of brevity in the proof of Theorem \ref{Theorem - Hodge Filtration on HEC or EC} below we define $\F_0:= \lbrace 1 \rbrace$ since $F^0 \UU_0=F^0\O_C=\O_C$. 
\end{remark}

The main theorem of this section may now be stated as follows. 

\begin{thm} \label{Theorem - Hodge Filtration on HEC or EC}
Let $C$ be an elliptic curve or an odd hyperelliptic curve of genus $g$ over a field $K$ of characteristic $0$ and let $\PP$ be the point at infinity. Let $X:=C-\lbrace \PP \rbrace$ be the affine curve with model $y^2=f(x)$ for some $f(x) \in K[x]$ where $\deg f=2g+1$. Let $b \in X(K)$ be a basepoint. Let $\U= \lbrace (\U_n,1) \rbrace_{n \geq 0}$ be the universal connection on $X$ as in Definition \ref{Definition - UC on General Affine Curve}. Let $\UU= \lbrace (\UU_n, 1) \rbrace$ be the logarithmic extension of $\U$ computed by Algorithm \ref{Algorithm - Computing GT on UC of HEC}. 

Then there is a unique filtration \index{$F^{\bullet}\UU_n$}$F^{\bullet}\UU_n$ of $\UU_n$ for all $n$ and this filtration $F^{\bullet}\UU_n$ is explicitly computable; the logarithmic connection $\UU_n$ has the structure of a filtered logarithmic connection on $C$ fitting into an exact sequence 

\[
0 \rightarrow V_{dr}^{\otimes n} \otimes \O_C \rightarrow \UU_n \rightarrow \UU_{n-1} \rightarrow 0
\] of filtered logarithmic connections; and $1 \in b^*F^0\UU_n$ for all $n$. 
\end{thm}

Since $F^1\UU_0=F^1\O_C=0$ and $F^1V_{dr}^{\otimes n} =0$ we easily deduce by induction that $F^1\UU_n=0$ for all $n$. So the first non-trivial case we consider is the computation of the $F^0$ part of the filtration. The proof can be conveniently split up into the following three parts. First, we derive necessary and sufficient conditions on lifts of generators of $F^0\UU_n$ over $X$ and $Y$ such that any lifts satisfying them will generate a sub-bundle of $\UU_n$ satisfying the conditions of Lemma \ref{Lemma - Hadian's Lemma}. Second, we show that any lift satisfying these conditions are unique. Finally, we provide an algorithm to compute such lifts, demonstrating that they exist.

In what follows we shall need to make use of the following function. 

\begin{definition} \label{Definition - Third Auxiliary Function}
Let $(i,j,p,q) \in \Z^4$. Then define \index{$\tau$}
\[
\tau(i,j,p,q) := \begin{cases} 1 & \text{if } i=j+(\psi(i,0,p,q)-1)q^p \\ 0 & \text{otherwise} \end{cases}\]

where $\psi$ is the function in Definition \ref{Definition - Auxiliary Functions}. 
\end{definition}

\begin{lemma}\label{Lemma - Conditions on Lifts}
Suppose that the conditions in the statement of Theorem \ref{Theorem - Hodge Filtration on HEC or EC} are satisfied. Then suppose that 
\begin{itemize}
\item[($A_n$)] $F^0\U_n$ is generated as an $\O_X$-module by \index{$a_{m,f}^{l,k}$}
\begin{align*}
w^f_n & \qquad \text{where }f \in \F_n\\
w^f_m + \sum_{l=m+1}^n \sum_{k=1}^{(2g)^l}a_{m,f}^{l,k}w^k_l & \qquad \text{where } m \in \lbrace 0,..,n-1 \rbrace, f \in \F_m 
\end{align*}

for some $a_{m,f}^{l,k} \in H^0(X,\O_X)$ such that $a_{m,f}^{l,f'}=0$ for all $f' \in \F_l$ and $a_{0,1}^{l,k}(b)= 0$ for all $l,k$. 
\item[($B_n$)] $F^0\U'_n$ is generated as an $\O_Y$-module by \index{$b_{m,f}^{l,k}$}
\begin{align*}
w^f_n & \qquad \text{where }f \in \F_n\\
w^f_m + \sum_{l=m+1}^n \sum_{k=1}^{(2g)^l}b_{m,f}^{l,k}w^k_l & \qquad \text{where } m \in \lbrace 0,..,n-1 \rbrace, f \in \F_m 
\end{align*}

for some $b_{m,f}^{l,k} \in H^0(Y,\O_Y)$ such that $b_{m,f}^{l,f'}=0$ for all $f' \in \F_l$. 
\end{itemize}
Suppose there are lifts $(A_{n+1})$ and $(B_{n+1})$ generating a sub-bundle $F^0\UU_{n+1}$ of $\UU_{n+1}$ such that the following sequence is exact
\[
0 \rightarrow F^0V_{dr}^{\otimes (n+1)} \otimes \O_C \rightarrow F^0\UU_{n+1} \rightarrow F^0\UU_{n} \rightarrow 0.
\]

Then the lifts satisfy the following conditions on restriction to $X \cap Y$: \index{$I_m$}
\begin{itemize} 
\item[$I_n$:] For $f \in \F_n$ and $k \in \lbrace 1,..,(2g)^{n+1}\rbrace$ \[b_{n,f}^{n+1,k}=a_{n,f}^{n+1,k} + h^{\psi(k,0,n,2g),n}_{n+1}\tau(k,f,n,2g)- \sum_{f' \in \F_{n+1}} h^{\psi(f',0,n,2g),n}_{n+1}\tau(f',f,m,2g) \delta_{f',k}\]
\item[$I_m$:] For $f \in \F_m$ and $k \in \lbrace 1,..,(2g)^{n+1}\rbrace$ \begin{align*}
b_{m,f}^{n+1,k} = &a_{m,f}^{n+1,k} + h^{\psi(k,0,m,2g),m}_{n+1}\tau(k,f,m,2g)- \sum_{f' \in \F_{n+1}} h^{\psi(f',0,m,2g),m}_{n+1}\tau(f',f,m,2g) \delta_{f',k} \\
& + \sum_{l=m+1}^{n} a^{l,k-(\psi(k,0,l,2g)-1)(2g)^l}_{m,f}h^{\psi(k,0,l,2g),l}_{n+1}  \\
&- \sum_{p=m+1}^{n}\sum_{f' \in \F_p} b_{p,f'}^{n+1,k}h^{\psi(f',0,m,2g),m}_p\tau(f',f,m,2g)
\end{align*}
\end{itemize} where $m \in \lbrace 0,..,n-1 \rbrace$. Additionally, $a^{n+1,k}_{0,1}(b)=0$ for all $k$. 
\end{lemma} 

\begin{proof}
Assume that ($A_n$) and ($B_n$) hold for some $n$. We require that $F^0\UU_{n+1}$ fits into an exact sequence 

\[
0 \rightarrow F^0V_{dr}^{\otimes (n+1)} \otimes \O_C \rightarrow F^0\UU_{n+1} \rightarrow F^0\UU_{n} \rightarrow 0.
\]

Considering this exact sequence over $X$ we lift the generators of $F^0\U_n$ arbitrarily to $F^0\U_{n+1}$ and suppose said lifts, together with the generators of $F^0V_{dr}^{\otimes (n+1)} \otimes \O_C$, are generators of $F^0\U_{n+1}$. We let these generators be \index{$T^f_{m}$}
\begin{align*}
T^f_{n+1}:=w^f_{n+1} & \qquad \text{where }f \in \F_{n+1}\\
T^f_{m}:=w^f_m + \sum_{l=m+1}^{n+1} \sum_{k=1}^{(2g)^l}a_{m,f}^{l,k}w^k_l & \qquad \text{where } m \in \lbrace 0,..,n \rbrace, f \in \F_m 
\end{align*} for some $a_{m,f}^{n+1,k} \in H^0(X,\O_X)$. Since we suppose these are generators of $F^0\U_{n+1}$ we may assume that $a_{m,f}^{n+1,f'}=0$ for all $f' \in \F_{n+1}$. 

We repeat this over $Y$ obtaining lifts\index{$S^f_{m}$} \begin{align*}
S^f_{n+1}:=w^f_{n+1} & \qquad \text{where }f \in \F_{n+1}\\
S^f_{m}:=w^f_m + \sum_{l=m+1}^{n+1} \sum_{k=1}^{(2g)^l}b_{m,f}^{l,k}w^k_l & \qquad \text{where } m \in \lbrace 0,..,n \rbrace, f \in \F_m 
\end{align*} for some $b_{m,f}^{n+1,k} \in H^0(X,\O_X)$. We take these as candidate generators for $F^0\U'_{n+1}$, and again we may assume that $b_{m,f}^{n+1,f'}=0$ for all $f' \in \F_{n+1}$. 

We need to show the following that over $X \cap Y$ we have an isomorphism of bundles

\[
G_{n+1}:F^0\U_{n+1}|_{X \cap Y} \cong F^0\U'_{n+1}|_{X \cap Y}.
\]

We consider the action of $G_{n+1}$ on the candidate generators for $F^0\U_{n+1}$. In what follows we must work over $X \cap Y$, but we suppress the notation $|_{X \cap Y}$ for ease of exposition. Using Lemma \ref{Lemma - Action of G_n on B_n} we determine the action of $G_{n+1}$ on the candidate generators $T^f_m$ of $F^0\U_{n+1}$. We summarise these calculations below: \index{$\tilde{T}^f_m$}
\begin{align*}
T^f_{n+1}  \mapsto  \tilde{T}^f_{n+1}& := w^f_{n+1}\\
T^f_m \mapsto  \tilde{T}^f_m:= & w^f_m+\sum_{s=m+1}^{n+1} \sum_{t=1}^{(2g)^{s-m}} w_s^{f+(t-1)(2g)^m}h_s^{t,m} \\ 
& + (1-\delta_{m,n})\sum_{l=m+1}^{n} \sum_{k=1}^{(2g)^l}a_{m,f}^{l,k}\left(w^k_l+\sum_{s=l+1}^{n+1} \sum_{t=1}^{(2g)^{s-l}} w_s^{k+(t-1)(2g)^l}h_s^{t,l}\right) \\
& + \sum_{k=1}^{(2g)^{n+1}}a_{m,f}^{n+1,k}w^k_{n+1}.
\end{align*} 
for $m \in \lbrace 0,..,n\rbrace$ and $f \in \F_m$. Define $\tilde{S}^f_m$\index{$\tilde{S}^f_m$} recursively as follows
\begin{align*}
\tilde{S}^f_{n+1} & := S^f_{n+1} \\
\tilde{S}^f_m &:= \tilde{T}^f_m - \sum_{p=m+1}^{n+1}\sum_{f' \in \F_p} \lambda^{f'}_p\tilde{S}^{f'}_p
\end{align*} where $m \in \lbrace 0,..,n \rbrace$ and $\lambda^{f'}_p \in H^0(X \cap Y, \O_{X \cap Y})$ is the co-efficient of $w^{f'}_p$ appearing in $\tilde{T}^f_m$. As the coefficient of $w^{f'}_{m'}$ in $\tilde{S}^f_m$ is $1$ if $f=f',m=m'$ and is $0$ otherwise we have the following equality for all $m,f$ 
\begin{equation}\label{Equation - Equality of Generators on X cap Y}
\tilde{S}^f_m = S^f_m |_{X \cap Y}.
\end{equation}

Using Lemma \ref{Lemma - Hodge Basis for Vdr} we calculate that $\lambda^{f'}_p=h^{\psi(f',0,m,2g),m}_p\tau(f',f,m,2g)$. The equality (\ref{Equation - Equality of Generators on X cap Y}) is trivially true by definition when $m=n+1$. For $m<n+1$ we proceed recursively, finding first that
\begin{align*}
\tilde{S}^f_n& = w^f_n+\sum_{t=1}^{2g} w_{n+1}^{f+(t-1)(2g)^n}h_{n+1}^{t,n} + \sum_{k=1}^{(2g)^{n+1}}a_{n,f}^{n+1,k}w^k_{n+1} \\ 
&  - \sum_{f' \in \F_{n+1}} h^{\psi(f',0,n,2g),n}_{n+1}\tau(f',f,n,2g)w^{f'}_{n+1}\\
&= w^f_n+ \sum_{k=1}^{(2g)^{n+1}}b_{n,f}^{n+1,k}w^k_{n+1}.
\end{align*}

By considering the coefficient of $w^q_{n+1}$ on both sides of the equality above we conclude that on restriction to $X \cap Y$ the sections $a_{n,f}^{n+1,q}, b_{n,f}^{n+1,q}$ must satisfy
\[
b_{n,f}^{n+1,q}=a_{n,f}^{n+1,q} + h^{\psi(q,0,n,2g),n}_{n+1}\tau(q,f,m,2g)- \sum_{f' \in \F_{n+1}} h^{\psi(f',0,n,2g),n}_{n+1}\tau(f',f,m,2g) \delta_{f',q}
\]

Now suppose that $m \in \lbrace 0,..,n-1 \rbrace$. Then if we suppose that we have an equality $\tilde{S}^f_p= S^f_p|_{X \cap Y}$ for $p>m$ we find that 

\begin{align*}
\tilde{S}^f_m&= w^f_m+\sum_{s=m+1}^{n+1} \sum_{t=1}^{(2g)^{s-m}} w_s^{f+(t-1)(2g)^m}h_s^{t,m} + \sum_{k=1}^{(2g)^{n+1}}a_{m,f}^{n+1,k}w^k_{n+1}\\ 
& + \sum_{l=m+1}^{n} \sum_{k=1}^{(2g)^l}a_{m,f}^{l,k}\left(w^k_l+\sum_{s=l+1}^{n+1} \sum_{t=1}^{(2g)^{s-l}} w_s^{k+(t-1)(2g)^l}h_s^{t,l}\right) \\
& - \sum_{p=m+1}^{n}\sum_{f' \in \F_p} h^{\psi(f',0,m,2g),m}_p\tau(f',f,m,2g)\left(w^{f'}_p + \sum_{l=p+1}^{n+1} \sum_{k=1}^{(2g)^l}b_{p,f'}^{l,k}w^k_l\right) \\
& - \sum_{f' \in \F_{n+1}} h^{\psi(f',0,m,2g),m}_{n+1}\tau(f',f,m,2g)w^{f'}_{n+1}\\
& = w^f_m + \sum_{l=m+1}^{n+1} \sum_{k=1}^{(2g)^l}b_{m,f}^{l,k}w^k_l= S^f_m
\end{align*}

Since the $a^{l,q}_{m,f}, b^{l,q}_{m,f}$ are known for $l < n+1$, we may assume that we have equality of coefficients among words of length at most $n$. So we need only concern ourselves with words of length $n+1$ in the above. We consider the coefficient of the word $w^q_{n+1}$ where $q \in \lbrace 1,..,(2g)^{n+1} \rbrace$. Arguing as when $m=n$ we conclude that on restriction to $X \cap Y$ the sections $a_{m,f}^{n+1,q}, b_{m,f}^{n+1,q}$ satisfy
\begin{align*}
b_{m,f}^{n+1,q} = &a_{m,f}^{n+1,q} + h^{\psi(q,0,m,2g),n}_{n+1}\tau(q,f,m,2g)- \sum_{f' \in \F_{n+1}} h^{\psi(f',0,m,2g),m}_{n+1}\tau(f',f,m,2g) \delta_{f',q} \\
& + \sum_{l=m+1}^{n} a^{l,q-(\psi(q,0,l,2g)-1)(2g)^l}_{m,f}h^{\psi(q,0,l,2g),l}_{n+1}  - \sum_{p=m+1}^{n}\sum_{f' \in \F_p} b_{p,f'}^{n+1,q}h^{\psi(f',0,m,2g),m}_p\tau(f',f,m,2g).
\end{align*}
Therefore, we have shown that any such lifts must satisfy the conditions in the statement of the lemma. 
  \end{proof}

We now show that any sections satisfying these conditions must be unique.

\begin{lemma}\label{Lemma - Uniqueness of Lifts}
Suppose that there are lifts \index{$S^f_m$}\index{$T^f_m$}$S^f_m, T^f_m$ for $m \in \lbrace 0,..,n \rbrace$ and $f \in \F_m$ satisfying the conditions $(I_m)$ of Lemma \ref{Lemma - Conditions on Lifts}. Then the lifts $S^f_m, T^f_m$ are unique. 
\end{lemma}

\begin{proof}
We suppose that we have a second set of lifts satisfying the conditions $I_m$ of Lemma \ref{Lemma - Conditions on Lifts}:
\begin{align*}
\overline{T}^f_{m}&:=w^f_m + \sum_{l=m+1}^{n} \sum_{k=1}^{(2g)^l}a_{m,f}^{l,k}w^k_l+\sum_{k=1}^{(2g)^{n+1}}\tilde{a}_{m,f}^{n+1,k}w^k_{n+1}\\
\overline{S}^f_{m}&:=w^f_m + \sum_{l=m+1}^{n} \sum_{k=1}^{(2g)^l}b_{m,f}^{l,k}w^k_l+\sum_{k=1}^{(2g)^{n+1}}\tilde{b}_{m,f}^{n+1,k}w^k_{n+1}
\end{align*} $\text{where } m \in \lbrace 0,..,n \rbrace, f \in \F_m$ and $\tilde{a}_{m,f}^{n+1,k},\tilde{b}_{m,f}^{n+1,k}$ are sections in $H^0(X,\O_X)$ and $H^0(Y,\O_Y)$ respectively. First, for $f \in \F_n$ and $k \in \lbrace 1,..,(2g)^{n+1}\rbrace$ we consider the difference
\[
b_{n,f}^{n+1,k}-\tilde{b}_{n,f}^{n+1,k}= a_{n,f}^{n+1,k}-\tilde{a}_{n,f}^{n+1,k}.
\]

Note that since $a_{n,f}^{n+1,k}-\tilde{a}_{n,f}^{n+1,k}\in H^0(X,\O_X)$ and $b_{n,f}^{n+1,k}-\tilde{b}_{n,f}^{n+1,k} \in H^0(Y,\O_Y)$ and these sections agree on $X \cap Y$ they must glue to give a global section in $H^0(C,\O_C)$. Therefore, $a_{n,f}^{n+1,k}-\tilde{a}_{n,f}^{n+1,k}$ and $b_{n,f}^{n+1,k}-\tilde{b}_{n,f}^{n+1,k}$ are constant and equal. 

We now proceed by induction on $m \in \lbrace 0,..,n-1 \rbrace$ with the following induction hypothesis:

\begin{itemize}
\item[$(C_m)$:] For all $f \in \F_{m}$, $f' \in \F_{m+1}$ and $k,k' \in \lbrace 1,..,(2g)^{n+1}\rbrace$ the following hold:
\begin{enumerate}
\item $b_{m,f}^{n+1,k}-\tilde{b}_{m,f}^{n+1,k}= a_{m,f}^{n+1,k}-\tilde{a}_{m,f}^{n+1,k} \in H^0(C,\O_C)=K$
\item $0=b_{m+1,f'}^{n+1,k'}-\tilde{b}_{m+1,f'}^{n+1,k'}= a_{m+1,f'}^{n+1,k'}-\tilde{a}_{m+1,f'}^{n+1,k'} \in H^0(C,\O_C)=K$
\end{enumerate}
\end{itemize}

Take $m=n-1$ as our base case. To show $(C_{n-1})$ we consider for each $f \in \F_{n-1}$ and $k \in \lbrace 1,..,(2g)^{n+1}\rbrace$ the difference

\begin{equation}\label{Equation - HF Proof 1}
b_{n-1,f}^{n+1,k}-\tilde{b}_{n-1,f}^{n+1,k}= a_{n-1,f}^{n+1,k}-\tilde{a}_{n-1,f}^{n+1,k}-\sum_{f' \in \F_n} \left(b_{n,f'}^{n+1,k}-\tilde{b}_{n,f'}^{n+1,k}\right)h^{\psi(f',0,n-1,2g),n-1}_n\tau(f',f,n-1,2g). 
\end{equation}

Recall that $h^{\psi(f',0,n-1,2g),n-1}_n= h^{\psi(f',0,n-1,2g),0}_1$. As $f' \in \F_n$ we find that $g+1 \leq \psi(f',0,n-1,2g) \leq 2g$ for each $f' \in \F_n$. Using Algorithm \ref{Algorithm - Computing GT on UC of HEC} we note that $\alpha_{t-1}$ has a pole of order $2(t-g)$ at $\PP$ and $h^{t,0}_1$ is chosen so that $dh^{t,0}_1-\alpha_{t-1}$ has at worst logarithmic poles at $\PP$ and hence $h^{t,0}_1$ has a pole of order $2(t-g)-1$ at $\PP$ for $t>g$. Hence the order of the pole of $h^{\psi(f',0,n-1,2g),n-1}_n$ at $\PP$ is in $\lbrace 1,3,...,2g-1 \rbrace$.

Note that $x$ has a pole of order $2$ at $\PP$ and $y$ has a pole of order $2g+1$ at $\PP$ and hence $a_{n-1,f}^{n+1,k}-\tilde{a}_{n-1,f}^{n+1,k}$ can only have a pole of order lying in $\lbrace 2,4,..,2g\rbrace \cup \lbrace o \in \Z : o \geq 2g+1 \rbrace$. Since $b_{n-1,f}^{n+1,k}-\tilde{b}_{n-1,f}^{n+1,k}$ is regular at $\PP$ the only way we can have equality over $X \cap Y$ as in (\ref{Equation - HF Proof 1}) is if
\[
\left(b_{n,f'}^{n+1,k}-\tilde{b}_{n,f'}^{n+1,k}\right)\tau(f',f,n-1,2g) =0 
\]for all $f' \in \F_n$ and $k \in \lbrace 1,..,(2g)^{n+1}\rbrace$. By varying over all $f \in \F_{n-1}$ we deduce, therefore, that 
\[
b_{n,f'}^{n+1,k}-\tilde{b}_{n,f'}^{n+1,k}=a_{n,f'}^{n+1,k}-\tilde{a}_{n,f'}^{n+1,k}=0 
\]for all $f' \in \F_n$ and $k \in \lbrace 1,..,(2g)^{n+1}\rbrace$. Finally, this implies that 
\[
b_{n-1,f}^{n+1,k}-\tilde{b}_{n-1,f}^{n+1,k}= a_{n-1,f}^{n+1,k}-\tilde{a}_{n-1,f}^{n+1,k}
\] and as before we conclude that these sections glue to a global section and hence are constant. Therefore, we find that $(C_{n-1})$ holds. By considering the differences $b_{m,f}^{n+1,k}-\tilde{b}_{m,f}^{n+1,k}$ for $f \in \F_m$ and $k \in \lbrace 1,..,(2g)^{n+1}\rbrace$ we conclude that $(C_m)$ holds for all $m$. 

Therefore, for $m \geq 1$ the sections $a_{m,f}^{n+1,k}, b_{m,f}^{n+1,k}$ are unique. Additionally, $b_{0,f}^{n+1,k}-\tilde{b}_{0,f}^{n+1,k}=a_{0,f}^{n+1,k}-\tilde{a}_{0,f}^{n+1,k} \in K$. However, recall that our lifts should be such that $1 \in b^*F^0\UU_{n+1}$. Therefore, these sections must be $0$ at $b$ and hence they must be exactly $0$. Hence, for $m \geq 0$ the sections $a_{m,f}^{n+1,k}, b_{m,f}^{n+1,k}$ are unique and thus the lifts are unique. 
  \end{proof}

The existence of suitable lifts is demonstrated by the following Algorithm.

\begin{algorithm}\label{Algorithm - Compute HF on HEC or EC}
(Computing $F^0\UU_n$ on elliptic and odd hyperelliptic curves $X=C- \lbrace \PP \rbrace$) \newline

\textbf{Input}
\begin{itemize}
\item $C$ a complete elliptic or odd hyperelliptic curve of genus $g$ over a field $K$, $D=\lbrace \PP \rbrace$ and $X:= C-D$ with affine model $y^2=f(x)$ for some $f(x) \in K[x]$ with $\deg f=2g+1$.
\item The logarithmic extension $\UU_n$ of $\U_n$ computed by Algorithm \ref{Algorithm - Compute HF on HEC or EC}.
\item Generators of $F^0\U_n$ as an $\O_X$-module: \index{$a_{m,f}^{l,k}$}
\begin{align*}
w^f_n & \qquad \text{where }f \in \F_n\\
w^f_m + \sum_{l=m+1}^n \sum_{k=1}^{(2g)^l}a_{m,f}^{l,k}w^k_l & \qquad \text{where } m \in \lbrace 0,..,n-1 \rbrace, f \in \F_m 
\end{align*}
and generators of $F^0\U'_n$ as an $\O_Y$-module: \index{$b_{m,f}^{l,k}$}
\begin{align*}
w^f_n & \qquad \text{where }f \in \F_n\\
w^f_m + \sum_{l=m+1}^n \sum_{k=1}^{(2g)^l}b_{m,f}^{l,k}w^k_l & \qquad \text{where } m \in \lbrace 0,..,n-1 \rbrace, f \in \F_m. 
\end{align*}
\end{itemize}

\textbf{Output}

\begin{itemize}
\item Generators of $F^0\U_{n+1}$ as an $\O_X$-module:
\begin{align*}
w^f_n & \qquad \text{where }f \in \F_n\\
w^f_m + \sum_{l=m+1}^{n+1} \sum_{k=1}^{(2g)^l}a_{m,f}^{l,k}w^k_l & \qquad \text{where } m \in \lbrace 0,..,n-1 \rbrace, f \in \F_m 
\end{align*}
and generators of $F^0\U'_{n+1}$ as an $\O_Y$-module:
\begin{align*}
w^f_{n+1} & \qquad \text{where }f \in \F_{n+1}\\
w^f_m + \sum_{l=m+1}^{n+1} \sum_{k=1}^{(2g)^l}b_{m,f}^{l,k}w^k_l & \qquad \text{where } m \in \lbrace 0,..,n \rbrace, f \in \F_m.
\end{align*}
\end{itemize}

\textbf{Algorithm}

\begin{enumerate}
\item[(1)] Fix a local parameter $\pi$ at $\PP$ e.g. $\pi= \frac{x^g}{y}$.
\item[(2)] Compute the $\pi$-expansion of the $\O_X$-sections $x$ and $y$ and each $h^{t,0}_{1}$:
\begin{align*}
x(\pi) & = \chi \pi^{-2} +... \\
y(\pi) & = \gamma \pi^{-(2g+1)}+..\\
h^{t,0}_{1}(\pi) & = \eta^{t,0}_1 \pi^{-(2(t-g)-1)}+..\\
\end{align*}
for some constants $\chi, \gamma, \eta^{t,0}_1 \in K$. 
\item[(3)] For $f \in \F_{n+1}$ do:
\begin{align*}
T^f_{n+1}&:=w^f_{n+1} \\
S^f_{n+1}&:=w^f_{n+1} \\
\end{align*}
\item[(4)] For $f \in \F_n$ do:
\begin{enumerate}
\item For $k \in \F_{n+1}$ do:
\[
a^{n+1,k}_{n,f},b^{n+1,k}_{n,f}:=0
\]
\item For $k \in \lbrace 1,..,(2g)^{n+1} \rbrace - \F_{n+1}$ do:
\[
A^{n+1,k}_{n,f}:=0
\] 
\item Let $m:=n-1$.
\end{enumerate}
\item [(5)] While $m \geq 0$, for $f \in \F_m$ and $k \in \lbrace 1,..,(2g)^{n+1} \rbrace$ do:
\begin{enumerate}
\item[If:] $k \in \F_{n+1}$ do:
\[
a^{n+1,k}_{m,f},b^{n+1,k}_{m,f}:=0
\]
\item[Else:] 
\item Define 
\begin{align*}
s^{n+1,k}_{m,f}& := h^{\psi(k,0,m,2g),m}_{n+1}\tau(k,f,m,2g)+ \sum_{l=m+1}^{n} a^{l,k-(\psi(k,0,l,2g)-1)(2g)^l}_{m,f}h^{\psi(k,0,l,2g),l}_{n+1}\\
& - \sum_{f' \in \F_{m+1}} \left(b_{m+1,f'}^{n+1,k}-\lambda^{n+1,k}_{m+1,f'}\right)h^{\psi(f',0,m,2g),m}_{m+1}\tau(f',f,m,2g)\\
&- (1- \delta_{m,n-1})\sum_{p=m+2}^{n}\sum_{f' \in \F_p} b_{p,f'}^{n+1,k}h^{\psi(f',0,m,2g),m}_p\tau(f',f,m,2g).
\end{align*}
\item Compute the $\pi$-expansion of $s^{n+1,k}_{m,f}$.
\item Compute $A^{n+1,k}_{m,f}(x,y) \in K[x,y]/(y^2-f(x))$ such that 
\[A^{n+1,k}_{m,f}(x(\pi),y(\pi))+s^{n+1,k}_{m,f}(\pi) \equiv 0 \mod \pi^{-(2g-1)}K[[\pi]].\]
\item Suppose that 
\[
A^{n+1,k}_{m,f}(x(\pi),y(\pi))+s^{n+1,k}_{m,f}(\pi) = \mu_{2g-1} \pi^{-(2g-1)}+...
\]
Then define $ \lambda^{n+1,k}_{m+1,f+(2g-1)(2g)^m}:= \frac{\mu_{2g-1}}{\eta^{2g,0}_{1}}$. 
\item Define 
\[
a^{n+1,k}_{m+1,f+(2g-1)(2g)^m}:=A^{n+1,k}_{m+1,f+(2g-1)(2g)^m}+\lambda^{n+1,k}_{m+1,f+(2g-1)(2g)^m}.
\]
\item Suppose that 
\begin{align*}
&A^{n+1,k}_{m,f}(x(\pi),y(\pi))+s^{n+1,k}_{m,f}(\pi) - \lambda^{n+1,k}_{m+1,f+(2g-1)(2g)^m}h^{2g,m}_{m+1}(\pi)\\ 
& = \mu_{2g-2} \pi^{-(2g-2)}+...
\end{align*}

Then define $A^{n+1,k}_{m,f}:=A^{n+1,k}_{m,f}(x,y)-\frac{\mu_{2g-2}}{\chi^{g-1}}x^{g-1}$. 
\item Let $\tilde{t}=2g-1$. While $\tilde{t} \neq g+1$ do 
\begin{enumerate}
\item[(i)] Suppose that 
\begin{align*}
&A^{n+1,k}_{m,f}(x(\pi),y(\pi))+s^{n+1,k}_{m,f}(\pi) -\sum_{t=\tilde{t}+1}^{2g} \lambda^{n+1,k}_{m+1,f+(t-1)(2g)^m}h^{t,m}_{m+1}(\pi) \\
& = \mu_{2(\tilde{t}-g)-1}\pi^{-(2(\tilde{t}-g)-1)}+..
\end{align*}
Then define $ \lambda^{n+1,k}_{m+1,f+(\tilde{t}-1)(2g)^m}:= \frac{\mu_{2(\tilde{t}-g)-1}}{\eta^{\tilde{t},m}_{m+1}}$.
\item[(ii)] Define 
\[
a^{n+1,k}_{m+1,f+(\tilde{t}-1)(2g)^m}:=A^{n+1,k}_{m+1,f+(\tilde{t}-1)(2g)^m}+\lambda^{n+1,k}_{m+1,f+(\tilde{t}-1)(2g)^m}.
\]
\item[(iii)] Suppose that 
\begin{align*}
&A^{n+1,k}_{m,f}(x(\pi),y(\pi))+s^{n+1,k}_{m,f}(\pi) -\sum_{t=\tilde{t}}^{2g} \lambda^{n+1,k}_{m+1,f+(t-1)(2g)^m}h^{t,0}_{1}(\pi)\\
&= \mu_{2(\tilde{t}-g)-2}\pi^{-(2(\tilde{t}-g)-2)}+..
\end{align*}

Then define $A^{n+1,k}_{m,f}:=A^{n+1,k}_{m,f}(x,y)-\frac{\mu}{\chi^{\tilde{t}-g-1}}x^{\tilde{t}-g-1}$.
\item[(iv)] $\tilde{t}:=\tilde{t}-1$. 
\end{enumerate}
\item Repeat Step (5) (g) (i) for $\tilde{t}=g+1$.
\item Let $m:=m-1$.
\end{enumerate}
\item[(6)] For $k \in \lbrace 1,..,(2g)^{n+1} \rbrace - \F_{n+1}$ do:
\[
a^{n+1,k}_{0,1}:= A^{n+1,k}_{0,1}(x,y)-A^{n+1,k}_{0,1}(x(b),y(b))
\]
\item[(7)] For $k \in \lbrace 1,..,(2g)^{n+1} \rbrace - \F_{n+1}$ do:
\begin{enumerate}
\item[(i)] For $f \in \F_n$ define $b^{n+1,k}_{n,f}$ by
\[
b^{n+1,k}_{n,f}=a^{n+1,k}_{n,f}+s^{n+1,k}_{n,f}.
\] 
\item[(ii)] For $m \in \lbrace 0,..,n-1\rbrace$ and $f \in \F_m$ define $b^{n+1,k}_{m,f}$ by
\[
b^{n+1,k}_{m,f}=a^{n+1,k}_{m,f}+s^{n+1,k}_{m,f} -\sum_{t=g+1}^{2g} \lambda^{n+1,k}_{m+1,f+(t-1)(2g)^m}h^{t,m}_{m+1}.
\] 
\end{enumerate}
\end{enumerate}
\end{algorithm}

\begin{proof}
We shall now show that the lifts $S^f_m,T^f_m$ exist. Before we start, recall that in the proof of Lemma \ref{Lemma - Conditions on Lifts} we observed when taking the lifts $T^f_m,S^f_m$ that we could assume $a^{n+1,k}_{m,f},b^{n+1,k}_{m,f}$ are both $0$ if $k \in \F_{n+1}$. Henceforth we shall assume that $k \not \in \F_{n+1}$. We proceed inductively with the following inductive hypothesis:

\begin{itemize}
\item[$(D_m)$:] For $f \in \F_m$ and $k \in \lbrace 1,..,(2g)^{n+1}\rbrace - \F_{n+1}$ the following hold:
\begin{enumerate}
\item There exists an $A^{n+1,k}_{m,f}(x,y) \in K[x,y]/(y^2-f(x))$ such that $\lambda^{n+1,k}_{m,f}:=a^{n+1,k}_{m,f}-A^{n+1,k}_{m,f}(x,y)$ is constant
\item The constant $\lambda^{n+1,k}_{p,f}$ is known for $p > m+1$. 
\end{enumerate}
\end{itemize}

First we consider condition $(I_n)$. Take $f \in \F_n$ and $k \in \lbrace 1,..,(2g)^{n+1}\rbrace - \F_{n+1}$. Observe that in condition $(I_n)$ the expression 
\[ s^{n+1,k}_{n,f}:=h^{\psi(k,0,n,2g),n}_{n+1}\tau(k,f,m,2g)- \sum_{f' \in \F_{n+1}} h^{\psi(f',0,n,2g),n}_{n+1}\tau(f',f,m,2g) \delta_{f',k}\] is known. We shall now analyse the possible orders of poles of $s^{n+1,k}_{n,f}$ at $\PP$. 

Since $k \not \in \F_{n+1}$ we conclude that $s^{n+1,k}_{m,f}= h^{\psi(k,0,n,2g),n}_{n+1}\tau(k,f,m,2g)$. Since $f \in \F_n$ and $k \not \in \F_{n+1}$ by Lemma \ref{Definition - Hodge Filtration on Vdr} we conclude that $\psi(k,0,n,2g) \in \lbrace 1,..,g \rbrace$ and, hence, each $h^{\psi(k,0,n,2g),n}_{n+1}$ is regular at $\PP$. Therefore, each $s^{n+1,k}_{n,f}$ is regular at $\PP$. 

As $b^{n+1,k}_{n,f}=a^{n+1,k}_{n,f}+ s^{n+1,k}_{n,f}$ on $X \cap Y$ and both $b^{n+1,k}_{n,f}$ and $s^{n+1,k}_{n,f}$ are regular at $\PP$ we conclude that $a^{n+1,k}_{n,f}$ must be constant as all the non-constant sections in $H^0(X,\O_X)$ have poles of order at least $2$ at $\PP$. We let $a^{n+1,k}_{n,f}=\lambda^{n+1,k}_{n,f} \in K$ be arbitrary for now. 

Observe that for $f \in \F_n$ and $k \in \lbrace 1,..,(2g)^{n+1}\rbrace - \F_{n+1}$ we have $A^{n+1,k}_{n,f}= 0$ and $(D_n)$ holds trivially. Now we consider condition $(I_m)$. Suppose we have shown that $(D_{m+1})$ holds. We define $s^{n+1,k}_{m,f}$ for  $f \in \F_m$ and $k \not \in \F_{n+1}$ as in Step (5)(a) of the Algorithm. Then rearranging condition $(I_m)$ we find that \begin{equation}\label{Equation - Hodge Proof Eqn 2}b^{n+1,k}_{m,f}=a^{n+1,k}_{m,f}+s^{n+1,k}_{m,f} -\sum_{f' \in \F_{m+1}} \lambda^{n+1,k}_{m+1,f'}h^{\psi(f',0,m,2g),m}_{m+1}\tau(f',f,m,2g).\end{equation} Observe that for $p > m+1$ the expansions in the fixed local parameter $\pi$ of $b_{p,f'}^{n+1,k}$ appearing in $s^{n+1,k}_{m,f}$ are known. This follows from a simple induction on $p$ together with the original induction hypothesis $(D)$ upon writing $b_{p,f'}^{n+1,k}$ as follows:

\begin{align*}
b_{p,f'}^{n+1,k}&= A^{n+1,k}_{p,f'}(x,y)+\lambda^{n+1,k}_{p,f'}+s^{n+1,k}_{p,f'} -\sum_{f^* \in \F_{m+1}} \lambda^{n+1,k}_{p+1,f^*}h^{\psi(f^*,0,m,2g),m}_{m+1}\tau(f^*,f',m,2g)
\end{align*}

Similarly, since we know by our inductive hypothesis that $a_{m+1,f'}^{n+1,k}=A_{m+1,f'}^{n+1,k}+\lambda_{m+1,f'}^{n+1,k}$ for some unknown constant $\lambda_{m+1,f'}^{n+1,k}$ it follows that $b_{m+1,f'}^{n+1,k}-\lambda_{m+1,f'}^{n+1,k}$ is known. Therefore, we conclude that we may compute the $\pi$-expansion of $s_{m+1,f'}^{n+1,k}$.

From the definition of $\tau$ and Lemma $\ref{Lemma - Hodge Basis for Vdr}$ we may express (\ref{Equation - Hodge Proof Eqn 2}) as 

\[
b^{n+1,k}_{m,f}=a^{n+1,k}_{m,f}+s^{n+1,k}_{m,f} -\sum_{t=g+1}^{2g} \lambda^{n+1,k}_{m+1,f+(t-1)(2g)^m}h^{t,m}_{m+1}.
\] 

Recall that $h^{t,m}_{m+1}=h^{t,0}_1$ has a pole of order $2(t-g)-1$ at $\PP$. The orders of the poles at $\PP$ of the $s^{n+1,k}_{m,f}, h^{t,m}_{m+1}$ force the choices made in the above iterative process. As we vary over all $f \in \F_m$ we will compute all $\lambda^{n+1,k}_{m,f'}$. Finally, we deduce that $a^{n+1,k}_{m,f}$ must be of the form $A^{n+1,k}_{m,f}+\lambda^{n+1,k}_{m,f}$ for some constant $\lambda^{n+1,k}_{m,f}$. So we have shown that $(D_{m+1}) \Rightarrow (D_m)$ and by induction $(D_m)$ is true for all $m$. Hence, using the above iterative process we have computed $a^{n+1,k}_{m,f}$ for all $m>0$ and $f \in \F_m$. We know for $f \in \F_0=\lbrace 1 \rbrace$ that $a^{n+1,k}_{0,f} =A^{n+1,k}_{0,f}+\lambda^{n+1,k}_{0,f}$ for some known $A^{n+1,k}_{0,f}(x,y) \in K[x,y]/(y^2-f(x))$ and a constant $\lambda^{n+1,k}_{0,f} \in K$ to be determined. Since we require that $1 \in b^*F^0\UU_n$ we must have for all $k$ that $a^{n+1,k}_{0,f}(b)=0$. Hence, we require that $\lambda^{n+1,k}_{0,f}=-A^{n+1,k}_{0,f}(x(b),y(b))$. 
  \end{proof}

\begin{proof}[Theorem \ref{Theorem - Hodge Filtration on HEC or EC}]
We proceed by induction. Following Example \ref{Example - Level 1} we know that the statement of the theorem is true for $n=1$ so we now suppose that it is true for some $n$. That is, we have unique generators for $F^0\UU_n$ over $X$ and $Y$ respectively that lift generators of $F^0\UU_{n-1}$. By Lemma \ref{Lemma - Conditions on Lifts}, Lemma \ref{Lemma - Uniqueness of Lifts} and Algorithm \ref{Algorithm - Compute HF on HEC or EC} there exist unique generators of a bundle $F^0\UU_{n+1}$ lifting generators of $F^0\UU_n$ which satisfy the conditions of Lemma \ref{Lemma - Hadian's Lemma} (Hadian's lemma). 
For $p<0$ we now proceed to show the existence and uniqueness of $F^p\UU_n$ using the approach used above together with Griffith's transversality on the generators of $F^0\UU_n$. 
  \end{proof}

\subsection{An application of Theorem 4 to elliptic curves}

In this subsection we focus on the case where $C$ is a complete elliptic curve. We apply Algorithm \ref{Algorithm - Compute HF on HEC or EC} to compute the generators of $F^0\UU_n$. 

\begin{remark} \label{Remark - Dropping subscript f}
Before we proceed note that in the case that $C$ is an elliptic curve $\F_n = \lbrace 2^n \rbrace$. Therefore, we can dispense with identifying a distinguished $f \in \F_n$ and, as such, we will drop the subscript $f$ throughout this section. 
\end{remark}

To describe the algorithm in question we will first give a restatement of the conditions $I_m$ appearing in the proof of Theorem \ref{Theorem - Hodge Filtration on HEC or EC}. 

\begin{lemma}\label{Lemma - Conditions for EC}
When $C$ is an elliptic curve the conditions $I_m$ of Lemma \ref{Lemma - Conditions on Lifts} may be restated as follows:
\begin{itemize}\index{$I_m$}
\item[$I_n$:] For $k \in \lbrace 1,..,2^{n+1}-1\rbrace$ \[b_{n}^{n+1,k}=a_{n}^{n+1,k} + h^{1,n}_{n+1}\delta_{k,2^n}\]
\item[$I_m$:] For $k \in \lbrace 1,..,2^{n+1}-1\rbrace$ \begin{align*}
b_{m}^{n+1,k} = &a_{m}^{n+1,k} + \sum^{2^{n+1-m}-1}_{t=1}h^{t,m}_{n+1}\delta_{k,t2^m}+ \sum_{l=m+1}^{n} a^{l,k-(\psi(k,0,l,2g)-1)(2g)^l}_{m}h^{\psi(k,0,l,2g),l}_{n+1} \\
&- \sum_{p=m+1}^{n}b^{n+1,k}_p h^{2^{p-m},m}_p
\end{align*}
\end{itemize} where $m \in \lbrace 0,..,n-1 \rbrace$.
\end{lemma}

\begin{proof}
Recall that if $k = 2^{n+1}$ then $b_{m}^{n+1,k}=a_{m}^{n+1,k}=0$ and so we consider only the cases where $k < 2^{n+1}$. The remaining cases are simple to deal with by observing the following: $\tau(k,2^n,n,2)=1$ if $k = 2^n+(\psi(k,0,n,2)-1)2^n = \psi(k,0,n,2)2^n$ and is $0$ otherwise; since $k < 2^{n+1}$ and $\psi(k,0,n,2) \in \lbrace 1,2 \rbrace$; $\psi(2^p,0,m,2)= 2^{p-m}$ by definition of $\psi$; and, finally, $\tau(2^p,2^m,m,2)=1$.
  \end{proof}

Using the above formulation we will now make several explicit applications of Theorem \ref{Theorem - Hodge Filtration on HEC or EC} to compute $F^0\U_n$ on an elliptic curve $C$ for several new values of $n$. The case when $n=1$ is known by Example \ref{Example - Level 1}. In \cite[Lemma 3.2]{kim10} Kim shows that $F^0\U_2$ is generated by $1,A_1,A_1^2$ and so we expect that Algorithm \ref{Algorithm - Compute HF on HEC or EC} should reproduce this as its output. 

\begin{proposition} \label{Proposition - HF on UC of EC Level 2}
The generators of $F^0\UU_2$ over both $X$ and $Y$ are $1, A_1, A_1^2$.
\end{proposition}

\begin{proof}
By Proposition \ref{Proposition - GT for EC Level 2} we have that the gauge transformation $G_2$ is of the form 

\[
G_2 = \left( 	\begin{array}{cc}
					I_{4 \times 4} & H_2 \\
					0_{3 \times 4} & G_1 
				\end{array}\right), \; \; H_2 = \left( 	\begin{array}{ccc}
					0 & 0 & 0 \\
					0 & 0 & 0 \\
					F & 0 & 0 \\
					0 & F & \frac{1}{2}F^2 
				\end{array} \right) 
\]
with respect to the basis $\B_2$, where $F \in K(C)$ is such that $- \alpha_1 + dF$ is regular at $\PP$. In Example \ref{Example - Level 1} we saw that $F^0\UU_1$ is generated by $1$ and $A_1$. We look for lifts of these generators over $X$:
\[
1 + \sum_{k=1}^{3}a^{2,k}_0 w^k_2, \quad A_1 + \sum_{k=1}^{3} a^{2,k}_1 w^k_2, \quad A_1^2
\]
and over $Y$:
\[
1 + \sum_{k=1}^{3}b^{2,k}_0 w^k_2, \quad A_1 + \sum_{k=1}^{3} b^{2,k}_1 w^k_2, \quad A_1^2
\]
where $a^{2,k}_m \in H^0(X,\O_X)$ and $b^{2,k}_m \in H^0(Y,\O_Y)$. Using Lemma \ref{Lemma - Conditions for EC}  these sections must satisfy the following conditions on restriction to $X \cap Y$:
\begin{center}
\begin{tabular}{cc}
    \multicolumn{2}{c}{}\\
 $I_1: \; b^{2,k}_1 = a^{2,k}_1 $ & $I_0: \; b^{2,k}_0 = a^{2,k}_0 - Fb^{2,k}_1$\\
\end{tabular}
\end{center}

We conclude that $a^{2,k}_1=b^{2,k}_1 \in K$ are both constant. Since $F$ must have a simple pole at $\PP$, $a^{2,k}_0$ cannot have a simple pole at $\PP$ and $b^{2,k}_0$ must be regular at $\PP$ we conclude that $a^{2,k}_1=b^{2,k}_1=0$. Therefore, $a^{2,k}_0=b^{2,k}_0\in K$ are constant and following Step (6) of Algorithm \ref{Algorithm - Compute HF on HEC or EC} we conclude these sections must also be $0$. Therefore, $a^{2,k}_m=b^{2,k}_m=0$ for all $k$ and we are done. 
  \end{proof}

\begin{proposition} \label{Proposition - HF on UC of EC Level 3}
The generators of $F^0\UU_3$ over both $X$ and $Y$ are 

\[
1, \; A_1 + \lambda [[A_0,A_1],A_1], \;  A_1^2, \; A_1^3
\]
where $[A_i,A_j]=A_iA_j-A_jA_i$ is the commutator bracket and $\lambda$ is as in Proposition \ref{Proposition - GT on UC of EC Level 3}.
\end{proposition}

\begin{proof}
This is a straightforward application using Proposition \ref{Proposition - GT on UC of EC Level 3}, Algorithm \ref{Algorithm - Compute HF on HEC or EC}, Lemma \ref{Lemma - Conditions for EC} and the filtration computed in Proposition \ref{Proposition - HF on UC of EC Level 2}. 
  \end{proof}

\begin{remark}
It is at this stage that we see some dependency in the Hodge filtration on the choice of basis of $H^1_{dr}(X)$. If we let $\alpha_0=\frac{dx}{y}$ and $\alpha_1=\frac{xdx}{y}$ then we may take $F=\frac{y}{x}$. In this case $\lambda=-2$. 
\end{remark}

\begin{proposition} \label{Proposition - HF on UC of EC Level 4}
The generators of $F^0\U_4$ as an $\O_X$-module are

\[
1 +\frac{1}{3}\nu (x-x(b)) [A_1,[A_1,[A_1,A_0]]]
\]

\[
A_1 +\lambda [[A_0,A_1],A_1] -(\mu + \frac{1}{3}\kappa)[A_1,[A_1,[A_1,A_0]]]
\]

\[
A_1^2 +\lambda [[A_0,A_1],A_1^2], \; \; A_1^3, \; \; A_1^4
\]
where $\lambda, \mu$ are as in Proposition \ref{Proposition - GT on UC of EC Level 4} and $\kappa,\nu \in K$ are constants such that $\nu x+ \lambda F^2$ and $\nu x + \kappa F+\lambda F^2$ are regular at $\PP$. 
\end{proposition}

\begin{proof}
We again proceed by application of Proposition \ref{Proposition - GT on UC of EC Level 4} to compute $G_4$, Algorithm \ref{Algorithm - Compute HF on HEC or EC} and Lemma \ref{Lemma - Conditions for EC}. We determine sections $a^{4,k}_m \in H^0(X, \O_X)$ defining the generators as in the algorithm. These are displayed in the following table:

\begin{table}[h]
\centering
\begin{tabular}{ |p{1cm}|p{2.5cm}|p{3cm}|  }
\hline
$m=$&$k$&$a^{4,k}_m=$ \\
\hline \hline
$3$ & all $k$ & $0$ \\
\hline
 & $\neq 8,12,14,15$ & $0$ \\\cline{2-3}
 & $=8$ &  $\lambda$\\\cline{2-3}
$2$ & $=12$ & $-\lambda$ \\\cline{2-3}
 & $=14$ & $-\lambda$ \\\cline{2-3}
 & $=15$ & $\lambda$ \\
\hline
 & $\neq 8,12,14,15$ & $0$ \\\cline{2-3}
& $=8$ &  $\mu + \frac{1}{3}\kappa$\\\cline{2-3}
$1$ & $=12$ & $-3\mu - \kappa$ \\\cline{2-3}
 & $=14$ & $3\mu +  \kappa$ \\\cline{2-3}
 & $=15$ & $-\mu - \frac{1}{3}\kappa$ \\
\hline

& $\neq 8,12,14,15$ & $0$ \\\cline{2-3}
& $=8$ &  $-\frac{1}{3}\nu(x-x(b))$\\\cline{2-3}
$0$ & $=12$ & $ \nu(x-x(b))$ \\\cline{2-3}
& $=14$ & $- \nu(x-x(b))$ \\\cline{2-3}
& $=15$ & $\frac{1}{3} \nu(x-x(b))$ \\
\hline
\end{tabular}
\end{table}

It is then a simple case of expressing them in terms of the commutator brackets. 
  \end{proof}

\begin{remark}
If we again take $\alpha_0=\frac{dx}{y}, \alpha_1=\frac{xdx}{y}$ and $F= \frac{y}{x}$ then we will find that $\mu=\kappa=0$, $\lambda=2$ and $\nu=8$. \end{remark}

\subsection{An application of Theorem 4 to odd hyperelliptic curves}

In this section we focus on the case where $C$ is a complete odd hyperelliptic curve. We focus on providing a complete description of the generators of $F^0\U_2$. First we provide a more suitable restatement of the conditions $I_m$ appearing at level $2$.

\begin{lemma}\label{Lemma - Conditions for HEC}
The conditions $I_m$ of Lemma \ref{Lemma - Conditions on Lifts} when $n=1$ may be restated as follows:
\begin{itemize}
\item[$I_1$:] For $k \in \lbrace 1,..,(2g)^{n+1}\rbrace-\F_2$ and $f \in \F_1$ \[b_{1,f}^{2,k}=a_{1,f}^{2,k} + h^{k_1+1,1}_{2}\delta_{1+k_0,f}\]
\item[$I_0$:] For $k \in \lbrace 1,..,(2g)^{n+1}\rbrace-\F_2$ \begin{align*}
b_{0,1}^{2,k} = &a_{0,1}^{2,k} + h^{k,0}_2 - \sum_{f=g+1}^{2g} b^{2,k}_{1,f}h^{f,0}_1
\end{align*}
\end{itemize} where $k = 1 + k_0 + k_1(2g)$ for some $k_0,k_1 \in \lbrace 0,..,2g-1 \rbrace$. 
\end{lemma}

\begin{proof}
The proof of this lemma rests on a few simple computations. First we have $\psi(k,0,0,2g)= k$ by definition of $\psi$. Therefore, by definition of $\tau$ we find that $\tau(k,1,0,2g)=1$ since $k = 1 +(k-1)(2g)^0$. Second, observe that $k= 1+ k_0 + k_1 (2g)$ for some unique $k_0,k_1 \in \lbrace 0,..,2g-1 \rbrace$. It is easy then to check that $\psi(k,0,1,2g)=k_1+1$. Therefore, we find that $\tau(k,f,1,2g)=1$ if and only if $k=f+k_1(2g)$ and is $0$ otherwise. Putting all of these calculations together we deduce the formulation of the conditions $I_0$ and $I_1$ as in the statement of the lemma. 
  \end{proof}

Using this formulation of the conditions $I_1$ and $I_0$ we can prove the following proposition. 

\begin{proposition} \label{Proposition - HF on HEC Level 2}
The generators of $F^0\U_2$ as an $\O_X$-module are 
\begin{align*}
1 + \sum_{\substack{0 \leq i < g \\ g \leq j < 2g}}&a_{ij}[A_i,A_j]\\
A_{k} + \sum_{\substack{0 \leq i < g \\ g \leq j < 2g}}&c_{ijk}[A_i,A_j]\\
&A_rA_s
\end{align*}
where $k,r,s \in \lbrace g,..,2g-1 \rbrace$ and $a_{ij}, c_{ijk} \in H^0(X,\O_X)$ are sections such that

\begin{enumerate}
\item[(1)] $c_{ijk}$ are constant for all $i,j,k$
\item[(2)] $a_{ij}$ are such that \[ a_{ij} + h^{1+j+i(2g),0}_2 - \sum_{f=g+1}^{2g} c_{ijk}h^{f,0}_1\] are regular at $\PP$ and evaluate to $0$ at $b$. 
\end{enumerate}
\end{proposition}

\begin{proof}
We proceed by applying Algorithm \ref{Algorithm - Compute HF on HEC or EC} with the reformulation of the conditions as in Lemma \ref{Lemma - Conditions for HEC}. First we consider condition $I_1$. We need to examine the term $h^{k_1+1,1}_2\delta_{1+k_0,f}=h^{k_1+1,0}_1\delta_{1+k_0,f}$. Since $k \not \in \F_2$ and $f \in \F_1$ by Lemma \ref{Lemma - Hodge Basis for Vdr} $1+k_0=f$ can only occur when $k_1<g$. Hence $k_1+1 \in \lbrace 1,..,g \rbrace$. Therefore, when they appear, each $h^{k_1+1,1}_2$ must be regular at $\PP$ and hence each $a^{2,k}_{1,f}$ is regular at $\PP$ and hence constant on $C$. 

We now examine condition $I_0$. Observe that $c_{ijk}=a^{2,1+j+i(2g)}_{1,k+1}$ is the coefficient of $A_iA_j$ in the lift of the generator $A_k$ of $F^0\U_1$ and that $a_{ij}=a^{2,1+j+i(2g)}_{0,1}$ is the coefficient of $A_iA_j$ in the lift of $1$. To complete the proof of the proposition we, therefore, need to show:
\begin{align*}
c_{ijk}&=-c_{jik} \text{ for all } k,i,j \text{ and } c_{ijk}=0 \text{ if } i,j<g \\
a_{ij}&=-a_{ji} \text{ for all } i,j \text{ and } a_{ij}=0 \text{ if } i,j<g 
\end{align*} Condition $I_0$ for $k=1+j+i(2g),1+i+j(2g)$ are
\begin{align}\label{Proof - Proposition on HEC HF Equation 1}
b_{0,1}^{2,1+j+i(2g)} = &a_{ij} + h^{1+j+i(2g),0}_2 - \sum_{f=g+1}^{2g} b^{2,1+j+i(2g)}_{1,f}h^{f,0}_1 \\ 
b_{0,1}^{2,1+i+j(2g)} = &a_{ji} + h^{1+i+j(2g),0}_2 - \sum_{f=g+1}^{2g} b^{2,1+i+j(2g)}_{1,f}h^{f,0}_1.\label{Proof - Proposition on HEC HF Equation 2} 
\end{align}
By Algorithm \ref{Algorithm - Computing GT on UC of HEC} we know that $h^{1+j+i(2g),0}_2$ and $h^{1+i+j(2g),0}_2$ are such that 
\begin{align}\label{Proof - Proposition on HEC HF Equation 3}
&dh^{2gi+j+1,0}_2-\alpha_ih_1^{j+1,0}-c_1^{j+1,0}h^{i+1,0}_1\\
& dh^{2gj+i+1,0}_2-\alpha_jh_1^{i+1,0}-c_1^{i+1,0}h^{j+1,0}_1 \label{Proof - Proposition on HEC HF Equation 4}
\end{align} have at worst logarithmic poles at $\PP$. Here $c^{j+1,0}_1 = dh^{j+1,0}_1-\alpha_j$. 

Now if $i,j<g$ then $\alpha_i,\alpha_j \in \Omega^1_{C/K}$ are regular on $C$. Therefore, $h^{i+1,0}_1,h^{j+1,0}_1$ must be regular at $\PP$ and hence $h^{1+j+i(2g),0}_2$ and $h^{1+i+j(2g),0}_2$ must also be regular at $\PP$. Applying Algorithm \ref{Algorithm - Compute HF on HEC or EC} we conclude that when $i,j<g$ then $c_{ijk}=0$ for all $k$ and that $a_{ij}$ are constant. Since they must evaluate to $0$ at $b$ we conclude that they too are all $0$. 

Suppose now that $i< g$ and $j \geq g$. Adding conditions in (\ref{Proof - Proposition on HEC HF Equation 3}) and (\ref{Proof - Proposition on HEC HF Equation 4}) together we obtain $d(h^{2gi+j+1,0}_2+h^{2gj+i+1,0}_2-h^{i+1,0}_1h^{j+1,0}_1)$. Hence $h^{2gi+j+1,0}_2+h^{2gj+i+1,0}_2-h^{i+1,0}_1h^{j+1,0}_1$ must be regular at $\PP$. If we now add (\ref{Proof - Proposition on HEC HF Equation 1}) and (\ref{Proof - Proposition on HEC HF Equation 2}) together and use condition $I_1$ we obtain 
\begin{align*}
b_{0,1}^{2,1+j+i(2g)}+b_{0,1}^{2,1+i+j(2g)} = & a_{0,1}^{2,1+j+i(2g)}+a_{0,1}^{2,1+i+j(2g)} - \sum_{f=g+1}^{2g} (a^{2,1+j+i(2g)}_{1,f}+a^{2,1+i+j(2g)}_{1,f})h^{f,0}_1 \\
& +h^{1+j+i(2g),0}_2+h^{1+i+j(2g),0}_2-h^{i+1,0}_1h^{j+1,0}_1.
\end{align*}

As before, we conclude that we must have $c_{ijk}+c_{jik}=0$ and $a_{ij}+a_{ji}$ constant for all $i <g, j \geq g$. Since $a_{ij}+a_{ji}$ must evaluate to $0$ at $b$ we conclude that this too is exactly $0$. Thus, we have shown that the lifts satisfy the conditions claimed in the statement of the proposition.  
  \end{proof}

We conclude this section with an explicit example of the above lemma for genus $2$ odd hyperelliptic curves. 

\begin{example} \label{Example - HF on Genus 2 HEC}
Let $C$ be a hyperelliptic curve of genus $2$ over $K$ with odd affine model $y^2=f(x)$ for some $f(x) \in K[x]$ where $\deg f =5$. Let us take $\alpha_i : = \dfrac{x^idx}{y}$, the standard $K$-basis for $H^1_{dr}(X/K)$. Let $F:=\frac{y}{x^2} \in K(C)$ and note that it has a logarithmic pole at $\PP$. Applying Algorithm \ref{Algorithm - Computing GT on UC of HEC} and Algorithm \ref{Algorithm - Compute HF on HEC or EC} we conclude that $F^0\U_2$ has the following generators as an $\O_X$-module lifting the generators of $F^0\U_1$:

\[
1-\frac{2}{3}(x-x(b))[A_1,A_3],\; A_2,\; A_3,\; A_2^2,\; A_2A_3,\; A_3A_2,\; A_3^2
\]
\end{example}

\section{Computing the de Rham period maps} \label{Section 5}

This section is concerned with explicit computation of the $p$-adic de Rham period maps introduced by Kim in \cite{kim09}. As outlined in loc.cit. the coordinates of these maps can be described in terms of explicit $p$-adic iterated Coleman integrals. We will make use of the algorithms and results from Sections \ref{Section 3} and \ref{Section 4} to compute coordinates of some of these maps at new levels as well as providing a conjectural closed form for the map on elliptic curves. First, we recall how the period map is defined. 
\subsection{Defining the de Rham period map $j^{dr}_n$}

Throughout this section we shall assume the following: $K$ is a number field, $v$ is a non-Archimedean valuation on $K$ and $K_v$ is the completion of $K$ with respect to $v$. Let $C$ be an elliptic or odd hyperelliptic curve and define $X:=C- \lbrace \PP \rbrace$. We let $X_v:=X \otimes K_v$ denote the basechange of $X$. Take a rational basepoint $b \in X(K)$ and suppose that $x \in X_v(K_v)$. We wish to find explicit descriptions for the maps 
\index{$j^{dr},j^{dr}_n$}
\[
j^{dr}_n: X_v(K_v) \rightarrow F^0U^{dr}_n \setminus U^{dr}_n.
\]

Recall from Section \ref{Section 2} that in order to compute $j^{dr}_n(x)$ we need to identify the following trivialisations over $K_v$: a Frobenius trivialisation $p^{cr}_n(x) \in P^{dr}_n(x)$; a Hodge trivialisation $p^H_n(x) \in F^0P^{dr}_n(x)$; and a trivialisation $u_n(x) \in U^{dr}_n$ such that $p^{cr}_n(x)=p^H_n(x)u_n(x)$. We know that \index{$p_n^{cr}(x)$}
\begin{equation}\label{Equation - pcr(x)}
p_n^{cr}(x) = 1 + \sum_{|w| \leq n} \int_b^x \alpha_w w
\end{equation}
where the sum is taken over all words of length at most $n$. We have shown in Section \ref{Section 4} that $\U_n$ possesses a Hodge filtration $F^{\bullet}\U_n$. This induces a filtration on the dual bundles $\U_n^{\vee}$: define $F^{i}(\U_n^{\vee}):=(F^{1-i}\U_n)^{\perp}= \lbrace l \in \U_n^{\vee} | F^{1-i}(\U_n) \subseteq \ker l \rbrace$. Recall that $F^{-1+i}\mathcal{P}^{dr}_n$ is the defining ideal for $F^iP^{dr}_n$. Therefore, $F^iP^{dr}_n$ has defining ideal $(F^{i}\U_n)^{\perp}$, or $F^iP^{dr}_n = \Spec(\U_n^{\vee}/(F^{i}\U_n)^{\perp})$. Therefore, arguing as in Proposition \ref{Proposition - Coordinate ring of Pdr} the group-like elements of $x^*F^i\U_n$ correspond to elements of $F^iP^{dr}_n(x)$. We utilise this property in what follows. 

The first explicit description of the de Rham period maps beyond level $1$ was given in \cite{kkb11} for $n=2$ when $C$ is an elliptic curve.

\begin{proposition}[\cite{kkb11}, Proof of Corollary 0.2']
The level $2$ unipotent Albanese map $j^{dr}_2: X_v(K_v) \rightarrow F^0U^{dr}_2 \setminus U^{dr}_2$ is defined by
\begin{equation} \label{UAM - Elliptic Curve - Level 2}
x \mapsto \int_b^x \alpha_0 A_0 + \int_b^x \alpha_0\alpha_1 [A_0,A_1].
\end{equation}
\end{proposition}

\begin{proof}
The point is that we know $1,A_1,A_1^2$ generate $F^0\U_2$ using Proposition \ref{Proposition - HF on UC of EC Level 2}. It is a simple computation to check that 
\[
p^{cr}_2(x) = \exp\left(\int_b^x \alpha_1 A_1\right) \exp\left(  \int_b^x \alpha_0 A_0 + \int_b^x \alpha_0\alpha_1 [A_0,A_1]\right).
\]

Note that $\int_b^x \alpha_0 A_0 + \int_b^x \alpha_0\alpha_1 [A_0,A_1]$ is primitive since linear combinations and commutators of primitive elements are primitive. We conclude that we have found the decomposition $p^{cr}_2(x)=p^H_2(x)u_2(x)$ as required. 
  \end{proof}

\begin{remark}
Observe that the map of the above proposition is actually the logarithm of $j^{dr}_2$. As per Remark \ref{Remark - Exp and Log and Grouplike}, however, there is a bijection between $\text{Lie}U^{dr}_2$ and $U^{dr}_2$ coming from the $\exp$ and $\log$ maps. For the sake of brevity it will be more convenient here and in what follows to present the image of the map in $\text{Lie}U^{dr}_n$. However, we will suppress this in the notation, making implicit use of the previously stated bijective correspondence. 
\end{remark}

\begin{remark}
Note that in order to compute the image of the map $j^{dr}_2$ for specific values of $x \in X_v(K_v)$ we need to compute the iterated integrals $\int_b^x \alpha_0$ and $\int_b^x \alpha_0 \alpha_1$. These can be computed using the algorithms of Balakrishnan in \cite{balakrishnan10}; this involves computation of the matrix of action of Frobenius on $H^1_{dr}(X/K)$ by making use of Kedlaya's algorithm and solving a system of linear equations to compute the integrals between Teichm\"{u}ller points.
\end{remark}

In what follows we will make use of two properties of iterated integrals (\cite[Proposition 5.2.1]{balakrishnan11}):

\begin{equation} \label{Iterated Integrals Property 1}
\int \omega_1..\omega_r \int \omega_{r+1}..\omega_s = \sum_{\sigma \in S(r,s)}\int \omega_{\sigma(1)}..\omega_{\sigma(s)}
\end{equation}

\begin{equation} \label{Iterated Integrals Property 2}
\int \underbrace{\omega\omega...\omega}_\text{$n$ times} = \frac{1}{n!}\left(\int \alpha \right)^n
\end{equation}

where in (\ref{Iterated Integrals Property 1}) we sum over all permutations of shuffle type $(r,s)$. Note that (\ref{Iterated Integrals Property 2}) is a simple consequence of (\ref{Iterated Integrals Property 1}). 

\subsection{The levels $3$ and $4$ period maps on affine elliptic curves}

In this section we will compute the levels $3$ and $4$ maps on affine elliptic curves using the results of the previous two sections. Recall from Proposition \ref{Proposition - HF on UC of EC Level 3} the Hodge filtration of $F^0\U_3$ is generated by 
\[
1, \; A_1 + \lambda [[A_0,A_1],A_1], \;  A_1^2, \; A_1^3.
\]

We want to find a primitive $p^H_3(x)$ in the $K_v$-algebra generated by these, and a primitive $u_3(x) \in x^*\U_3$ such that $p^{cr}_3(x)=\exp(p^H_3(x))\exp(u_3(x))$. We may then define $j^{dr}_3(x)=u_3(x)$. 

\begin{proposition} \label{Prop - Level 3 Map EC}
The level $3$ unipotent Albanese map $j^{dr}_3: X_v(K_v) \rightarrow F^0U^{dr}_3 \setminus U^{dr}_3$ is defined by 

\[
\begin{split}
x \mapsto u_3(x):=& \int_b^x \alpha_0A_0  + \int_b^x \alpha_0\alpha_1 [A_0,A_1]+\frac{1}{2}\int_b^x\alpha_0\alpha_1\alpha_0[A_0,[A_1,A_0]]  \\
	& + \int_b^x( \alpha_0\alpha_1\alpha_1- \lambda\alpha_1)				[[A_0,A_1],A_1]
\end{split}
\]

\end{proposition} 

\begin{proof} Using properties (\ref{Iterated Integrals Property 1}),(\ref{Iterated Integrals Property 2}) of iterated integrals and by comparing $p^{cr}_3(x)$ to 
\[
\exp\left(\int_b^x (\alpha_1 A_1+\lambda [[A_0,A_1],A_1])\right) \exp\left(  \int_b^x \alpha_0 A_0 + \int_b^x \alpha_0\alpha_1 [A_0,A_1]\right)
\]
we can rewrite $p^{cr}_3(x)$ as follows

\begin{align*}
p^{cr}_3(x)=& 1 + \int_b^x \alpha_0 A_0 + \int_b^x \alpha_1 A_1 + \frac{1}{2}\left(\int_b^x \alpha_0\right)^2A_0^2 + \int_b^x \alpha_0\alpha_1A_0A_1\\
& +\left( \int_b^x \alpha_1\int_b^x\alpha_0 - \int_b^x\alpha_0\alpha_1 \right)A_1A_0+ \left( \int_b^x \alpha_1 \int_b^x \alpha_0\alpha_1 - 2 \lambda \int_b^x \alpha_1 \right) A_1A_0A_1 \\
& + \left( \frac{1}{2}\left(\int_b^x \alpha_0\right)^2 \int_b^x \alpha_1 - \frac{1}{2}\int_b^x \alpha_0 \int_b^x \alpha_0\alpha_1 \right)A_1A_0^2 \\
&  + \left(\frac{1}{2}\left(\int_b^x \alpha_1 \right)^2 \int_b^x \alpha_0 - \int_b^x \alpha_1 \int_b^x \alpha_0\alpha_1+\lambda \int_b^x \alpha_1 \right) A_1^2A_0 + \frac{1}{3!}\left( \int_b^x\alpha_1\right)^3 A_1^3 \\
&+ \frac{1}{2}\int_b^x \alpha_0\alpha_1\alpha_0 [A_0,[A_1,A_0]] + \left( \int_b^x \alpha_0\alpha_1\alpha_1 - \lambda \int_b^x \alpha_1 \right) [[A_0,A_1],A_1] \\
 & 	=\exp \left( \int_b^x \alpha_1 (A_1+\lambda[[A_0,A_1],A_1])\right)  \exp (u_3(x))
\end{align*}
where $u_3(x)$ is as in the statement of the proposition. Since $A_0,A_1$ are primitive and commutators of primitives are primitive we are done.
  \end{proof}

To compute the level $4$ map we try to mimic the approach taken at level $4$ by comparing a modification of the decomposition at level $3$ to the Frobenius invariant element at level $4$. 
\begin{proposition} \label{Prop - Level 4 Map EC}
The level $4$ unipotent Albanese map $j^{dr}_4: X_v(K_v) \rightarrow F^0U^{dr}_4 \setminus U^{dr}_4$ is defined by 
\[
\begin{split} 
x \mapsto u_4(x):= &\int_b^x \alpha_0A_0  + \int_b^x \alpha_0\alpha_1 ([A_0,A_1]+ \frac{\lambda}{2} [[A_0,[A_0,A_1]],A_1]) +\frac{1}{2}\int_b^x\alpha_0\alpha_1\alpha_0[A_0,[A_1,A_0]]  \\
	& + \int_b^x (\alpha_0\alpha_1\alpha_1- \lambda\alpha_1)				[[A_0,A_1],A_1] +\frac{1}{6}\int_b^x \alpha_0 \alpha_0 \alpha_1 \alpha_0 [[A_0,[A_0,A_1]],A_0]\\
& + \frac{1}{6}\int_b^x \alpha_0 \alpha_1 \alpha_0 \alpha_0 [[[A_0,A_1],A_0],A_0] + \frac{1}{2} \int_b^x (\alpha_0\alpha_1 \alpha_0 \alpha_1 - \lambda  \alpha_1 \alpha_0) [[A_0,[A_1,A_0]],A_1] \\
&  +  \int_b^x (\alpha_0\alpha_1 \alpha_1 \alpha_1 - \lambda  \alpha_1 \alpha_1 - (\mu + \frac{ \kappa}{3})  \alpha_1 )[[[A_0,A_1],A_1],A_1]\\
& + \frac{1}{2} \int_b^x \alpha_0\alpha_1 \alpha_0 \alpha_1 [[[A_0,A_1],A_1],A_0]
\end{split}
\]
where $\lambda, \mu, \kappa$ are as in Proposition \ref{Proposition - HF on UC of EC Level 4}. 

\end{proposition}  

\begin{proof}
We proceed by a comparison of $p^{cr}_4(x)$ with 

\[
\exp\left( \int_b^x \alpha_1(A_1 +\lambda [[A_0,A_1],A_1] -(\mu + \frac{1}{3} \kappa)[A_1,[A_1,[A_1,A_0]]])\right) \exp(u_3(x)).
\]

A computation similar to that in the proof of Proposition \ref{Prop - Level 3 Map EC} shows that 
\[
p^{cr}_4(x) = \exp\left( \int_b^x \alpha_1 \left( A_1 + \lambda [[A_0,A_1],A_1] -(\mu + \frac{\kappa}{3})[A_1,[A_1,[A_1,A_0]]] \right) \right) \exp(u_4(x)) 
\]
where $u_4(x)$ is as in the statement of the proposition. Therefore since the coefficient of $\int_b^x \alpha_1$ is primitive and a generator of $F^0\U_4$ by Proposition \ref{Proposition - HF on UC of EC Level 4} we are done. 
  \end{proof}

\subsection{A useful lemma to compute period maps}

The examples of the previous subsection suggest the following useful lemma to compute the maps $j^{dr}_{n+1}$. Of course, this will rely on us having already computed $F^0\U_{n}$ using Algorithm \ref{Algorithm - Compute HF on HEC or EC}, as well as the co-ordinates of the map $j^{dr}_{n-1}$. We suppose we have already computed some decomposition $p^{cr}_{n-1}(x)=p^H_{n-1}(x)u_{n-1}(x)$. Provided a certain lifting condition is satisfied, we can describe the decomposition at level $n$ exactly. 

\begin{lemma}
Suppose that $B_{n-1}$ is primitive in $F^0\U_{n-1}$ and that $p^{cr}_{n-1}(x)=\exp(\int_b^x \alpha B_{n-1}(x))\exp(j^{dr}_{n-1}(x))$ for some $\alpha$ a linear combination of $\alpha_w$. Then suppose that the lift $B_n$ of $B_{n-1}$ to $F^0\U_n$ is also primitive. Then \index{$j^{dr},j^{dr}_n$}$j^{dr}_n$ has coordinates
\[ 
j^{dr}_{n}(x):=j^{dr}_{n-1}(x)+p ^{cr}_{n}(x)- \exp\left( \int_b^x \alpha B_{n}(x)\right)\exp(j^{dr}_{n-1}(x)).
\]
\end{lemma}
\begin{proof}
Suppose that we have a decomposition
\begin{equation}\label{Equation - computing j^dr_n eqn 1}
p^{cr}_{n-1}(x) = \exp\left(\int_b^x \alpha B_{n-1}(x)\right)\exp(j^{dr}_{n-1}(x)).
\end{equation}
As we did in the proofs of Propositions \ref{Prop - Level 3 Map EC} and \ref{Prop - Level 4 Map EC} consider the difference
\begin{equation}\label{Equation - computing j^dr_n eqn 2}
\begin{split}
p & ^{cr}_{n}(x) - \exp\left(\int_b^x \alpha B_{n}(x)\right)\exp(j^{dr}_{n-1}(x)) \\
=  p & ^{cr}_{n}(x) - \exp\left(\int_b^x \alpha \left(B_n(x)-B_{n-1}(x)\right)\right)\exp\left(\int_b^x \alpha B_{n-1}(x)\right)\exp(j^{dr}_{n-1}(x)). \\
\end{split}
\end{equation}

Since we have the previously stated decomposition (\ref{Equation - computing j^dr_n eqn 1}) at level $n-1$ then \[\exp(\int_b^x \alpha B_{n-1}(x))\exp(j^{dr}_{n-1}(x)) = p^{cr}_{n-1}(x)+\text{ words of degree }n\]. Therefore (\ref{Equation - computing j^dr_n eqn 2}) contains only words of length $n$. Therefore \[\exp(p ^{cr}_{n}(x) - \exp(\int_b^x \alpha B_{n}(x))\exp(j^{dr}_{n-1}(x)) )\] commutes with $\exp(\int_b^x \alpha B_{n}(x))\exp(j^{dr}_{n-1}(x))$ and we find that 

\begin{align*}
&\exp\left(\int_b^x \alpha B_{n}(x)\right)\exp\left(j^{dr}_{n-1}(x)+p ^{cr}_{n}(x)- \exp\left( \int_b^x \alpha B_{n}(x)\right)\exp(j^{dr}_{n-1}(x)) \right) \\
=&  \exp\left(\int_b^x \alpha B_{n}(x)\right)\exp\left(j^{dr}_{n-1}(x)\right)\exp \left(p ^{cr}_{n}(x)- \exp\left( \int_b^x \alpha B_{n}(x)\right)\exp(j^{dr}_{n-1}(x))\right)\\
 =&  \left(p^{cr}_{n}(x) + \exp\left(\int_b^x \alpha B_{n}(x)\right)\exp(j^{dr}_{n-1}(x)) - p^{cr}_{n}(x) \right)\\
& \times  \exp \left(p ^{cr}_{n}(x)- \exp\left( \int_b^x \alpha B_{n}(x)\right)\exp(j^{dr}_{n-1}(x))\right)\\ 
 =&  p^{cr}_{n}(x) + \exp\left(\int_b^x \alpha B_{n}(x)\right)\exp(j^{dr}_{n-1}(x)) - p^{cr}_{n}(x) + p ^{cr}_{n}(x) \\
& - \exp\left( \int_b^x \alpha B_{n}(x)\right)\exp(j^{dr}_{n-1}(x)) \\
 =&  p^{cr}_{n}(x)
\end{align*}
This gives us a decomposition of the form
\[
p^{cr}_{n}(x)=\exp\left(\int_b^x \alpha B_{n}(x)\right)\exp\left(j^{dr}_{n-1}(x)+p ^{cr}_{n}(x)- \exp\left( \int_b^x \alpha B_{n}(x)\right)\exp(j^{dr}_{n-1}(x)) \right)
\]
Since $B_n$ is primitive then $\exp(\int_b^x \alpha B_n(x))$ is group-like and we conclude that we can define $j^{dr}_n$ as in the statement of the theorem. 
\end{proof}

\subsection{The level 2 period map on affine hyperelliptic curves}
We can use the computation of the Hodge filtration on $\U_2$ for a generic genus $g$ odd hyperelliptic curve $C$ that we determined in Proposition \ref{Proposition - HF on HEC Level 2} to determine the level $2$ de Rham period map on $X:=C- \lbrace \PP \rbrace$. Fix a basepoint $b \in X(K)$. As above, we will find that

\[
p^{cr}_2(x) = 1 + \sum _{k=0}^{2g-1} \int^x_b\alpha_k A_k + \sum _{k,l = 0}^{2g-1} \int_b^x \alpha_k\alpha_l A_kA_l
\]

\begin{proposition} \label{Prop - level 2 map hec}
The map $j^{dr}_2: X_v(K_v) \rightarrow F^0U^{dr}_2 \setminus U^{dr}_2$ is defined by 

\begin{align*}
x \mapsto u_2(x):=& \sum_{k=0}^{g-1} \int_b^x \alpha_k A_k + \frac{1}{2} \sum_{k,l=0}^{g-1} \int_b^x \alpha_k\alpha_l [A_k,A_l] + \sum_{k=0}^{g-1}\sum_{l=g}^{2g-1} \int_b^x \alpha_k \alpha_l [A_k,A_l]\\
& + \sum_{k=g}^{2g-1} \sum_{\substack{ 0 \leq i < g \\ g \leq j < 2g}} c_{ijk}\int_b^x \alpha_k[A_j,A_i] 
\end{align*}

\end{proposition}

\begin{proof}

A careful computation shows that 

\begin{align*}
p^{cr}_2(x)=& \exp\left( \sum_{k=g}^{2g-1} \int_b^x \alpha_k \left(A_k + \sum_{\substack{ 0 \leq i < g \\ g \leq j < 2g}} a^{2,1+j+i(2g)}_{1,k+1} [A_i,A_j]\right) + \frac{1}{2} \sum_{l,k=g}^{2g-1} \int_b^x \alpha_k\alpha_l [A_k,A_l]\right)\\ 
& \times \exp(u_2(x))
\end{align*}
where $u_2(x)$ is as in the statement of the proposition and we are done. 
  \end{proof}

\section{De Rham period maps with tangential basepoints} \label{Section 6}

\subsection{Computing iterated integrals with tangential basepoints}
Throughout Sections \ref{Section 2} to \ref{Section 5} we have assumed that the basepoint $b$ is a $K$-rational point on $X$. It may be the case that we cannot readily determine such a rational basepoint. Indeed part of the value of studying the unipotent Albanese map is that they should give us the means to compute the rational or $S$-integral points on a curve (for some finite set of primes $S$). Therefore, we need to find a way to  circumvent the obstacle of finding a rational basepoint and this is afforded to us by the theory of tangential basepoints or basepoints at infinity. These basepoints correspond to tangent vectors at the points of $D$ and philosophically speaking they should give us access to a wider range of canonical maps. To replace the rational basepoint $b$ in the maps $j^{dr}_n$ with a tangential basepoint will require a method of ``analytic continuation" to the tangential basepoint. We briefly outline the method by which we can compute this, following the description given by Besser and Furusho in \cite{besfur06}. 

We assume once more that $C$ is a smooth projective curve over $K$ a number field punctured at a single point $P$ defined over $K$ and let $X:= C- \lbrace P \rbrace$. Let $v$ be a non-Archimedean valuation on $K$ and take $K_v$ the completion of $K$ at $v$. Let $\U$ be the universal unipotent connection on $X_v$, and $\UU$ its logarithmic extension to $C_v$. Let \index{$t$}$t$ be a local parameter at $P$ inducing a parameter $\overline{t}$ on \index{$T_{P}C$}$T_{P}C$, which is a normal vector at $P$ taking the value $1$ at the tangent vector $b=\frac{d}{dt}$. Let \index{$T^0_{P}C$}$T^0_{P}C:=T_{P}C-\lbrace 0 \rbrace$. Then we can associate to each logarithmic connection $\overline{\mathcal{V}}$ on $C$ with logarithmic poles at $P$ a connection on $T^0_{P}C$:

\begin{definition}
Define \index{$\Res_P\,\overline{\mathcal{V}}$}$\Res_P\,\overline{\mathcal{V}}:= (P^*\overline{\mathcal{V}} \otimes \O_{T^0_PC}, d+\Res_P\Omega d\log \overline{t})$ where $\Omega=(\omega_{ij})$ is the connection matrix of $\overline{\mathcal{V}}$ near $P$ and $\Res_P\Omega= (\res_P\omega_{ij})$ is the residue matrix of $\Omega$ at $P$.
\end{definition}

In \cite[\S 15, Th\'{e}orie alg\'{e}brique]{del87} Deligne shows that the definition of $\Res_P$ does not depend on the choice of local parameter $t$ at $P$. The map $\Res_P$ gives us a functor $\Res_P : \Un(X) \rightarrow \Un(T^0_PC)$ which associates to a unipotent connection $\mathcal{V}$ on $X$ the residue connection $\Res_P\overline{\mathcal{V}}$ of the canonical logarithmic extension of $\mathcal{V}$. Using this construction we obtain more fibre functors $P$ which we now define: 

\begin{definition}
Let $y \in T^0_PC(K_v)$ and let $\overline{\mathcal{V}}$ be a logarithmic connection on $C$ with logarithmic poles at $P$. Then the fibre functor $e_y: \Un(X) \rightarrow \text{Vect}_{K_v}$ at $y$ is such that $\mathcal{V} \mapsto y^*\Res_P \overline{\mathcal{V}}$ where $\overline{\mathcal{V}}$ is the canonical logarithmic extension of $\mathcal{V}$. 
\end{definition} 

The existence of a canonical Frobenius invariant de Rham path between any two points $x,y$ tangential or otherwise was demonstrated in \cite[Theorem 4.1]{besfur06}. In loc.cit. Proposition 4.5 they identify the path $p_{x,y}$ from $x$ to $y$ for given points $x \in X_v(K_v)$ and $y \in T_P^0C(K_v)$ as being the constant term of a formal local solution to the differential equation defined by the logarithmic connection $\overline{\mathcal{V}}$. We now outline this construction in further detail. 

There is a basis of global solutions to $\Res_P \overline{\mathcal{V}}=0$ with coefficients in $k[\log \overline{t}]$, with all solutions having the form $\exp( \Res_P\, \Omega \log \overline{t})\cdot g$ with $g \in P^*\overline{\mathcal{V}}$. The exponential will be finite since $\overline{\mathcal{V}}$ and hence $\Res_P\,\overline{\mathcal{V}}$ are unipotent. We can also find formal local horizontal solutions $s$ of  $\overline{\mathcal{V}}$ near $P$ with coefficients in the ring $K_v[[t]][\log t]$. Here $\log t$ is treated as a formal variable with the property that $d \log t = dt/t$. 

To analytically continue the horizontal section $s$ near $P$ to the tangent vector $\frac{d}{dt}$ we specialise to the fibre $P^*\overline{\mathcal{V}} \otimes \O_{T^0_PC}$ by taking the constant term of this formal solution: that is, we let $t=\log t =0$. Call this constant term $s_0$. In Proposition 4.5, Besser and Furusho show that the path $p_{x,y}$ is the path  $\overline{\mathcal{V}} \mapsto ( s \mapsto \exp(\Res_P C'_n \log \overline{t})\cdot s_0) )$. The specialisation of $\exp(\Res_P \Omega \log \overline{t})\cdot s_0$ at $\overline{t}=1$ then is the analytic continuation along Frobenius of the horizontal section $s$ to the punctured tangent space, and this is precisely the constant term $s_0$. We recast this below as a computational algorithm:

\begin{algorithm}\label{Algorithm - Computing Integrals with TB}
(Computing iterated integrals with tangential basepoints on $X=C-D$) \newline

\textbf{Input}
\begin{itemize}
\item A smooth projective curve $C$ over a number field $K$ of characteristic $0$, a $K$-point $P$, $X:=C-\lbrace P \rbrace$ with affine model $f(x,y)=0$. 
\item A non-Archimedean valuation $v$ on $K$ and the completion $K_v$ of $K$ at $v$. 
\item A tangential basepoint $b$ and $x \in X_v(K_v)$.
\item Differentials $\omega_1,..,\omega_n \in \Omega^1_{C/K_v}(P)$ with at worst logarithmic poles at $P$.  
\end{itemize}

\textbf{Output}

\begin{itemize}
\item The value of 
\[
\int_b^x \omega_1...\omega_n \in K_v
\]
\end{itemize}

\textbf{Algorithm}
\begin{enumerate}
\item[I] If $x \in ]P[$ (the residue disk of $P$):
\begin{enumerate}
\item[(1)] Let $\partial_b$ be the derivation associated to $b$ and let $t$ be a local parameter at $P$ such that $\partial_bt=1$. 
\item[(2)] Let $\sigma$ be a dummy variable and define 
\begin{align*}
\epsilon &= (x(\sigma), y(\sigma)) \\
z&=t(x) \\
\omega_i(t)&= f_i(t)dt
\end{align*}
where $f_i(t)dt$ is the Laurent expansion of $\omega_i$ at $P$ in $t$ for each $i$. 
\item[(3)] Compute the formal iterated integral $\int_{\epsilon}^x \omega_1..\omega_n$:
\[\int_{\sigma}^z f_1(t_1) \left( \int_{\sigma}^{t_1} f_2(t_2)...\left( \int_{\sigma}^{t_{n-1}}f_{r-1}(t_{n-1}) \left(\int_{\sigma}^{t_n}f_n(t_n)dt_n \right)dt_{n-1} \right)...dt_2\right)dt_1
\] with output the logarithmic Coleman function
\[
a^x_0(\sigma) + a^x_1(\sigma)\log(\sigma) + a^x_2(\sigma) \log(\sigma)^2+....
\]
where the $a^x_i(\sigma)$ analytic functions in the variable $\sigma$.
\item[(4)] Define
\[
\int_b^x \omega_1...\omega_n := a^x_0(0)
\] i.e. set $\sigma=\log \sigma = 0$ in the output of the previous step. 
\end{enumerate}
\item[II] Else $x \not \in ]P[$: 
\begin{enumerate}
\item[(1)] Choose $y \in ]P[$ and do:
\begin{enumerate}
\item[(a)] For $i=1,..,n$ compute
\[
\int_y^x \omega_1...\omega_i
\]
using Coleman integration. 
\item[(b)] For $i=1,..,n$ compute
\[
\int_b^y \omega_i...\omega_n
\]
using Step I.
\end{enumerate}
\item[(2)] Define 
\[
\int_b^x \omega_1...\omega_n := \sum_{i=0}^n \int_y^x \omega_1..\omega_i \int_b^y \omega_{i+1}..\omega_n
\]
where the empty integral is defined to be $1$. 
\end{enumerate}
\end{enumerate}
\end{algorithm}

If we had chosen a different tangential basepoint $b'$ then we need to make a different choice of normalising parameter $t'$ with $\partial_{b'}t'=1$, where $\partial_{b'}$ is the derivation associated to $b'$. In \cite[Lemma 3.2]{bb15} it is shown that the above definition of the integral is independent of the choice of parameter $t'$ satisfying this normalisation condition. For points $x$ lying outside $]P[$ we made use the following co-product formula for iterated integrals \cite[Lemma 5.2.3]{balakrishnan11}: if $x,y,z$ are points on $C$ such that that a path is to be taken from $x$ to $z$ through $y$ and $\omega_1,...,\omega_n$ are holomorphic $1$-forms at these points then 

\[
\int_x^z \omega_1..\omega_n = \sum_{i=0}^n \int_y^z\omega_1...\omega_i\int_x^y \omega_{i+1}..\omega_n
\]

It is clear from the definition of the integral $\int_b^x \omega_1...\omega_n$ that the co-product formula above still holds even if one of the endpoints of the path is a tangent vector. Hence, if we want to compute $\int_b^x \omega_1...\omega_n$ with $x$ outside of $]P[$ we find a point $y \in ]P[$ and split the path at $y$. The co-product formula then gives us integrals of the form 

\[
\int_y^x \omega_1...\omega_i, \int_b^y \omega_{i+1}...\omega_n
\]

where the first integral is computed using the algorithms in loc.cit. and the second integral is computed using Algorithm \ref{Algorithm - Computing Integrals with TB}. Note also that if the differentials $\omega_i$ are regular also at $P$ then in fact the above algorithm simply gives us 

\[
\int_{P}^x \omega_1..\omega_n
\]

We wish to make applications of this to computing the de Rham period maps. If we replace the rational basepoint $b$ with a tangential basepoint we obtain a de Rham path space $\pi_{1,dr}(X; b, x) = \Isom^{\otimes}(e_b,e_x)$ for $x \in X_v(K_v)$. This has a Hodge structure which is a limit Hodge structure of that on $\pi_{1,dr}(X; y,x)$ as $y$ varies over $X(K_v)$ (\cite{kim10}). As noted in loc.cit. it is sufficient to compute the Hodge filtration at rational basepoints. 

Instead, we need to consider the logarithmic connection $\UU_n$ near $\PP$: the restriction to the open $Y$ is $\U'_n$ which is a logarithmic connection with connection matrix which can be computed using the algorithms of Section \ref{Section 3}. Near $\PP$ we can compute the parallel transport of $1 \in b^*\Res_{\PP}\UU_n$ to $x^*\UU_n$ for an $x \in Y$ using Proposition 4.5 in loc.cit. as described above. 

Let us consider how we do this in practice. Assume for now that $C$ is an elliptic curve or odd hyperelliptic curve, that $P$ is the point at infinity $\PP$ and fix a branch of the $p$-adic logarithm. The connection $\U'_n$ is given  by $\nabla'=d + C'_n$ on $Y$ as computed by Algorithm \ref{Algorithm - Computing GT on UC of HEC}. This connection matrix then defines an iterated integral which gives the parallel transport of $1 \in b^*\Res_{\PP}\UU_n$ to $x^*\UU_n$ as before. By computing a decomposition into a product of the Hodge trivialisation and a trivialisation of $F^0U^{dr}_n\setminus U^{dr}_n$ we may then describe the coordinates of the map $j^{dr}_n$ by an iterated integral with tangential basepoint. Finally we use Algorithm \ref{Algorithm - Computing Integrals with TB} to compute the image of this map for $x \in X_v(K_v)$.  

\subsection{The levels $2$ and $3$ period maps on elliptic curves}

In this section we provide some explicit examples of the period maps on elliptic curves when the basepoint is tangential. First we consider the level $2$ map which was considered in \cite{kkb11}. Throughout this section recall that we take $F \in K(C)$ with a simple pole at $\PP$ such that $dF-\alpha_1$ is regular at $\PP$. 

\begin{proposition} \label{Prop - Level 2 TB Map EC}
With $b$ a tangential basepoint at $\PP$ on the elliptic curve $X$, the level $2$ unipotent Albanese map $j^{dr}_2: X_v(K_v) \rightarrow F^0U^{dr}_2 \setminus U^{dr}_2$ on points $x \in X_v(K_v) \cap Y_v(K_v)$ is given by

\[
x \mapsto \int_b^x \alpha_0 A_0+ \int_b^x(F\alpha_0+\alpha_0\alpha_1')[A_0,A_1]
\]

where $\alpha_1' = \alpha_1-dF$. 

\end{proposition}

\begin{proof}
In Proposition \ref{Proposition - GT for EC Level 2} we computed that the connection matrix of logarithmic extension $\UU_2$ over $Y$ can be taken to be 

\[
C_2' = \left( 	\begin{array}{ccccccc}
					0 & 0 & 0 & 0 & -\alpha_0 & 0 & 0 \\
					0 & 0 & 0 & 0 & 0 & -\alpha_0 & -F \alpha_0\\
					0 & 0 & 0 & 0 & -\alpha_1' & 0 & F \alpha_0 \\
					0 & 0 & 0 & 0 & 0 & -\alpha_1' & 0 \\
					0 & 0 & 0 & 0 & 0 & 0 & -\alpha_0 \\
					0 & 0 & 0 & 0 & 0 & 0 & -\alpha_1' \\
					0 & 0 & 0 & 0 & 0 & 0 & 0
				\end{array}\right)
\]

where $\alpha_1'=\alpha_1-dF$ for $F \in K(C)$ with a simple pole at $\PP$ such that $\alpha_1'$ has at most a simple pole there also. Let $t$ be a local parameter at $\PP$. The parallel transport map from $1$ at $b=\dfrac{d}{dt}$ to $x \in Y_v(K_v)$ then is given by

\begin{align*}
p^{cr}_2(x)  =&  1+ \int_b^x \alpha_0 A_0 + \int_b^x \alpha_1' A_1 + \int_b^x \alpha_0\alpha_0 A_0A_0 + \int_b^x (F\alpha_0+\alpha_0\alpha_1') A_0A_1 \\
& + \int_b^x( \alpha_1'\alpha_0-F\alpha_0)A_1A_0 + \int_b^x \alpha_1'\alpha_1'A_1A_1 \\
= &  1+ \int_b^x\alpha_0A_0 + \int^x_b \alpha_1'A_1+\frac{1}{2}\left( \int_b^x \alpha_0 \right)^2 A_0^2 + \int_b^x \alpha_0 \alpha_1' [A_0,A_1]  \\
& + \int_b^x (\alpha_1'\alpha_0+ \alpha_0\alpha_1')A_1A_0 + \int_b^x F \alpha_0 [A_0,A_1] + \frac{1}{2}\left( \int_b^x \alpha_1' \right)^2 A_1^2 \\
= &  1+ \int_b^x\alpha_0A_0 + \int^x_b \alpha_1'A_1+\frac{1}{2}\left( \int_b^x \alpha_0 \right)^2 A_0^2 + \int_b^x (F\alpha_0+\alpha_0 \alpha_1') [A_0,A_1] \\
 & + \int_b^x \alpha_1'\int_b^x\alpha_0A_1A_0  + \frac{1}{2}\left( \int_b^x \alpha_1' \right)^2 A_1^2 \\
 = & \exp\left( \int_b^x \alpha_1' A_1\right) \exp\left( \int_b^x \alpha_0 A_0 + \int_b^x(F\alpha_0+\alpha_0\alpha_1')[A_0,A_1]\right)
\end{align*}

Recall that we computed the Hodge filtration on $\UU_2$ over $Y$ to be generated by $1,A_1,A_1^2$. Hence we deduce that

\[
j^{dr}_2(x)= \int_b^x \alpha_0 A_0+ \int_b^x(F\alpha_0+\alpha_0\alpha_1')[A_0,A_1]
\]

for $x \in X_v(K_v) \cap Y_v(K_v)$.  \end{proof}

\begin{remark}
To define the map $j^{dr}_2$ at points $x \in X_v(K_v)-Y_v(K_v)$ with tangential basepoint $b$ we proceed as follows: first we need to compute $p^{cr}_2(x)$. Recall that $p^{cr}_n$ is the parallel transport from $b$ to $x$ under $\UU_n$ so we are looking for a horizontal section $s$ which is the analytic continuation from $1 \in b^* \Res_P \UU_n$. Fix some $z \in Y_v(K_v) \cap X_v(K_v)$ and compute $p^{cr}_2(z)$ which is an element of $z^*\U'_2$. Compute $(G_{2}^{-1})_Zp^{cr}_2(z)= s_z$. Given this initial condition for $s$ over $X$, we find that $p^{cr}_2(x)$ is given by 

\[
(G_2^{-1})_zp^{cr}_2(z) + \sum_{|w| \leq 2}\int_z^x \alpha_w w
\]

We then  compute $j^{dr}_2(x)$ as before by expressing $p^{cr}_2(x)=p^H(x)u_2(x)$ and putting $j^{dr}_2(x)=u_2(x)$. 
\end{remark}

We now move onto a more complicated example at level $3$ but the principle here is the same. 

\begin{proposition} \label{Prop - Level 3 TB Map EC}
With $b$ a tangential basepoint at $\PP$ on the elliptic curve $X$, the level $3$ unipotent Albanese map $j^{dr}_3: X_v(K_v) \rightarrow F^0U^{dr}_3 \setminus U^{dr}_3$ on points $x \in X_v(K_v) \cap Y_v(K_v)$ is given by

\[
\begin{split}
x \mapsto & \int_b^x \alpha_0A_0  + \int_b^x (\alpha_0\alpha'_1+ F\alpha_0) [A_0,A_1]+\frac{1}{2}\int_b^x(\alpha_0\alpha_1'\alpha_0+F\alpha_0\cdot\alpha_0-\alpha_0\cdot F\alpha_0)[A_0,[A_1,A_0]]  \\
	& +  \int_b^x (\alpha_0\alpha_1'\alpha_1'+F\alpha_0.\alpha_1'+\alpha_0'- \lambda\alpha_1')				[[A_0,A_1],A_1]
\end{split}
\]

where $\alpha_0' = \frac{1}{2}F^2\alpha_0-\lambda dF$ and $\alpha_1' = \alpha_1-dF$. 

\end{proposition}

\begin{proof}
First we note that with the previously computed gauge transformation $G_3$ as computed in Proposition \ref{Proposition - GT on UC of EC Level 3}:

\[
C_3'=\left(
\begin{array}{cc}
0_4 & D_3' \\
0_{3 \times 4} & C_2' 
\end{array}
\right), \;\; 
D_3' = \left( \begin{array}{ccccccc}
				-\alpha_0 & 0 & 0 & 0 & 0 & 0 & 0 \\
				0 & -\alpha_0 & 0 & 0 & 0 & 0 & 0 \\
				0 & 0 & -\alpha_0 & 0 & -F \alpha_0 & 0 & 0 \\
				0 & 0 & 0 & -\alpha_0 & 0 & -F \alpha_0 & -\alpha_0' \\
				-\alpha_1' & 0 & 0 & 0 & F \alpha_0 & 0 & 0 \\
				0 & -\alpha_1' & 0 & 0 & 0 & F \alpha_0 & 2\alpha_0' \\
				0 & 0 & -\alpha_1' & 0 & 0 & 0 & -\alpha_0' \\
				0 & 0 & 0 & -\alpha_1' & 0 & 0 & 0 \end{array} \right)
\] where $F$ is as in the proof of Proposition \ref{Prop - Level 2 TB Map EC}. Therefore, the parallel transport of $1$ on $b^* \Res_{\PP}\U_3$ to $x^*\U'_3$ for $x \in Y_v(K_v)$ is given by

\begin{align*}
p^{cr}_3(x)= &1+ \sum_{|w| \leq 3} \int_b^x \alpha_w w + \int_b^x F\alpha_0 (A_0A_1 -A_1A_0)  + \int_b^x \alpha_0' (A_0A_1A_1-2A_1A_0A_1 +A_1A_1A_0)\\
& + \int_b^x F\alpha_0\cdot \alpha_0 (A_0A_1A_0-A_1A_0A_0) + \int_b^x \alpha_0\cdot F\alpha_0 (A_0A_0A_1-A_0A_1A_0) \\
& + \int_b^x F\alpha_0\cdot \alpha_1' (A_0A_1A_1-A_1A_0A_1) + \int_b^x \alpha_1'\cdot F\alpha_0 (A_1A_0A_1-A_1A_1A_0).
\end{align*}
We find that
\begin{align*}
&\int_b^x \alpha_0\cdot F\alpha_0 (A_0A_0A_1-A_0A_1A_0) + \int_b^x F\alpha_0\cdot \alpha_0 (A_0A_1A_0-A_1A_0A_0) \\
& = \frac{1}{2}\int_b^x \alpha_0\cdot F\alpha_0 [A_0,[A_0,A_1]] + \frac{1}{2} \int_b^x F\alpha_0\cdot \alpha_0 [[A_0,A_1],A_0]\\
& + \frac{1}{2}\int_b^x F\alpha_0 \int_b^x \alpha_0 ([A_0,A_1]A_0+A_0[A_0,A_1]).
\end{align*}
Observe also that 
\begin{align*}
& \int_b^x F\alpha_0\cdot \alpha_1' (A_0A_1A_1-A_1A_0A_1) + \int_b^x \alpha_1'\cdot F\alpha_0 (A_1A_0A_1-A_1A_1A_0) \\
& =\int_b^x F\alpha_0 \cdot \alpha_1' [[A_0,A_1],A_1] + \int_b^x \alpha_1'\int_b^xF \alpha_0 A_1[A_0,A_1].	
\end{align*}

Putting these formulations together and using the approach that we have previously taken we find that
\[
p^{cr}_3(x) = \exp\left(\int_b^x \alpha_1' (A_1 + \lambda[[A_0,A_1],A_1]\right) \exp(u_3(x))
\]
where $u_3(x)$ is as in the statement of the proposition. To define $j^{dr}_3$ at points not in $Y_v(K_v)$ we proceed as before. 
  \end{proof}

\subsection{The level $2$ period map on hyperelliptic curves}

We again consider a genus $g$ odd hyperelliptic curve $C$ and determine the level $2$ unipotent Albanese map on $X:=C- \lbrace \PP \rbrace$ with tangential basepoint following the approach used above. 

\begin{proposition} \label{Prop - level 2 tb map hec}
With $b$ a tangential basepoint at $\PP$ on an odd affine hyperelliptic curve $X$ of genus $g$ the level $2$ unipotent Albanese map $j^{dr}_2: X_v(K_v) \rightarrow F^0U^{dr}_2 \setminus U^{dr}_2$ on points $x \in X_v(K_v) \cap Y_v(K_v)$ is given by

\[
\begin{split}
j^{dr}_2(x)= & \sum_{k=0}^{g-1} \int_b^x \alpha_k A_k + \frac{1}{2} \sum_{k,l=0}^{g-1} \int_b^x \alpha_k\alpha_l [A_k,A_l] + \sum_{k=0}^{g-1}\sum_{l=g}^{2g-1} \int_b^x (\alpha_k \alpha'_l+c_2^{2gl+k+1,0}) [A_k,A_l] \\
& + \sum_{k=g}^{2g-1} \sum_{\substack{ 0 \leq i < g \\ g \leq j < 2g}} c_{ijk}\int_b^x \alpha'_k[A_j,A_i]
\end{split}
\]

\end{proposition}

\begin{proof}
We will find after application of Algorithm \ref{Algorithm - Computing GT on UC of HEC} that 

\begin{align*}
p^{cr}_2(x) = & 1 + \sum _{k=0}^{g-1} \int^x_b\alpha_k A_k + \sum _{k=g}^{2g-1} \int^x_b\alpha'_k A_k + \sum _{k=0}^{g-1}\sum _{l=g}^{2g-1} \int_b^x \alpha_k\alpha'_l A_kA_l + \sum _{k=g}^{2g-1}\sum _{l=0}^{g-1} \int_b^x \alpha'_k\alpha_l A_kA_l \\
& + \sum _{k=g}^{2g-1}\sum _{l=g}^{2g-1} \int_b^x \alpha'_k\alpha'_l A_kA_l -  \sum _{k=0}^{2g-1}\sum _{l=0}^{2g-1} \int_b^x c_2^{2gk+l+1,0} A_kA_l\\
\end{align*}

where $\alpha_k'=\alpha_k - dh^{k+1,0}_1$ and $c_2^{2gk+l+1,0}$ are $1$-forms on $C$ with at worst logarithmic poles at $\PP$. The computations in the proof of Proposition \ref{Proposition - HF on HEC Level 2} show that we can take $c_2^{2gk+l+1,0}=0$ when $k,l <g$ and otherwise $c_2^{2gk+l+1,0}=c_2^{2gl+k+1,0}$. Now the extra expression appearing in $p^{cr}_2(x)$ 

\[
-  \sum _{k=0}^{2g-1}\sum _{l=0}^{2g-1} \int_b^x c_2^{2gk+l+1,0} A_kA_l
\]
can be rewritten as 
\begin{align*}
& -  \sum_{k=0}^{g-1}\sum_{l=g}^{2g-1} \int_b^x c_2^{2gk+l+1,0} [A_k,A_l]- \frac{1}{2} \sum_{k,l=g}^{2g-1}\int_b^x c_2^{2gk+l+1,0} [A_k,A_l] \\
& = \sum_{k=0}^{g-1}\sum_{l=g}^{2g-1} \int_b^x c_2^{2gl+k+1,0} [A_k,A_l]+ \frac{1}{2} \sum_{k,l=g}^{2g-1} \int_b^x c_2^{2gl+k+1,0} [A_k,A_l].
\end{align*}

We can then conclude by noting that 
\begin{align*}
p^{cr}_2(x) = & \exp\left( \sum_{k=g}^{2g-1} \int_b^x \alpha_k' \left(A_k + \sum_{\substack{0 \leq i < g \\ g \leq j < 2g}} a^{2,1+j+i(2g)}_{1,k+1}[A_i,A_j] \right) + \frac{1}{2}\sum_{l,k=g}^{2g-1} \int_b^x \alpha_k'\alpha_l'[A_k,A_l] \right) \\
& \times \exp \left( \tilde{u}_2(x)\right) \exp \left( \sum_{k=0}^{g-1}\sum_{l=g}^{2g-1} \int_b^x c_2^{2gl+k+1,0} [A_k,A_l] \right) \exp \left( \frac{1}{2} \sum_{k,l=g}^{2g-1} \int_b^x c_2^{2gl+k+1,0} [A_k,A_l] \right) 
\end{align*}
where
\begin{align*}
\tilde{u}_2(x)=& \sum_{k=0}^{g-1} \int_b^x \alpha_k A_k + \frac{1}{2} \sum_{k,l=0}^{g-1} \int_b^x \alpha_k\alpha_l [A_k,A_l] + \sum_{k=0}^{g-1}\sum_{l=g}^{2g-1} \int_b^x \alpha_k \alpha'_l [A_k,A_l] \\
& + \sum_{k=g}^{2g-1} \sum_{\substack{ 0 \leq i < g \\ g \leq j < 2g}} a^{2,1+j+i(2g)}_{1,k+1}\int_b^x \alpha'_k[A_j,A_i].
\end{align*}

Since $\sum_{k=0}^{g-1}\sum_{l=g}^{2g-1} \int_b^x c_2^{2gl+k+1,0} [A_k,A_l]$ and $\frac{1}{2} \sum_{k,l=g}^{2g-1} \int_b^x c_2^{2gl+k+1,0} [A_k,A_l]$ are expressions in words of length 2 the last two exponentials above commute with all others. Since $[A_k,A_l]$ is a generator of $F^0\U_2$ for $k,l \geq g$ then we can conclude that $j^{dr}_2(x)$ is as in the statement of the proposition. 
  \end{proof}
  

\indexprologue{This index lists the notation used throughout the paper for ease of reference.}
\printindex \label{index}



\end{document}